\documentclass[10pt,reqno]{amsart}

\usepackage{amsmath,amsfonts,amsbsy,amsgen,amscd,mathrsfs,amssymb,amsthm}
\usepackage{enumerate,mathtools}

\usepackage[usenames,dvipsnames]{xcolor}
\usepackage[colorlinks=true,citecolor=blue,linkcolor=blue]{hyperref}

\usepackage{lmodern}
\usepackage{mathabx}

\usepackage{tikz}
\usetikzlibrary{shadings}

\newtheorem{thm}{Theorem}[section]
\newtheorem{lem}[thm]{Lemma}

\newtheorem{prop}[thm]{Proposition}
\newtheorem{cor}[thm]{Corollary}

\theoremstyle{definition}

\newtheorem{defn}[thm]{Definition}
\newtheorem{example}[thm]{Example}
\newtheorem{rem}[thm]{Remark}

\newtheorem*{qst}{Question}

\numberwithin{equation}{section} 
\numberwithin{figure}{section}
\numberwithin{table}{section}

\newcommand{\wto}{\stackrel{w}{\to}}

\newcommand{\Tr}{\mathop{\mathrm{Tr}}}

\newcommand{\Vol}{\mathrm{Vol}}
\newcommand{\Dom}{\mathop{\mathrm{Dom}}}

\newcommand{\intr}{\mathop{\mathrm{int}}}

\newcommand{\rank}{\mathop{\mathrm{rank}}}
\newcommand{\spn}{\mathop{\mathrm{span}}}
\newcommand{\supp}{\mathop{\mathrm{supp}}}
\newcommand{\spec}{\mathop{\mathrm{spec}}}
\newcommand{\W}{\mathrm{W}}
\newcommand{\cl}{\mathop{\mathrm{cl}}}

\newcommand{\cof}{\mathop{\mathrm{cof}}}
\newcommand{\D}{\mathsf{D}}
\newcommand{\V}{\mathsf{V}}
\renewcommand{\div}{\mathrm{div}}

\newcommand{\II}{\mathrm{II}}

\begin{document}

\title{The Extremals of Minkowski's Quadratic Inequality}

\author{Yair Shenfeld}
\address{Department of Mathematics, Massachusetts Institute of Technology, 
Cambridge, MA, USA}
\email{shenfeld@mit.edu}

\author{Ramon van Handel}
\address{Fine Hall 207, Princeton University, Princeton, NJ 
08544, USA}
\email{rvan@math.princeton.edu}

\begin{abstract}
In a seminal paper ``Volumen und Oberfl\"ache'' (1903), Minkowski 
introduced the basic notion of mixed volumes and the corresponding
inequalities that lie at the heart of convex geometry.
The fundamental importance of characterizing the extremals of these
inequalities was already emphasized by Minkowski himself, but has to date 
only been resolved in special cases. In this paper, we completely settle 
the extremals of Minkowski's quadratic inequality, confirming a
conjecture of R.\ Schneider. Our proof is based on the representation
of mixed volumes of arbitrary convex bodies as Dirichlet forms associated 
to certain highly degenerate elliptic operators. A key ingredient of the
proof is a quantitative rigidity property associated to these operators.
\end{abstract}

 % Manifolds: Spectral problems; spectral geometry


\maketitle

\thispagestyle{empty}

\section{Introduction}
\label{sec:intro}

\subsection{History of the problem}
\label{sec:history}

The systematic study of the geometry of convex bodies dates back to the 
work of Brunn and Steiner in the 1880s. It is however arguably the work of 
Minkowski that laid the foundation for the modern theory of convex 
geometry. In his seminal paper ``Volumen und Oberfl\"ache'' (1903) 
\cite{Min03}, and in an unfinished manuscript ``Theorie der konvexen 
K\"orper'' that was published posthumously \cite{Min11}, Minkowski 
introduced the basic notion of mixed volumes and the corresponding 
inequalities that play a central role in the modern theory 
\cite{BF87,Sch14}. The aim of this paper is to settle a fundamental 
question arising from Minkowski's original paper that has hitherto 
remained open.

As was customary at that time, Minkowski restricted attention to 
3-dimensional bodies. While our main results are formulated in any 
dimension, let us first explain the problem investigated here in its 
original context. Let $K_1,K_2,K_3$ be convex bodies in $\mathbb{R}^3$ 
(throughout this paper, a \emph{convex body} is a nonempty compact convex 
set). The starting point for Minkowski's theory is the fact that the 
volume of convex bodies is a homogeneous polynomial: for any 
$\lambda_1,\lambda_2,\lambda_3\ge 0$, we have
\begin{equation}
\label{eq:poly3d}
	\Vol(\lambda_1K_1+\lambda_2K_2+\lambda_3K_3) =
	\sum_{i_1,i_2,i_3=1}^3 \V(K_{i_1},K_{i_2},K_{i_3})\,
	\lambda_{i_1}\lambda_{i_2}\lambda_{i_3},
\end{equation}
where we denote $\lambda K+\mu L:=\{\lambda x+\mu y:x\in K,~y\in L\}$. The
coefficients $\V(K,L,M)$ in this polynomial are called \emph{mixed 
volumes}. They are nonnegative, symmetric in their arguments, and linear 
in each argument. Mixed volumes admit various natural geometric 
interpretations, and give rise to many familiar notions as special cases. 
For example, if $K$ is any convex body and $B$ denotes the (Euclidean) 
unit ball, then the volume, surface area, and mean width 
of $K$ may be expressed as
$$
	\Vol(K) = \V(K,K,K), \qquad
	\mathrm{S}(K) = 3\,\V(B,K,K), \qquad
	\W(K) = \frac{3}{2\pi}\V(B,B,K),
$$
respectively. We refer to \cite{Sch14,BZ88,BF87} for a detailed 
exposition.

Once the central role of mixed volumes has been realized, it is 
natural to expect that many geometric properties of convex bodies may be 
expressed in terms of relations between mixed volumes. This perspective 
lies at the heart of Minkowski's theory. In particular, Minkowski established 
\cite[p.\ 479]{Min03} the following fundamental inequality for three 
convex bodies $K,L,M$ in $\mathbb{R}^3$:
\begin{equation}
\label{eq:minkq3d}
	\V(K,L,M)^2 \ge \V(K,K,M)\,\V(L,L,M).
\end{equation}
We refer to \eqref{eq:minkq3d} as (the 3-dimensional case of) 
\emph{Minkowski's quadratic inequality}. This inequality unifies many
classical geometric inequalities for 3-dimensional convex bodies, and 
gives rise to numerous new ones.

\begin{example}
The special cases of \eqref{eq:minkq3d} that involve only
$\Vol$, $\mathrm{S}$, and $\W$ are
$$
	\mathrm{S}(K)^2 \ge 6\pi\,\W(K)\,\Vol(K),\qquad\quad
	\pi\,\W(K)^2 \ge  \mathrm{S}(K).
$$
When combined, these recover
the classical isoperimetric and Urysohn inequalities
$$
	\mathrm{S}(K)^3 \ge 36\pi\,\Vol(K)^2,\qquad\quad
	\W(K)^3 \ge  \frac{6}{\pi}\Vol(K).
$$
\end{example}

\begin{example}
Define $C_t := (1-t)K+tL$ for
$t\in[0,1]$. Then
$\frac{d^2}{dt^2}\Vol(C_t)^{1/3} =
-\frac{2}{(1-t)^{2}}\Vol(C_t)^{-5/3}\{\V(L,C_t,C_t)^2-
\V(L,L,C_t)\V(C_t,C_t,C_t)\}\le 0$ by \eqref{eq:minkq3d}. Thus
$$
	\Vol((1-t)K+tL)^{1/3} \ge
	(1-t)\Vol(K)^{1/3} + t\,\Vol(L)^{1/3},
$$
that is, we recover the Brunn-Minkowski inequality as another consequence
of \eqref{eq:minkq3d}.
\end{example}

These are only some basic examples of the breadth of implications
of \eqref{eq:minkq3d}, cf.\ \cite{BZ88,Sch14};
many others are obtained from more general choices of convex bodies or 
from higher-dimensional analogues that will be discussed below.

Minkowski viewed \eqref{eq:minkq3d} as a far-reaching generalization of 
the classical isoperimetric inequality. To deduce a genuine isoperimetric 
statement, however, an inequality in itself does not suffice: one must 
also understand the associated extremal problem. For example, the 
classical isoperimetric theorem states that among all bodies $K$ with 
fixed volume, surface area is minimized if and only if $K$ is a ball. This 
extremal principle follows from the isoperimetric inequality 
$\mathrm{S}(K)^3 \ge 36\pi\,\Vol(K)^2$ once it is understood that balls 
are the unique \emph{equality cases} or \emph{extremals} of this 
inequality: these are precisely the bodies for which the left-hand side is 
minimized when the right-hand side is fixed. It is far from obvious what 
is the analogous statement for the general inequality \eqref{eq:minkq3d}. 
These considerations motivate the following:

\begin{qst}
For which $K,L,M$ is equality attained in \eqref{eq:minkq3d}?
\end{qst}

Minkowski realized that striking new phenomena arise from this question 
already in the most basic example. Consider the inequality
\begin{equation}
\label{eq:cap}
	\V(B,K,K)^2\ge \V(B,B,K)\,\V(K,K,K),
\end{equation}
which is equivalent to $\mathrm{S}(K)^2 \ge 6\pi\,\W(K)\,\Vol(K)$.
This inequality has the following isoperimetric interpretation:
\emph{among all bodies $K$ with fixed volume and mean width, the surface 
area is minimized if and only if $K$ attains equality in \eqref{eq:cap}.} 
Remarkably, it turns out that this generalized isoperimetric problem 
possesses many unusual extremals, in sharp contrast to the classical 
isoperimetric theorem. For example, equality holds in \eqref{eq:cap} 
whenever $K$ is any cap body of the ball, that is, the convex hull of $B$ 
with a finite or countable number of points so that the ``caps'' emanating 
from these points are disjoint (see Figure \ref{fig:cap}).
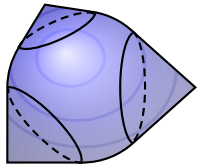
\begin{figure}
\centering
\begin{tikzpicture}

  % tangent lines to z=(x,y) hit circle at angles
  % arccos(x/|z|) +/- arccos(1/|z|).

  % arc segments
  \shade[ball color = blue, opacity = 0.15] (0,-1) arc (270:313.2:1) -- (1.5,0) -- (0.666,0.745) arc (46.8:80.3:1) -- (-.5,1.1) -- (-.853,.521) arc (148.6:180:1) -- (-1,-1) -- (0,-1);

  \draw[thick, densely dashed] (0,-1) [rotate=-225] arc [start angle = 170, end angle = 10,
    x radius = .707, y radius = .2];

  \draw[thick] (-1,0) [rotate=-45] arc [start angle = 170, end angle = 10,
    x radius = .707, y radius = .2];

  \draw[thick] (0.666,-0.745) [rotate=90] arc [start angle = 170, end angle = 10,
    x radius = .745, y radius = .2];

  \draw[thick, densely dashed] (0.666,0.745) [rotate=-90] arc [start angle = 170, end angle = 10,
    x radius = .745, y radius = .2];

  \draw[thick] (.169,.986) [rotate=204.5] arc [start angle = 170, end angle = 10,
    x radius = .561, y radius = .15];

  \draw[thick, densely dashed] (-.853,.521) [rotate=24.5] arc [start angle = 170, end angle = 10,
    x radius = .561, y radius = .15];

  \draw[thick] (0,-1) arc (270:313.2:1) -- (1.5,0) -- (0.666,0.745) arc (46.8:80.3:1) -- (-.5,1.1) -- (-.853,.521) arc (148.6:180:1) -- (-1,-1) -- (0,-1);

\end{tikzpicture}
\caption{A cap body of the ball.\label{fig:cap}}
\end{figure}

These observations motivate the attention paid by Minkowski to the 
extremals of his inequalities. In particular, in \cite[p.\ 477]{Min03}, he 
asserts that the only extremals of \eqref{eq:cap} are cap bodies of the 
ball. No proof of this statement appears, however, in \cite{Min03,Min11}, 
and it seems unlikely that Minkowski had a correct proof of this fact. 
Nonetheless, the statement is correct, as was shown 40 years later by Bol 
\cite{Bol43} using methods that are specific to this special case. On the 
other hand, Minkowski does not formulate any general conjecture on the 
extremals of \eqref{eq:minkq3d}. The characterization of these extremals 
is a long-standing open problem in convex geometry; see, for example, 
\cite[p.\ 248]{Fav33}, \cite[p.\ 80]{Ale96}, \cite[\S 52 and \S 55]{BF87}, 
\cite{Sch85}, \cite[\S 20.5]{BZ88}, \cite[\S 7.6]{Sch14}. The main results
of this paper will provide a complete solution to this problem.

\subsection{Setting and results}

The results of this paper will in fact be developed in a more general 
setting than the 3-dimensional case considered by Minkowski. We will 
presently describe the general setting considered in this paper, as well 
as how the results of this paper fit in a broader context.

The theory of mixed volumes is by no means restricted to three dimensions: 
a full theory was developed by many authors, following key contributions 
by Bonnesen, Fenchel, Favard, and Alexandrov. If $K_1,\ldots,K_m$ are 
convex bodies in $\mathbb{R}^n$, then
\begin{equation}
\label{eq:mixedv}
	\Vol(\lambda_1K_1+\cdots+\lambda_mK_m) =
	\sum_{i_1,\ldots,i_n=1}^m \V(K_{i_1},\ldots,K_{i_n})\,
	\lambda_{i_1}\cdots\lambda_{i_n}
\end{equation}
for $\lambda_1,\ldots,\lambda_m\ge 0$ in direct analogy with the 
$3$-dimensional case \eqref{eq:poly3d}. This identity serves as the 
definition of mixed volumes in dimension $n$.

With this general definition in hand, one can directly adapt the proof of
\eqref{eq:minkq3d} to higher dimension \cite[p.\ 99]{BF87}, which yields the 
inequality
\begin{equation}
\label{eq:minkq}
	\V(K,L,M,\ldots,M)^2 \ge \V(K,K,M,\ldots,M)\,\V(L,L,M,\ldots,M)
\end{equation}
for any convex bodies $K,L,M$ in $\mathbb{R}^n$.
We will henceforth refer to \eqref{eq:minkq} as (the $n$-dimensional case 
of) \emph{Minkowski's quadratic inequality}.
The main result of this paper is a \emph{complete characterization of the 
extremals of \eqref{eq:minkq}}. In particular, the 3-dimensional case of 
our main result fully settles the extremal problem that arises in 
Minkowski's 1903 paper. As the structure of the family of extremals is 
somewhat intricate, we postpone a detailed description of our main results 
to section \ref{sec:main}.

The characterization of the extremals provided by our main results was 
conjectured by Schneider \cite{Sch85} (where credit for a special case is 
given to Loritz) in a more general setting. We now briefly describe the 
broader context of Schneider's conjectures, and how it relates to the 
setting of this paper.

Minkowski's quadratic inequality \eqref{eq:minkq} involves mixed volumes 
of three convex bodies $K,L,M$. In the $3$-dimensional case, all mixed 
volumes are of this form. However, in the $n$-dimensional case, mixed 
volumes may in principle involve $n$ distinct convex bodies.
It was conjectured by Fenchel that Minkowski's quadratic 
inequality can be generalized to arbitrary $n$-dimensional mixed volumes; 
such a general inequality, however, does not follow from Minkowski's 
methods. Its existence was finally proved by 
Alexandrov in 1937 \cite{Ale96}, who showed that
\begin{equation}
\label{eq:af}
        \V(K,L,C_1,\ldots,C_{n-2})^2 \ge
        \V(K,K,C_1,\ldots,C_{n-2}) \,
        \V(L,L,C_1,\ldots,C_{n-2})
\end{equation}
for any convex bodies $K,L,C_1,\ldots,C_{n-2}$ in $\mathbb{R}^n$. The 
fundamental inequality \eqref{eq:af} is known as the 
\emph{Alexandrov-Fenchel inequality}.

The Alexandrov-Fenchel inequality
has numerous applications and connections with various areas of 
mathematics and is a deep result in its own right, cf.\ 
\cite{Ale96,Sch14,BZ88,AGM15,SvH18}. The characterization 
of its equality cases is a well-known open problem that dates back to 
Alexandrov's 1937 paper.\footnote{%
	Alexandrov writes \cite[p.\ 80]{Ale96}: 
	``Serious difficulties occur in determining the conditions
	for equality to hold in the general inequalities just derived.
	This question has not been solved even for the 
	Minkowski quadratic inequality in three dimensional space.''
}
Not even a plausible conjecture was available on the nature of these 
equality cases, until detailed conjectures were proposed by Schneider in 
1985 \cite{Sch85}. To date, only very limited cases of Schneider's 
conjectures have been settled, cf.\ \cite[section 7.6]{Sch14} and the 
references therein. Our main results confirm Schneider's conjectures in 
the setting of Minkowski's quadratic inequality (which already fully 
captures the 3-dimensional case). While a part of our proofs is specific 
to this setting (cf.\ Remark \ref{rem:whyjustmink}), many techniques that 
are introduced in this paper apply to general mixed volumes and may be of 
independent interest. Some key elements of our approach are described in 
section \ref{sec:keyideas}.

Beside their direct significance to the foundations of convex geometry, 
the extremals of Minkowski's quadratic inequality and of the 
Alexandrov-Fenchel inequality are closely connected to several other 
problems. They arise, for example, in the study of infinitesimal rigidity 
of convex surfaces \cite{Fil92,Bla12}, in the mixed analogue of 
Minkowski's uniqueness problem \cite[Theorem 7.4.2]{Sch14}, and in graph 
theory \cite{Izm10}. Moreover, remarkable connections with algebraic 
geometry \cite[section 27]{BZ88} relate extremals of the 
Alexandrov-Fenchel inequality to those of the Hodge index theorem, which 
are not well understood. These connections are orthogonal to the main 
contribution of this paper, so we will not develop them further here.

\subsection{Some key ideas}
\label{sec:keyideas}

In first instance, it may be expected that the characterization of the 
extremals of the Minkowski and Alexandrov-Fenchel inequalities should 
follow from a careful analysis of the proofs of these inequalities. It 
turns out, however, that none of the classical proofs provides information 
on the cases of equality: the proofs rely on strong regularity 
assumptions (such as smooth bodies or polytopes with restricted face 
directions) under which only trivial equality cases arise, and deduce the 
general result by approximation. The study of the nontrivial extremals 
requires one to work directly with general convex bodies, whose analysis 
gives rise to basic open questions in the foundation of convex geometry.

From their definition, it is clear that mixed volumes are well defined for 
arbitrary convex bodies. The behavior of mixed volumes 
$\V(K_1,\ldots,K_n)$ as a function of the bodies $K_1,\ldots,K_n$ is far 
from clear, however. In the early literature on convex geometry 
\cite{Min03,BF87}, explicit formulae for mixed volumes were only available 
for smooth bodies or for polytopes, which do not provide a basis for the 
direct study of general convex bodies.
An important advance in convex geometry, due independently to Fenchel 
and Jessen \cite{FJ38} and Alexandrov \cite{Ale96} in the late 1930s, was 
the explicit formulation of mixed volumes as a functional of a 
\emph{single} body 
$K\mapsto \V(K,C_1,\ldots,C_{n-1})$. These authors discovered that for 
arbitrary convex bodies $C_1,\ldots,C_{n-1}$ in $\mathbb{R}^n$, 
there exists a finite measure $S_{C_1,\ldots,C_{n-1}}$ on $S^{n-1}$ so 
that
$$
	\V(K,C_1,\ldots,C_{n-1}) =
	\frac{1}{n}\int h_K\,dS_{C_1,\ldots,C_{n-1}},
$$
where $h_K:S^{n-1}\to\mathbb{R}$ denotes the support function 
$h_K(x):=\sup_{y\in K}\langle y,x\rangle$. The \emph{mixed area measure} 
$S_{C_1,\ldots,C_{n-1}}$ plays a fundamental role in convex geometry: its 
introduction made it possible to study classical problems that were 
previously restricted to smooth bodies or polytopes (such as the 
Minkowski existence and uniqueness theorems \cite[chapter 8]{Sch14}) in 
the setting of general convex bodies.

The Minkowski and Alexandrov-Fenchel inequalities, however, capture a more 
subtle phenomenon: \eqref{eq:af} expresses the behavior of mixed volumes 
as a functional of \emph{two} bodies $(K,L)\mapsto 
\V(K,L,C_1,\ldots,C_{n-2})$. The structure of this ``quadratic form'', and 
in particular its spectral theory, is central both to proofs of the 
inequality and to the study of its extremals (cf.\ Lemma \ref{lem:hyper}). 
To date, however, such a spectral theory is known to exist only for very 
smooth bodies or polytopes with identical face directions, by virtue of 
the special representations that are available in these cases. The latter 
suffices to prove the Minkowski and Alexandrov-Fenchel inequalities by 
approximation, but provides no information on their nontrivial extremals.

One of the most basic ideas contained in the present paper is the 
introduction of an explicit formulation of mixed volumes as a function of 
two bodies, which gives rise to a spectral theory for mixed volumes of 
arbitrary convex bodies.
What we show (Theorem \ref{thm:forms}) is that for any 
collection of convex bodies $C_1,\ldots,C_{n-2}$ in $\mathbb{R}^n$, 
there is a 
self-adjoint operator $\mathscr{A}_{C_1,\ldots,C_{n-2}}$ on the Hilbert 
space 
$L^2(S^{n-1},S_{B,C_1,\ldots,C_{n-2}})$ so that mixed volumes can be 
expressed as the quadratic form
$$
	\V(K,L,C_1,\ldots,C_{n-2}) = 
	\langle 
	h_K,\mathscr{A}_{C_1,\ldots,C_{n-2}}
	h_L\rangle_{L^2(S^{n-1},S_{B,C_1,\ldots,C_{n-2}})}
$$
for any convex bodies $K,L$ (in particular, $h_K$ is 
guaranteed to lie in the domain of the quadratic form for every convex 
body $K$; see section \ref{sec:presa} for a brief review of the relevant 
notions of functional analysis). The operator
$\mathscr{A}_{C_1,\ldots,C_{n-2}}$ should be viewed as the natural
bilinear counterpart of the mixed area measure 
$S_{C_1,\ldots,C_{n-1}}$, providing a fundamental new tool for the 
study of mixed volumes of general convex bodies in the spirit of the 
Alexandrov-Fenchel-Jessen theory.

It should be emphasized that the introduction of mixed area measures and 
operators does not in itself solve any problem in convexity. For general 
convex bodies, these objects are quite abstract: even the structure of 
mixed area measures is still not fully understood more than 80 years after 
their introduction. However, these objects provide an essential foundation 
for the formulation and investigation of concrete problems involving mixed 
volumes of general convex bodies. In the proof of our main results, we 
will exploit the fact that tractable expressions for the operator 
$\mathscr{A}_{C_1,\ldots,C_{n-2}}$ exist for smooth bodies and for 
polytopes to develop quantitative rigidity estimates that enable us to 
pass to the limit of general convex bodies (while the existence of such 
expressions in the smooth case dates back to Hilbert \cite{Hil12}, they 
appear to be new even for polytopes where they give rise to ``quantum 
graphs'' \cite{BK13}). The structure of our proof, and in particular the 
quantitative estimates that lie at its core, will be described in further
detail in section \ref{sec:overview}.

\subsection{Organization of this paper}

The rest of this paper is organized as follows. The main results of this 
paper are formulated in section \ref{sec:main}. Section \ref{sec:overview} 
is devoted to a high-level overview of the main ingredients of our proofs, 
and of how they fit together. Section \ref{sec:pre} reviews, mostly 
without proofs, some basic results of convex geometry and functional 
analysis that will be used throughout the paper. Section \ref{sec:forms} 
is devoted to the construction and basic properties of the self-adjoint 
operators that form the foundation for our analysis. Sections 
\ref{sec:weak} and \ref{sec:rigid} develop the two main ingredients of the 
proof of our main result: a weak stability theorem and a quantitative 
rigidity theorem for Minkowski's inequality. In particular, the proof of 
our main results will be completed in section \ref{sec:rigid} in the case 
that $M$ in \eqref{eq:minkq} is full-dimensional. Finally, section 
\ref{sec:lower} is devoted to the case that $M$ in \eqref{eq:minkq} is 
lower-dimensional.

\section{The extremals of Minkowski's quadratic inequality}
\label{sec:main}

\subsection{Main results}

Before we can formulate our main results, we must recall some basic
notions of convex geometry. Here and throughout the paper, our standard
reference on convexity will be the monograph \cite{Sch14}.

Throughout the remainder of this paper we fix the dimension $n\ge 3$. A 
convex body is a nonempty compact convex subset of $\mathbb{R}^n$. Given a 
convex body $K$ and a vector $u\in\mathbb{R}^n\backslash\{0\}$, we denote 
by $F(K,u)$ the unique exposed face of $K$ with outer normal vector $u$. 
The supporting hyperplane of $K$ in the normal direction $u$ is the 
hyperplane $F(K,u)+u^\perp$. A key role in the present context will be 
played by normal directions that are extreme in the following sense 
\cite[p.\ 85]{Sch14}.

\begin{defn}
A vector $u\in\mathbb{R}^n\backslash\{0\}$ is called an \emph{$r$-extreme 
normal vector} of a convex body $K$ if there do not exist linearly 
independent normal vectors $u_1,\ldots,u_{r+2}$ at a
boundary point of $K$ such that $u=u_1+\cdots+u_{r+2}$.
\end{defn}

For example, if $K$ is a polytope, then $u$ is an $r$-extreme normal 
vector if and only if it is an outer normal of a face of $K$ of dimension
$\dim F(K,u)\ge n-1-r$.

We are now ready to describe the extremals of Minkowski's quadratic 
inequality \eqref{eq:minkq}. We must distinguish several cases, depending 
on the dimension of $M$. Let us observe at the outset that 
\eqref{eq:minkq} is invariant under translation and scaling of each  
body, so that the extremals must be invariant under homothety as well.

We begin by considering the main case where $M$ has nonempty 
interior. The following result confirms a conjecture of Schneider 
\cite{Sch85}, cf.\ \cite[section 7.6]{Sch14}.

\begin{thm}[\textbf{Extremals: full-dimensional case}]
\label{thm:mainfull}
Let $M\subset\mathbb{R}^n$ be a convex body with nonempty interior,
and let $K,L\subset\mathbb{R}^n$ be arbitrary convex bodies such that
$\V(L,L,M,\ldots,M)>0$. Then we have
$$
	\V(K,L,M,\ldots,M)^2 = \V(K,K,M,\ldots,M)\,\V(L,L,M,\ldots,M)
$$
if and only if there exist $a\ge 0$ and
$v\in\mathbb{R}^n$ such that $K$ and $aL+v$ have the same 
supporting hyperplanes in all $1$-extreme normal directions of $M$.
\end{thm}

When $M$ is lower-dimensional, the conclusion of Theorem 
\ref{thm:mainfull} is no longer valid and additional extremals appear. A 
suitable modification of Schneider's conjecture in this case was proposed 
by Ewald and Tondorf \cite{ET94}, see also \cite[section 4.2]{Sch94}. The 
following result confirms this conjecture in the present setting.

\begin{thm}[\textbf{Extremals: lower-dimensional case}]
\label{thm:mainlower}
Let $M\subset\mathbb{R}^n$ be a convex body with empty interior, so that
$M-M\subset w^\perp$ for some $w\in S^{n-1}$.
Let $K,L\subset\mathbb{R}^n$ be arbitrary convex bodies such that
$\V(L,L,M,\ldots,M)>0$. Then we have
$$
	\V(K,L,M,\ldots,M)^2 = \V(K,K,M,\ldots,M)\,\V(L,L,M,\ldots,M)
$$
if and only if $\tilde L:=\frac{\V(K,L,M,\ldots,M)}{\V(L,L,M,\ldots,M)}L$
satisfies that
$K+F(\tilde L,w)$ and $\tilde L+F(K,w)$ have the same supporting hyperplanes in 
all $1$-extreme normal directions of $M$.
\end{thm}

The only case that remains to be considered is  
$\V(L,L,M,\ldots,M)=0$, in which case equality in \eqref{eq:minkq} can 
only arise for the trivial reason $\V(K,L,M,\ldots,M)=0$. The 
characterization of positivity of mixed volumes is well-known and 
is unrelated to the questions in this paper. For completeness, we 
spell out the resulting characterization without further comment 
\cite[Theorem 5.1.8]{Sch14}.

\begin{thm}[\textbf{Extremals: trivial case}]
\label{thm:maintriv}
Let $K,L,M\subset\mathbb{R}^n$ be convex bodies such that
$\V(L,L,M,\ldots,M)=0$. Then we have
$$
        \V(K,L,M,\ldots,M)^2 = \V(K,K,M,\ldots,M)\,\V(L,L,M,\ldots,M)
$$
if and only if
one of the following holds: $\dim K=0$; $\dim L=0$; 
$\dim(K+L)\le 1$; $\dim M\le n-3$; $\dim(K+M)\le n-2$; $\dim(L+M)\le n-2$; 
$\dim(K+L+M)\le n-1$.
\end{thm}

We also note for completeness that $\V(L,L,M,\ldots,M)=0$ holds if and 
only if $\dim L\le 1$, $\dim M\le n-3$, or $\dim(L+M)\le n-1$. Thus the 
different cases covered by Theorems \ref{thm:mainfull}--\ref{thm:maintriv} 
are completely determined by the dimensions of $L,M,L+M$.

\subsection{Examples and prior work}
\label{sec:prior}

In order to understand the statement of Theorem \ref{thm:mainfull}, it is 
instructive to revisit the case of cap bodies. It is readily seen by 
inspection of Figure \ref{fig:cap} that the normal directions of a 
3-dimensional cap body $K$ that are \emph{not} $1$-extreme are exactly 
those that lie in the interior of the normal cone of the vertex of one of 
its caps. Therefore, any supporting hyperplane of $K$ whose normal 
direction is $1$-extreme also supports the ball $B$ from which the cap 
body was formed. Thus Theorem \ref{thm:mainfull} explains precisely why 
cap bodies arise as the extremals of the special case \eqref{eq:cap} of 
Minkowski's quadratic inequality.

The extremals of \eqref{eq:cap}, and more generally of the case $L=M$ of 
Theorem \ref{thm:mainfull}, were obtained by Bol in 1943 \cite{Bol43}. 
Prior to the present paper, this was the only case of \eqref{eq:minkq} 
whose extremals have been characterized for general convex bodies. Despite 
its importance, Bol's proof turns out to be fundamentally based on the 
following very special feature of the case $L=M$. Note that if equality is 
attained in \eqref{eq:minkq}, we may always assume (by rescaling $L$) that
\begin{equation}
\label{eq:bol}
	\V(K,K,M,\ldots,M) = \V(K,L,M,\ldots,M) = 
	\V(L,L,M,\ldots,M)
\end{equation}
as \eqref{eq:minkq} is invariant under scaling. By means of a delicate 
geometric argument using the method of inner parallel bodies, Bol was able 
to show that if \eqref{eq:bol} holds with $L=M$, then it must be the case 
that $K\subseteq L$ (up to translation of $L$). This reduces the extremal 
problem of Minkowski's quadratic inequality to equality 
cases of the monotonicity of mixed volumes, which were already 
characterized in the case $L=M$ by Minkowski himself \cite[p.\ 
227]{Min11} (cf.\ \cite[Theorem 7.6.17]{Sch14}). Unfortunately, the entire
premise of this argument fails for general $K,L,M$.

\begin{example}
Let $M=\mathrm{conv}\{B,x,-x\}$, where $B$ is a ball in $\mathbb{R}^3$ and 
$x\not\in B$. Let $L$ be any body that has the same supporting hyperplanes 
as $B$ in all normal directions except those in the interior of the
normal cone of $M$ at $x$, and 
let $K=-L$ (cf.\ Figure \ref{fig:noncap}). Then the supporting 
hyperplanes of $K$ and $L$ coincide in all $1$-extreme normal directions 
of $M$, so by Theorem \ref{thm:mainfull} we have equality in 
\eqref{eq:minkq}. Moreover, by symmetry
$\V(K,K,M,\ldots,M)=\V(L,L,M,\ldots,M)$, so \eqref{eq:bol} holds. But 
clearly $K+v\not\subseteq L$ and $L\not\subseteq K+v$ for any 
$v$. Thus \eqref{eq:bol} need not imply any monotonicity.
\begin{figure}
\centering
\begin{tikzpicture}[scale=.8]

\begin{scope}[xshift=-10cm]

  \draw (-1.3,0) node[left] {$K=$};

  \shade[ball color = blue, opacity = 0.15] (-1.2,0.268) -- (-0.666,0.745) arc (133.2:-133.2:1) -- (-1.2,-0.268) [rotate=90] arc [start angle = 170, end angle = 10, x radius = .268, y radius = .07];

  \draw[thick] (-0.64,-0.73) [rotate=90] arc [start angle = 170, end angle = 10,
    x radius = .745, y radius = .2];

  \draw[thick, densely dashed] (-0.66,0.745) [rotate=-90] arc [start angle = 170, end angle = 10,
    x radius = .745, y radius = .2];

  \draw[thick, densely dashed, dash phase=.5mm] (-1.2,-0.268) [rotate=-90] arc [end angle = 170, start angle = 10,
    x radius = .268, y radius = .07];

  \draw[thick] (-1.2,0.268) -- (-0.666,0.745) arc (133.2:-133.2:1) -- (-1.2,-0.268) [rotate=90] arc [start angle = 170, end angle = 10, x radius = .268, y radius = .07];

\end{scope}

\begin{scope}[xshift=-5.5cm]

  \draw (-1,0) node[left] {$L=$};

  \shade[ball color = blue, opacity = 0.15] (1.2,0.268) -- (0.666,0.745) arc (46.8:313.2:1) -- (1.2,-0.268) [rotate=-90] arc [end angle = 170, start angle = 10, x radius = .268, y radius = .07];

  \draw[thick] (0.66,-0.73) [rotate=90] arc [start angle = 170, end angle = 10,
    x radius = .745, y radius = .2];

  \draw[thick, densely dashed] (0.65,0.75) [rotate=-90] arc [start angle = 170, end angle = 10,
    x radius = .745, y radius = .2];

  \draw[thick] (1.2,-0.268) [rotate=90] arc [start angle = 170, end angle = 10,
    x radius = .268, y radius = .07];

  \draw[thick] (1.2,0.268) -- (0.666,0.745) arc (46.8:313.2:1) -- (1.2,-0.268) {[rotate=-90] arc [end angle = 170, start angle = 10, x radius = .268, y radius = .07]};

\end{scope}

\begin{scope}[xshift=0]

  \draw (-1.5,0) node[left] {$M=$};

  \shade[ball color = blue, opacity = 0.15] (1.5,0) -- (0.666,0.745) arc (46.8:133.2:1) -- (-1.5,0) -- (-0.666,-0.745) arc (226.8:313.2:1) -- (1.5,0);

  \draw[thick] (0.666,-0.75) [rotate=90] arc [start angle = 170, end angle = 10,
    x radius = .745, y radius = .2];

  \draw[thick, densely dashed] (0.666,0.745) [rotate=-90] arc [start angle = 170, end angle = 10,
    x radius = .745, y radius = .2];

  \draw[thick] (-0.666,-0.73) [rotate=90] arc [start angle = 170, end angle = 10,
    x radius = .745, y radius = .2];

  \draw[thick, densely dashed] (-0.666,0.745) [rotate=-90] arc [start angle = 170, end angle = 10,
    x radius = .745, y radius = .2];

  \draw[thick] (1.5,0) -- (0.666,0.745) arc (46.8:133.2:1) -- (-1.5,0) -- (-0.666,-0.745) arc (226.8:313.2:1) -- (1.5,0);

\end{scope}
\end{tikzpicture}
\caption{A non-monotone equality case.\label{fig:noncap}}
\end{figure}
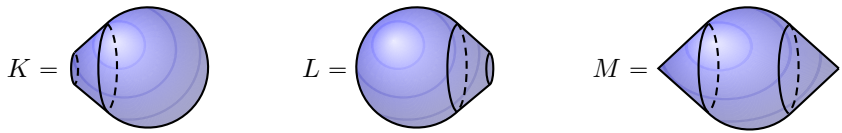
\end{example}

While the method of inner parallel bodies that is introduced by Bol
has found many other applications in convex geometry, the above 
example illustrates why Bol's methods cannot explain the more general 
extremals of \eqref{eq:minkq}.

Other than Bol's theorem, special cases of the extremal problem have been 
proved under various regularity assumptions on the bodies involved. If $M$ 
is a smooth convex body (that is, each boundary point has a unique 
normal), then all normal directions are $1$-extreme, so that equality can 
arise only in the trivial case that $K$ and $L$ are homothetic. The lack 
of nontrivial extremals in this case was proved by Hilbert \cite{Hil12} 
when $M$ has a sufficiently smooth boundary (see section \ref{sec:hilbert}),
and by Schneider \cite{Sch90} 
for general smooth bodies. Similarly, no nontrivial equality cases can 
arise when $K,L,M$ are all strongly isomorphic simple polytopes; this case 
follows from Alexandrov's proof of the Alexandrov-Fenchel inequality 
\cite{Ale96}. Another special case, where $M$ is a zonoid and $K,L$ are 
symmetric bodies, was proved by Schneider in \cite{Sch88}. The case that 
$M$ is lower-dimensional was investigated by Ewald and Tondorf \cite{ET94} 
when $K,L,M$ are polytopes; a simple example that illustrates Theorem 
\ref{thm:mainlower} may be found in their paper (see also \cite[p.\ 
429]{Sch14}).

Most importantly for our purposes, the validity of Theorems 
\ref{thm:mainfull} and \ref{thm:mainlower} was previously established by 
Schneider in the case that $M$ is a simple polytope and $K,L$ are 
arbitrary convex bodies; cf.\ \cite[Theorem 7.6.21]{Sch14} and 
\cite[Theorem 4.2]{Sch94} (in the lower-dimensional case, $M$ need not be 
simple). As will be explained below, our approach to Theorems 
\ref{thm:mainfull} and \ref{thm:mainlower} was partially inspired by 
Schneider's results for polytopes. The key contribution of this paper, 
however, is the introduction of new tools that open the door to the 
investigation of arbitrary convex bodies.

\section{Overview}
\label{sec:overview}

The formulation of our main results is of a purely geometric nature. 
Nonetheless, it will shortly become clear that the core difficulty in our 
proofs does not arise from geometry, but rather from analytic 
questions: at the heart of our results lies an analysis of the behavior of 
certain highly degenerate elliptic operators. The aim of this section is 
to give a high-level overview of the main ingredients of our proofs and 
how they fit together, in order to help the reader navigate the rest of 
the paper. We have deliberately kept this overview as concise as possible, 
as further discussion is best postponed until after formal definitions 
have been given.

\subsection{The Hilbert method}
\label{sec:hilbert}

The original proofs derive Minkowski's quadratic inequality as a limiting 
case of other inequalities, such as the Brunn-Minkowski inequality 
\cite[sections 49 and 52]{BF87}. These proofs do not lend themselves to 
the study of extremals, however, as extremals are not preserved by taking 
limits. Instead, our starting point will be the direct approach to 
Minkowski's inequality due to Hilbert \cite[Chapter XIX]{Hil12} (cf.\ 
\cite[section 52]{BF87}), which also forms the foundation for the 
Alexandrov-Fenchel inequality. Let us briefly explain its basic
premise.

It will be convenient to identify a convex body $K$ 
with its support function
$$
	h_K(u) := \sup_{y\in K}\langle y,u\rangle.
$$
Geometrically, if $u\in S^{n-1}$, then $h_K(u)$ is the signed distance to 
the origin of the supporting hyperplane of $K$ in the direction $u$; 
as any convex body is the intersection of its supporting halfspaces,
$h_K:S^{n-1}\to\mathbb{R}$ uniquely determines $K$.
Support functions satisfy $h_{\lambda K+\mu L}=
\lambda h_K+\mu h_L$ for $\lambda,\mu\ge 0$, i.e., they map linear 
operations on sets to linear operations on functions. In particular, we 
may view
$$
	(h_K,h_L) \mapsto \V(K,L,M,\ldots,M)
$$
as a symmetric quadratic form of the support functions. By 
linearity of mixed volumes, this quadratic form can be uniquely extended 
to the linear space of differences of support functions (which contains 
$C^2(S^{n-1})$, cf.\ Lemma \ref{lem:diffsf} below). Minkowski's 
inequality \eqref{eq:minkq} is nothing other than the statement that this 
quadratic form satisfies a reverse form of the Cauchy-Schwarz inequality.

We are therefore led to ask which quadratic forms satisfy reverse 
Cauchy-Schwarz inequalities. A particularly useful characterization arises 
if we consider closed symmetric quadratic forms on a Hilbert space, i.e., 
forms $\mathcal{E}(f,g)=\langle f,\mathscr{A}g\rangle$ associated to a 
self-adjoint operator $\mathscr{A}$. The reason this setting is powerful 
is that we can bring spectral theory to bear on the problem: $\mathcal{E}$ 
satisfies a reverse Cauchy-Schwarz inequality if and only if $\mathscr{A}$ 
has a one-dimensional positive eigenspace. Let us formulate for future 
reference a general statement for (possibly unbounded) 
self-adjoint operators on a Hilbert space, whose proof is given in 
section \ref{sec:presa}.\footnote{%
Here and in the sequel, we use the following notation for a self-adjoint 
operator $\mathscr{A}$ and the associated closed quadratic form 
$\mathcal{E}$: $\spec\mathscr{A}$ denotes the spectrum, $\ker\mathscr{A}$ 
the kernel, and $\rank\mathscr{A}$ the rank (i.e., the dimension of the 
range) of $\mathscr{A}$; and $\Dom\mathscr{A}$ and $\Dom\mathcal{E}$ 
denote the domains of $\mathscr{A}$ and $\mathcal{E}$. A brief review of 
the relevant notions of functional analysis may be found in section 
\ref{sec:presa}.}

\begin{lem}[\textbf{Hyperbolic quadratic forms}]
\label{lem:hyper}
Let $\mathscr{A}$ be a self-adjoint operator on a Hilbert 
space $H$ with $0<\sup\spec\mathscr{A}<\infty$, and let $\mathcal{E}(f,g)$ 
be the associated closed quadratic form. Then the following are equivalent:
\begin{enumerate}[$1$.]
\item $\mathcal{E}(f,g)^2\ge\mathcal{E}(f,f)\,\mathcal{E}(g,g)$
for all $f,g\in\Dom\mathcal{E}$ such that $\mathcal{E}(g,g)>0$.
\item $\rank \mathrm{1}_{(0,\infty)}(\mathscr{A})=1$.
\end{enumerate}
Moreover, if either \upn{(}hence both\upn{)} of these conditions is 
satisfied, the following are equivalent for given $f,g\in\Dom\mathcal{E}$ 
such that $\mathcal{E}(g,g)>0$:
\begin{enumerate}[$1'$.]
\item $\mathcal{E}(f,g)^2=\mathcal{E}(f,f)\,\mathcal{E}(g,g)$.
\item $f-ag\in\ker\mathscr{A}\subset\Dom\mathscr{A}$ for some 
$a\in\mathbb{R}$.
\end{enumerate}
\end{lem}

It is not clear, a priori, that mixed volumes fit into the setting of
Lemma \ref{lem:hyper}. However, Hilbert realized that \emph{for 
sufficiently smooth bodies}, a classical representation formula of 
Minkowski (see, e.g., Lemma \ref{lem:repc2}) can be used to write
$$
	\V(K,L,M,\ldots,M) =
	\langle h_K,\mathscr{A}h_L\rangle_{L^2(\omega)},
$$
where $\omega$ denotes the surface measure on $S^{n-1}$ and $\mathscr{A}$ 
is a certain elliptic second order differential operator. Elliptic 
regularity theory (cf.\ section \ref{sec:elliptic})
shows that $\mathscr{A}$ has a self-adjoint extension on 
$L^2(\omega)$ with compact resolvent. Hilbert exploited this 
regularity, using a clever homotopy argument, to prove 
condition $2$ of Lemma~\ref{lem:hyper}. This yields Minkowski's inequality 
for smooth bodies $M$, and the general case follows by approximation. As a 
key step in his proof, Hilbert shows that
$$
	\ker\mathscr{A} = \{x\mapsto\langle v,x\rangle:v\in\mathbb{R}^n\}
$$
consists precisely of the linear functions restricted to $S^{n-1}$. 
Thus for \emph{smooth} bodies $M$, Lemma \ref{lem:hyper} implies 
that equality holds in Minkowski's inequality if and only if
$h_K-ah_L=\langle 
v,\cdot\,\rangle$ for some $a,v$, i.e., when
$K,L$ are homothetic $K=aL+v$.

The above discussion shows that, while innocent for the purpose of proving 
Minkowski's inequality, the smoothness assumption on $M$ completely 
eliminates its nontrivial extremals. The key difficulty we face in the 
analysis of extremals is to make sense of the above ideas for arbitrary 
non-smooth bodies $M$.

\subsection{Mixed volumes and Dirichlet forms}

It is far from clear that it is possible, even in principle, to define 
self-adjoint representations of mixed volumes for non-smooth bodies $M$. 
While we have argued above that $(h_K,h_L)\mapsto \V(K,L,M,\ldots,M)$ 
defines a symmetric quadratic form on $C^2(S^{n-1})$, it is a basic 
fact of functional analysis that not every such form is associated to a 
self-adjoint operator. A self-adjoint representation is crucial, however, 
for the applicability of spectral theory in the infinite-dimensional 
setting. The basic criterion for a quadratic form to admit a self-adjoint 
representation is that it is \emph{closable}, see section \ref{sec:presa} 
below. This property may not be taken for granted. For example, one may 
verify that when $M$ is a polytope, the quadratic form 
$(h_K,h_L)\mapsto\V(K,L,M,\ldots,M)$ is not closable in $L^2(\omega)$ and 
therefore does not admit any self-adjoint representation on this space. 
For this reason, self-adjoint representations have only appeared in the 
literature in special cases: for smooth bodies (due to Hilbert 
\cite{Hil12}), and for certain special families of polytopes with common 
face normals (due to Alexandrov \cite{Ale96}).

The first step towards our analysis of extremals is the realization that 
mixed volumes do in fact admit a self-adjoint representation for 
\emph{any} body $M$, provided one chooses the Hilbert space appropriately. 
Let us recall that the mixed area measure $S_{K_1,\ldots,K_{n-1}}$ 
on $S^{n-1}$ was introduced in section \ref{sec:keyideas}.

\begin{thm}
\label{thm:formsminkq}
For any convex body $M$, there is a self-adjoint operator
$\mathscr{A}$ on $L^2(S_{B,M,\ldots,M})$ such that the associated
closed quadratic form $\mathcal{E}(f,g)$ satisfies 
$$
	h_K,h_L\in\Dom\mathcal{E}\quad\mbox{and}\quad
	\V(K,L,M,\ldots,M) = \mathcal{E}(h_K,h_L)
$$
for all convex bodies $K,L$.
\end{thm}

The operator $\mathscr{A}$ of Theorem \ref{thm:formsminkq} is a highly 
degenerate elliptic operator in the sense of the theory of Dirichlet forms 
\cite{FOT11,BGL14}. Its existence and basic properties will be 
investigated in section \ref{sec:forms} (in the general setting of 
arbitrary mixed volumes). While this operator is, in general, a rather 
abstract object, it provides the foundation for extending Hilbert's method 
to arbitrary convex bodies $M$.

Theorem \ref{thm:formsminkq} and Lemma \ref{lem:hyper} reduce the study of 
extremals of Minkowski's inequality to the characterization of 
$\ker\mathscr{A}$. Translation-invariance of mixed volumes  
implies that $\ker\mathscr{A}$ always contains the linear functions, so 
that homothetic bodies $K,L$ trivially yield equality. It should be 
emphasized at this point that there are two distinct mechanisms for the 
appearance of nontrivial extremals:
\begin{enumerate}[a.]
\itemsep\abovedisplayskip
\item It is possible that $\ker\mathscr{A}$ may be strictly larger
than the set of linear functions. Its additional (nonlinear) elements 
can give rise to new extremals.
\item Even if $\ker\mathscr{A}$ contains only linear functions, new
extremals can arise as the underlying measure
$S_{B,M,\ldots,M}$ need not be supported on the entire sphere. Thus
Lemma \ref{lem:hyper} only guarantees that $h_K-ah_L=\langle 
v,\cdot\,\rangle$ $S_{B,M,\ldots,M}$-a.e.
\end{enumerate}
The statements of Theorems \ref{thm:mainfull} and \ref{thm:mainlower} can
now be interpreted in a new light. A fundamental result of Schneider
(Theorem \ref{thm:supp} below) states that
$$
	\supp S_{B,M,\ldots,M} = \cl\{u\in S^{n-1}:u\mbox{ is a }
	1\mbox{-extreme normal vector of }M\}.
$$
Thus Theorem \ref{thm:mainfull} shows that when $M$ has nonempty 
interior, only mechanism b.\ arises.\footnote{
	More precisely, our results show that $\ker\mathscr{A}\cap
	\{h_Q-h_R:Q,R\mbox{ convex bodies}\}$ consists of linear functions 
	only, which is equivalent to Theorem \ref{thm:mainfull} by
	Lemma \ref{lem:hyper} (see also Remark \ref{rem:fdj}).
}
In contrast, 
both a.\ and b.\ arise when the body $M$ has empty 
interior, which explains the appearance of additional extremals in
Theorem \ref{thm:mainlower}. 

The proofs of these facts occupy the main part of this paper. The proof of 
Theorem \ref{thm:mainlower} will turn out to be conceptually simpler, as 
in this case one can compute explicitly the operator $\mathscr{A}$ and its 
kernel. This will be done in section \ref{sec:lower}. The main difficulty 
lies in the proof of Theorem \ref{thm:mainfull}, as in this case we do not 
have an explicit description of the operator $\mathscr{A}$ that is 
amenable to computation.

\begin{rem}
That the extremals of \eqref{eq:minkq} are determined by the support of 
the measure $S_{B,M,\ldots,M}$ was conjectured by Schneider \cite{Sch85} 
(in the general setting of the Alexandrov-Fenchel inequality). Even before we 
explain how the extremal characterization is proved, however, we observe 
that the abstract statement of Theorem \ref{thm:formsminkq} already 
demystifies why this particular measure appears: this is precisely the 
measure that must be chosen on the sphere so that the quadratic form 
$(h_K,h_L)\mapsto \V(K,L,M,\ldots,M)$ is closable. Thus the existence of a 
self-adjoint representation is far from being merely a technical device: 
it explains the fundamental structure of the extremals of Minkowski's 
quadratic inequality.
\end{rem}

\subsection{Weak stability and rigidity}
\label{sec:overviewmain}

The most natural starting point for understanding the extremals of a given 
inequality is to attempt to deduce these from a careful examination of the 
proof of the inequality. In the case of Minkowski's inequality, however, 
there is a surprising and apparently fundamental obstacle to such an 
approach. In view of Lemma \ref{lem:hyper} and Theorem 
\ref{thm:formsminkq}, it is clear that $S_{B,M,\ldots,M}$ is the relevant 
measure for the study of the \emph{extremals}. However, it is a different 
measure $S_{M,\ldots,M}$ that appears naturally (explicitly or implicitly) 
in the various known proofs of Minkowski's \emph{inequality} 
\cite{Ale96,KM17,SvH18}.\footnote{%
	We emphasize that the quadratic form 
	$(h_K,h_L)\mapsto\V(K,L,M,\ldots,M)$ is not 
	closable on $L^2(S_{M,\ldots,M})$ unless one assumes strong 
	regularity assumptions on $M$ (such as smoothness), which
	explains why regularity assumptions play a fundamental role in 
	proofs of the inequality.
}
As (cf.\ Theorem \ref{thm:supp})
$$
        \supp S_{M,\ldots,M} = \cl\{u\in S^{n-1}:u\mbox{ is a }
        0\mbox{-extreme normal vector of }M\},
$$
the support of $S_{M,\ldots,M}$ is generally much smaller than that of
$S_{B,M,\ldots,M}$, and does not suffice to characterize the 
extremals (see Example \ref{ex:weak} below). Thus there appears to be a 
disconnect between the setting of this paper that gives rise 
to a good analytic theory, and the setting used in the proofs of 
Minkowski's inequality that is needed to exploit the algebraic 
structure of mixed volumes.

The conceptual challenge behind the proof of Theorem \ref{thm:mainfull} is 
to understand how to surmount this obstacle. To this end, we will prove 
that the operator $\mathscr{A}$ satisfies a strong rigidity property: 
once the extremals of Minkowski's inequality have been fixed in the 
$0$-extreme directions of $M$, their extension to the $1$-extreme 
directions of $M$ is uniquely determined. This property closes the gap 
between the information provided by proofs of Minkowski's inequality and 
its extremals.

More precisely, we will proceed in two steps. First, in section 
\ref{sec:weak}, we obtain weak control on the extremals by refining an 
approach to Minkowski's inequality that was recently developed by 
Kolesnikov and Milman \cite{KM18,KM17}.

\begin{thm}
\label{thm:weakintro}
Let $K,L,M$ be convex bodies so that $M$ has nonempty interior and
$\V(L,L,M,\ldots,M)>0$. If equality holds in \eqref{eq:minkq}, 
then there exist $a\ge 0$, $v\in\mathbb{R}^n$ so that 
$h_K(x)-ah_L(x)=\langle v,x\rangle$ for all $x\in\supp S_{M,\ldots,M}$.
\end{thm}

We emphasize that Theorem \ref{thm:weakintro}, on its own, provides only 
very weak information on the extremals, as is illustrated by the following 
example.

\begin{example}
\label{ex:weak}
When $M$ is a polytope, the $0$-extreme directions are precisely the facet 
normals. Thus, for example, if equality holds in \eqref{eq:minkq} for 
convex bodies $K$ and $L=M$ with nonempty interior, then Theorem 
\ref{thm:weakintro} only implies that $K$ touches every facet of a 
homothet of $M$. While this is nontrivial geometric information, this 
property is much too weak to characterize the extremals. For example, if 
$K$ is the unit ball and $L=M$ is the unit cube, then $K$ touches every 
facet of $M$, but this is not an equality case of Minkowski's inequality.
\end{example}

In the second step, we amplify the weak control provided by Theorem 
\ref{thm:weakintro} to fully characterize the extremals. The following 
rigidity property, which will be proved in section \ref{sec:rigid}, is one 
of the central results of this paper.

\begin{thm}
\label{thm:rigidintro}
Let $K,L,M$ be convex bodies so that $M$ has nonempty interior and
$\V(L,L,M,\ldots,M)>0$. If equality holds in \eqref{eq:minkq} and 
$h_K(x)=h_L(x)$ for all $x\in\supp S_{M,\ldots,M}$, then necessarily
$h_K(x)=h_L(x)$ for all $x\in\supp S_{B,M,\ldots,M}$.
\end{thm}

Unfortunately, we cannot prove Theorems \ref{thm:weakintro} and 
\ref{thm:rigidintro} directly, as we do not have a sufficiently explicit 
description of $\mathscr{A}$ for general convex bodies. Instead, we will 
prove quantitative versions of both these theorems for special 
convex bodies. A quantitative form of Theorem \ref{thm:weakintro}, i.e., a 
weak stability form of Minkowski's inequality, will be proved for smooth 
bodies using Reilly's formula in Riemannian geometry. A quantitative 
form of Theorem \ref{thm:rigidintro}, i.e., a quantitative rigidity 
theorem, will be proved for polytopes: in this case $\mathscr{A}$ is a 
``quantum graph'' (cf.\ section \ref{sec:qgraph}), and the proof will
exploit a stability estimate for the solution of the Dirichlet problem 
for $\mathscr{A}$ with boundary data on $\supp S_{M,\ldots,M}$. In both 
cases, a key aspect of the proof is to discover the correct quantitative 
formulation that does not degenerate when we take the appropriate limit to 
approximate arbitrary $M$.

\begin{rem}
The formulation of Theorems \ref{thm:weakintro} and \ref{thm:rigidintro} 
was inspired by the proof of a result of Schneider \cite[Theorem 
7.6.21]{Sch14} for the case that $M$ is a simple polytope. In this 
setting, the statement of Theorem \ref{thm:weakintro} can be deduced from 
Alexandrov's polytope proof of the Alexandrov-Fenchel inequality 
\cite{Ale96}, while the statement of Theorem \ref{thm:rigidintro} follows 
from a form of Minkowski's uniqueness theorem. However, these qualitative 
statements do not allow one to pass to the limit of general convex 
bodies; the tools that are needed to do so are developed in this paper.

On the other hand, Schneider's result goes beyond Minkowski's 
inequality to cover some additional cases of the Alexandrov-Fenchel 
inequality. Similarly, most of the techniques that are developed in this 
paper are not specific to Minkowski's inequality, and extend to 
general mixed volumes. The only ingredient of this paper that is 
fundamentally restricted to Minkowski's classical setting is Theorem 
\ref{thm:weakintro}, for reasons that will become clear in section 
\ref{sec:weak}. While a replacement for this argument in the 
Alexandrov-Fenchel setting will require new ideas, we expect that the 
techniques that are introduced in this paper could provide a basis for 
further developments.
\end{rem}

\section{Preliminaries}
\label{sec:pre}

The aim of this section is to recall some basic definitions and facts from 
convex geometry and functional analysis that will be used in the sequel.

We highlight at the outset the following convention. In 
Minkowski's inequalities and in many arguments in this paper, 
mixed volumes $\V(K,L,C_1,\ldots,C_{n-2})$ of convex bodies 
$K,L,C_1,\ldots,C_{n-2}$ in $\mathbb{R}^n$ are considered for fixed 
$C_1,\ldots,C_{n-2}$, and only $K$ and $L$ are varied. We therefore 
introduce once and for all the notation
$$
	\mathcal{C}:=(C_1,\ldots,C_{n-2}),\qquad
	\mathcal{M}:=(\underbrace{M,\ldots,M}_{n-2}),
$$
and we will write $\V(K,L,\mathcal{C}):=\V(K,L,C_1,\ldots,C_{n-2})$,
$S_{B,\mathcal{M}}:=S_{B,M,\ldots,M}$, etc. Throughout this paper,
we always denote by $B$ the Euclidean unit ball in $\mathbb{R}^n$.

By convention, we will consider linear operations 
$\lambda_1K_1+\cdots+\lambda_mK_m$ on convex bodies to always have 
positive coefficients $\lambda_1,\ldots,\lambda_m\ge 0$. The linearity of 
mixed volumes $K\mapsto \V(K,C_1,\ldots,C_{n-1})$ should be intepreted in 
this sense. In contrast, differences of support functions
form a genuine linear space (over the reals).

\subsection{Mixed volumes and mixed area measures}
\label{sec:mixva}

Mixed volumes are defined by \eqref{eq:mixedv}. They have the following 
basic properties \cite[section 5.1]{Sch14}.

\begin{lem}
\label{lem:mvprop}
Let $K,K',K_1,\ldots,K_n$ be convex bodies in $\mathbb{R}^n$.
\begin{enumerate}[a.]
\item $\V(K,\ldots,K) = \Vol(K)$.
\item $\V(K_1,\ldots,K_n)$ is a symmetric and multilinear functional of
$K_1,\ldots,K_n$.
\item $\V(K_1,\ldots,K_n)\ge 0$.
\item $\V(K,K_2,\ldots,K_n)\ge \V(K',K_2,\ldots,K_n)$ if
$K\supseteq K'$.
\item $\V(K_1,\ldots,K_n)$ is invariant under translation $K_i\mapsto
K_i+v_i$.
\end{enumerate}
\end{lem}

There is a close connection between mixed volumes and mixed area measures
that we describe presently.
The \emph{surface area measure} of a convex body $K$ in 
$\mathbb{R}^n$ is the measure on $S^{n-1}$ such that
for any measurable set $A\subseteq S^{n-1}$
$$
	S(K,A) := \mathcal{H}^{n-1}(\{x\in\partial K: x\in F(K,u)
	\mbox{ for some }
	u\in A\}),
$$
where $\mathcal{H}^k$ denotes the $k$-dimensional Hausdorff measure.
That is, $S(K,A)$ is the surface area of the part of the 
boundary of $K$ with outer normal vectors in $A$. Just like volume, the 
surface area measure $S(K,\cdot\,)$ is polynomial in $K$, cf. \cite[p.\ 
279]{Sch14}:
$$
	S(\lambda_1K_1+\cdots+\lambda_m K_m,A) =
	\sum_{i_1,\ldots,i_{n-1}=1}^m S_{K_{i_1},\ldots,K_{i_{n-1}}}(A)
	\,\lambda_{i_1}\cdots\lambda_{i_{n-1}}
$$
for $\lambda_1,\ldots,\lambda_m\ge 0$.
The measures $S_{K_1,\ldots,K_{n-1}}$ are called \emph{mixed area 
measures}. They have the following properties \cite[section 5.1]{Sch14}.

\begin{lem}
\label{lem:maprop}
Let $K,K_1,\ldots,K_{n-1}$ be convex bodies in $\mathbb{R}^n$.
\begin{enumerate}[a.]
\item $S_{K,\ldots,K} = S(K,\cdot\,)$.
\item $S_{K_1,\ldots,K_{n-1}}$
is a symmetric and multilinear functional of $K_1,\ldots,K_{n-1}$.
\item $S_{K_1,\ldots,K_{n-1}}\ge 0$.
\item $S_{K_1,\ldots,K_{n-1}}$ is invariant under translation $K_i\mapsto
K_i+v_i$.
\item $\int \langle v,x\rangle\,S_{K_1,\ldots,K_{n-1}}(dx)=0$
for all $v\in\mathbb{R}^n$.
\end{enumerate}
\end{lem}

The key relation between mixed volumes and mixed area measures is the 
representation formula \cite[Theorem 5.1.7]{Sch14} that
we already introduced in section \ref{sec:keyideas}:
\begin{equation}
\label{eq:mvrep}
	\V(K_1,\ldots,K_n) = \frac{1}{n}\int h_{K_1}dS_{K_2,\ldots,K_n}.
\end{equation}
Note that while $K_1$ and $K_2,\ldots,K_n$ play different roles
in this representation, the expression is nonetheless symmetric under 
permutation of the $K_i$ by Lemma \ref{lem:mvprop}(\textit{b}).

We now record two important facts. First, mixed volumes and mixed area 
measures are continuous in the topology of Hausdorff convergence (i.e., 
$K^{(s)}\to K$ if and only if
$\|h_{K^{(s)}}-h_K\|_\infty\to 0$); see the proof of 
\cite[Theorem 5.1.7]{Sch14}.

\begin{thm}
\label{thm:conv}
Suppose that $K_1^{(s)},\ldots,K_n^{(s)}$ are convex bodies such that
$K_i^{(s)}\to K_i$ as $s\to\infty$ in the sense of Hausdorff convergence.
Then
$$
	\V(K_1^{(s)},\ldots,K_n^{(s)})\to
	\V(K_1,\ldots,K_n),\qquad
	S_{K_1^{(s)},\ldots,K_{n-1}^{(s)}}\wto
	S_{K_1,\ldots,K_{n-1}}
$$
as $s\to\infty$, where the limit of measures is in the sense of weak 
convergence.
\end{thm}

Second, we have the following support characterization
\cite[Theorem 4.5.3]{Sch14}.

\begin{thm}
\label{thm:supp}
Let $M$ be a convex body in $\mathbb{R}^n$ with nonempty interior. Then
\begin{align*}
	\supp S_{M,\mathcal{M}} &=
	\cl\{u\in S^{n-1}:u\mbox{ is a }0\mbox{-extreme normal vector
	of } M\},\\
	\supp S_{B,\mathcal{M}} &=
	\cl\{u\in S^{n-1}:u\mbox{ is a }1\mbox{-extreme normal vector
	of } M\}.
\end{align*}
\end{thm}

Let us finally note that as mixed volumes and mixed area measures are 
linear functionals of the underlying bodies (and hence of their support 
functions), their definitions extend naturally by linearity to functions 
that are differences of two support functions \cite[section 5.2]{Sch14}. 
By a slight abuse of notation, we will write
\begin{align*}
	&\V(f,K_2,\ldots,K_n) := \V(K,K_2,\ldots,K_n)-
	\V(K',K_2,\ldots,K_n),\\
	&S_{f,K_2,\ldots,K_{n-1}} :=
	S_{K,K_2,\ldots,K_{n-1}}-S_{K',K_2,\ldots,K_{n-1}}
\end{align*}
for functions of the form $f=h_{K}-h_{K'}$. We may similarly extend 
further arguments by linearity to write $\V(f,g,K_3,\ldots,K_n)$, etc.
In particular, as will be recalled in the following section, any $C^2$ 
function on the sphere can be written as the difference of two support 
functions, so that mixed volumes and mixed area measures are well defined 
when their arguments are arbitrary $C^2$ functions. Of course, 
$\V(f,K_2,\ldots,K_n)$ need not be nonnegative and 
$S_{f,K_2,\ldots,K_{n-1}}$ may be a signed measure.

\subsection{Smooth convex bodies}
\label{sec:smooth}

For sufficiently smooth bodies, mixed volumes and area measures may 
be expressed in a more explicit form, cf.\ \cite[section 2]{SvH18}.

To define the appropriate regularity, we recall that the support function 
$h_K$ may be viewed either as a function on the sphere $S^{n-1}$ or, 
equivalently, as a $1$-homogeneous function on $\mathbb{R}^n$. Now suppose 
that $h_K$ is a $C^2$ function on $\mathbb{R}^n\backslash\{0\}$. Then 
$\nabla h_K$ is $0$-homogeneous, so the derivative of $\nabla h_K$ in the 
radial direction vanishes. The Hessian $\nabla^2h_K(x)$ in $\mathbb{R}^n$ 
may therefore be viewed as a linear mapping from $x^\perp$ to itself. We 
denote this mapping by $D^2h_K$. For an arbitrary $C^2$ function 
$f:S^{n-1}\to\mathbb{R}$, the restricted Hessian $D^2f$ is defined 
analogously by applying the above construction to the $1$-homogeneous 
extension of $f$. We may also express this notion intrinsically as 
$D^2f=\nabla_{S^{n-1}}^2f+fI$ in terms of the covariant Hessian on the 
sphere. We now have the following classical fact \cite[section 
2.1]{SvH18}; here and below, $A>0$ ($A\ge 0$) denotes that the matrix $A$ 
is positive definite (semidefinite).

\begin{lem}
\label{lem:diffsf}
Let $f:S^{n-1}\to\mathbb{R}$ be a $C^2$ function. Then $f=h_K$ for some
convex body $K$ if and only if
$D^2f\ge 0$. In particular, any $C^2$ function
satisfies $f=h_K-h_L$ for some convex bodies $K,L$ \upn{(}as
$D^2(f+h_{\lambda B}) = D^2f + \lambda I \ge 0$ for large 
$\lambda$\upn{)}.
\end{lem}

We now formulate the following definition.

\begin{defn}
A convex body $K$ is of class $C^k_+$ ($k\ge 2$) if $h_K$ is $C^k$ and
$D^2h_K>0$.
\end{defn}

It can be shown \cite[sections 2.5 and 3.4]{Sch14} that $K$ is of class 
$C^k_+$ if and only if its boundary $\partial K$ is a $C^k$-submanifold 
of $\mathbb{R}^n$, and the function $n_K:\partial K\to S^{n-1}$ that 
maps each boundary point to its (unique) outer normal is a 
$C^{k-1}$-diffeomorphism.

For a $C^2_+$ body $K$, one can obtain an explicit expression for the 
surface area measure $S(K,\cdot\,)$ by using the outer normal map $n_K$ 
to perform a change of variables in its definition. Using the basic fact 
$n_K^{-1}=\nabla h_K$ \cite[Corollary 1.7.3]{Sch14}, this yields
$$
	S(K,d\omega) = \det(D^2h_K)\,d\omega,
$$
where $\omega$ denotes the surface measure on $S^{n-1}$. Geometrically,
this expression states that the density of the surface area measure of a 
$C^2_+$ body $K$ at the point $u\in S^{n-1}$ is the reciprocal Gauss 
curvature of its boundary $\partial K$ at the point $n_K^{-1}(u)$.

To extend this expression to mixed area measures, note that we can write
$$
	\det(\lambda_1A_1+\cdots+\lambda_mA_m) =
	\sum_{i_1,\ldots,i_{n-1}=1}^m
	\D(A_{i_1},\ldots,A_{i_{n-1}})\,\lambda_{i_1}\cdots\lambda_{i_{n-1}}
$$
for any $(n-1)$-dimensional matrices $A_i$ and $\lambda_i\ge 0$, as 
$A\mapsto\det A$ is a homogeneous polynomial. The coefficients 
$\D(A_1,\ldots,A_{n-1})$ are called \emph{mixed discriminants}.
The definition of mixed area measures and \eqref{eq:mvrep}
now yield the following.

\begin{lem}
\label{lem:repc2}
Let $K_1,\ldots,K_n$ be convex bodies in $\mathbb{R}^n$ of class
$C^2_+$. Then
\begin{align*}
	S_{K_2,\ldots,K_{n}}(d\omega) &=
	\D(D^2h_{K_2},\ldots,D^2h_{K_{n}})\,d\omega,
	\\
	\V(K_1,\ldots,K_n) &= \frac{1}{n}
	\int h_{K_1}\D(D^2h_{K_2},\ldots,D^2h_{K_{n}})\,d\omega.
\end{align*}
\end{lem}

Let us recall some basic properties of mixed 
discriminants.

\begin{lem}
\label{lem:mdprop}
Let $A,B,A_1,\ldots,A_{n-1}$ be symmetric $(n-1)$-dimensional matrices.
\begin{enumerate}[a.]
\item $\D(A,\ldots,A) = \det(A)$.
\item $\D(B,A,\ldots,A) = \frac{1}{n-1}\Tr[\cof(A)B]$.
\item $\D(A_1,\ldots,A_{n-1})$ is symmetric and multilinear in its 
arguments.
\item $\D(A,A_2,\ldots,A_{n-1})\ge \D(B,A_2,\ldots,A_{n-1})$ if
$A\ge B$, $A_2,\ldots,A_{n-1}\ge 0$.
\end{enumerate}
\end{lem}

These properties may be found in \cite[Lemma 2.6]{SvH18} except part
$(\textit{b})$, which follows readily by differentiation
$\frac{d}{dt}\det(A+tB)|_{t=0}=(n-1)\D(B,A,\ldots,A)$.

We finally recall that the directional derivatives of support functions 
have a geometric meaning for any (non-smooth) convex body 
\cite[Theorem 1.7.2]{Sch14}.

\begin{lem}
\label{lem:dirdir}
Let $K$ be any convex body in $\mathbb{R}^n$. Then
$\nabla_xh_K(u) = h_{F(K,u)}(x)$ for all $u,x\in S^{n-1}$, where
$\nabla_x$ denotes the directional derivative in $\mathbb{R}^n$ in 
direction $x$.
\end{lem}

\subsection{Self-adjoint operators and quadratic forms}
\label{sec:presa}

It is an elementary fact of linear algebra that every symmetric quadratic 
form $Q$ on $\mathbb{R}^n$ can be represented in terms of a symmetric 
matrix $A$ as $Q(x,y)=\langle x,Ay\rangle$. Moreover, the spectral theorem 
states that $A$ can be diagonalized by a proper choice of basis. Analogous 
notions exist in infinite dimension, but the definitions 
and properties of the relevant objects are more subtle. The aim of 
this section is to give a brief review of the requisite notions; the 
reader is referred to \cite{RS72,Dav95} for textbook treatments.

\subsubsection*{\bf Self-adjointness and the spectral theorem}

Let $H$ be a real\footnote{%
	While most references on spectral theory assume complex $H$, the 
	relevant notions apply readily to the real case
	\cite[Remark 20.18]{MV97}; 
	we work with real $H$ primarily for notational convenience.
} separable Hilbert space. A \emph{linear operator} on $H$ is a 
linear map $\mathscr{L}:\Dom \mathscr{L}\to H$ whose domain is a 
dense subspace $\Dom\mathscr{L}\subseteq H$. The specification of 
the domain of $\mathscr{L}$ is part of its definition. For any linear 
operator $\mathscr{L}$, the associated \emph{adjoint operator}
$\mathscr{L}^*$ is defined by $\langle \mathscr{L}^*x,y\rangle= \langle 
x,\mathscr{L}y\rangle$ for all $y\in\Dom\mathscr{L}$ and 
$x\in\Dom\mathscr{L}^*$, where
$$
	\Dom\mathscr{L}^* :=
	\{ x\in H:\mbox{there is }z\in H\mbox{ so that }
	\langle x,\mathscr{L}y\rangle =
	\langle z,y\rangle\mbox{ for all }y\in\Dom\mathscr{L}\}.
$$
A linear operator $\mathscr{L}$ is \emph{self-adjoint}
if $\mathscr{L}=\mathscr{L}^*$ (in particular, $\Dom\mathscr{L}=
\Dom\mathscr{L}^*$). Note that the condition
$\langle \mathscr{L}x,y\rangle= \langle x,\mathscr{L}y\rangle$ for
$x,y\in\Dom\mathscr{L}$ is \emph{not} sufficient to ensure 
self-adjointness (such $\mathscr{L}$ is said to be 
\emph{symmetric}).

The spectral theorem for self-adjoint operators 
\cite[Theorem VIII.4]{RS72} states that any self-adjoint operator
can be diagonalized.

\begin{thm}[Spectral theorem]
\label{thm:spectral}
For any self-adjoint operator $\mathscr{L}$, there exists a measure space
$(\Omega,\mu)$ with a finite measure $\mu$, an invertible linear isometry
$U:H\to L^2(\Omega,\mu)$, and a measurable function 
$\eta:\Omega\to\mathbb{R}$ such that
\begin{enumerate}[a.]
\item $\Dom\mathscr{L} = \{ U^{-1}\psi : \psi\in L^2(\Omega,\mu)
\mbox{ such that }\eta\psi\in L^2(\Omega,\mu)\}$; and
\item $U\mathscr{L}U^{-1}\psi = \eta\psi$ for every $\psi\in 
L^2(\Omega,\mu)$.
\end{enumerate}
\end{thm}

The spectral theorem gives rise to a spectral calculus in direct
analogy to the finite-dimensional case. Let
$\mathscr{L},\Omega,\mu,U,\eta$ be as in Theorem \ref{thm:spectral}. 
Then for any measurable function $\varphi:\mathbb{R}\to\mathbb{R}$, we can 
define a self-adjoint operator $\varphi(\mathscr{L})$ with 
$$
	\Dom \varphi(\mathscr{L}) :=
	\{ U^{-1}\psi : \psi\in L^2(\Omega,\mu)\mbox{ such that }
	(\varphi\circ\eta)\psi\in L^2(\Omega,\mu)\}
$$
and $\varphi(\mathscr{L})x:= U^{-1}(\varphi\circ\eta)Ux$ for $x\in\Dom 
\varphi(\mathscr{L})$.

The \emph{spectrum} $\spec\mathscr{L}$ of a self-adjoint operator is the 
essential range of the function $\eta$ in the spectral theorem. $\mathscr{L}$ 
is \emph{nonnegative} if $\spec\mathscr{L}\subseteq[0,\infty)$.

\subsubsection*{\bf Nonnegative forms}

We now consider the correspondence between self-adjoint operators and 
symmetric quadratic forms. Let 
$\mathcal{E}:\Dom\mathcal{E}\times\Dom\mathcal{E}\to\mathbb{R}$ be a 
bilinear map that is defined on a dense subspace $\Dom\mathcal{E}\subseteq 
H$. We call $\mathcal{E}$ a \emph{nonnegative form} if 
$\mathcal{E}(f,g)=\mathcal{E}(g,f)$ and $\mathcal{E}(f,f)\ge 0$. A 
nonnegative form is \emph{closed} if $\Dom\mathcal{E}$ is complete for the 
norm $\vvvert f\vvvert := [\|f\|^2+\mathcal{E}(f,f)]^{1/2}$. 

It is a basic fact that closed nonnegative forms are in one-to-one 
correspondence with nonnegative self-adjoint operators on $H$: for 
any nonnegative self-adjoint operator $\mathscr{L}$, the form 
$\mathcal{E}(f,g):=\langle\mathscr{L}^{1/2}f,\mathscr{L}^{1/2}g\rangle$ 
with $\Dom\mathcal{E}=\Dom\mathscr{L}^{1/2}$ is closed; while for every 
closed nonnegative form $\mathcal{E}$, there is a nonnegative self-adjoint 
operator $\mathscr{L}$ such that the above representation holds 
\cite[Theorem 4.4.2]{Dav95}.

Note that when $g\in\Dom\mathscr{L}$, 
we may also write $\mathcal{E}(f,g)=\langle f,\mathscr{L}g\rangle$; in 
general, however, $\Dom\mathcal{E}$ is considerably larger than 
$\Dom\mathscr{L}$. On the other hand, any minimizer $f$ of 
$\mathcal{E}(f,f)$ must exhibit additional regularity 
$f\in\Dom\mathscr{L}$.

\begin{lem}
\label{lem:ker}
Let $\mathcal{E}$ be a closed nonnegative form and
$\mathscr{L}$ the associated self-adjoint operator.
Then $f\in\Dom\mathcal{E}$, $\mathcal{E}(f,f)=0$ if and only if  
$f\in\Dom\mathscr{L}$, $\mathscr{L}f=0$.
\end{lem}

\begin{proof}
As $\Dom 
\mathscr{L}=\{f\in\Dom\mathscr{L}^{1/2}:\mathscr{L}^{1/2}f\in 
\Dom\mathscr{L}^{1/2}\}$ by definition of $\mathscr{L}^{1/2}$,
the conclusion follows from 
$\mathcal{E}(f,f)=\|\mathscr{L}^{1/2}f\|^2=\langle f,\mathscr{L}f\rangle$
for $f\in\Dom\mathscr{L}$.
\end{proof}

\subsubsection*{\bf Semibounded forms}

In the context of Lemma \ref{lem:hyper}, the operator $\mathscr{A}$ is not 
nonnegative but rather bounded from above. The above notions extend 
directly to this setting modulo a change in notation.
A densely defined symmetric bilinear map $\mathcal{E}(f,g)$ 
is called a \emph{$c$-semibounded form} if $\mathcal{E}(f,f)\le c\|f\|^2$.
Then $\mathcal{E}'(f,g):=c\langle 
f,g\rangle-\mathcal{E}(f,g)$ defines a nonnegative form, and $\mathcal{E}$ 
is said to be closed if $\mathcal{E}'$ is closed, that is, if
$\Dom\mathcal{E}$ is complete for the norm $\vvvert f\vvvert := 
[(1+c)\|f\|^2-\mathcal{E}(f,f)]^{1/2}$ (here $1+c$ may be replaced
by any $C\ge 1+c$, as this defines an equivalent norm).

By applying the correspondence between nonnegative self-adjoint operators 
and closed nonnegative forms to $\mathcal{E}'$, we obtain the analogous 
statements for the semibounded form $\mathcal{E}$. In particular, for any 
closed $c$-semibounded form $\mathcal{E}$, there is a self-adjoint 
operator $\mathscr{A}\le cI$ so that $\mathcal{E}(f,g)=\langle 
f,\mathscr{A}g\rangle$ for $g\in\Dom\mathscr{A}$. Conversely, any 
self-adjoint operator $\mathscr{A}\le cI$ defines such a closed 
$c$-semibounded form $\mathcal{E}$.

% such that the above representation holds.
%As the situation will generally be clear from context, we will often 
%simply speak of the closed quadratic form associated to a self-adjoint 
%operator (which is either nonnegative or semibounded).

\subsubsection*{\bf Friedrichs extension}

It is often not a trivial matter to define self-adjoint operators. In 
practice, an operator is often naturally specified on a smaller domain, 
and one must understand how to extend its definition to a larger domain on 
which it is self-adjoint. Semibounded forms provide a powerful tool for 
this purpose.

A semibounded form $\mathcal{E}$ is \emph{closable} if it can 
be extended to a larger domain on which it is closed. The smallest closed 
extension is called the \emph{closure} $\mathcal{\bar E}$, and
$\Dom\mathcal{\bar E}$ is the completion of $\Dom\mathcal{E}$ 
with respect to the norm $\vvvert\cdot\vvvert$ defined above.

The following procedure is 
known as Friedrichs extension \cite[Theorem 4.4.5]{Dav95}.

\begin{lem}
\label{lem:friedrichs}
Let $\mathscr{A}$ be a densely defined operator on $H$ such that the 
bilinear map $\mathcal{E}(f,g)=\langle f,\mathscr{A}g\rangle$ is symmetric 
and semibounded on the domain of $\mathscr{A}$. Then 
$\mathcal{E}$ is closable and its closure defines a 
self-adjoint extension of $\mathscr{A}$.
\end{lem}

\subsubsection*{\bf Hyperbolic forms}

We conclude this section by proving Lemma \ref{lem:hyper}. By 
the spectral theorem, we may assume without loss of generality in the 
proof that $H=L^2(\Omega,\mu)$, $\mathscr{A}\psi=\eta\psi$ with
$0<\mathop{\mathrm{ess\,sup}}\eta<\infty$, and
$\mathcal{E}(f,g)=\int fg\eta\,d\mu$ with $\Dom\mathcal{E}=
\{f\in L^2(\Omega,\mu):|\int f^2\eta\,d\mu|<\infty\}$, for which 
the various arguments become explicit. Note that in this setting, the 
condition $\rank \mathrm{1}_{(0,\infty)}(\mathscr{A})=1$ holds if and 
only if $\mu$ has an atom at some point $\alpha\in\Omega$ with 
$\eta(\alpha)>0$ and $\eta(\omega)\le 0$ for $\mu$-a.e.\ $\omega\ne\alpha$.

\begin{proof}[Proof of Lemma \ref{lem:hyper}]
We prove each implication separately.

\vskip.1cm

$\boldsymbol{2\Rightarrow 1}$. Suppose condition $2$ holds. Then the 
spectral theorem 
implies that there exists $v\in \Dom\mathscr{A}$, $\lambda>0$
such that $\mathscr{A}v=\lambda v$, and such that
the following holds:
\begin{equation}
\label{eq:neg}
	\mathcal{E}(h,h)\le 0\quad\mbox{for }
	h\in\Dom\mathcal{E},~
	h\perp v.
\end{equation}
Let $f,g\in\Dom\mathcal{E}$ such that $\mathcal{E}(g,g)>0$.
Then $\langle g,v\rangle\ne 0$ by \eqref{eq:neg}. Define
$z = f-ag$ with $a=\langle f,v\rangle/\langle g,v\rangle$. Then
$z\perp v$, so applying \eqref{eq:neg} again yields
$$
	0\ge \mathcal{E}(z,z) =
	\mathcal{E}(f,f) -2a\mathcal{E}(f,g)+a^2\mathcal{E}(g,g)
	\ge
	\mathcal{E}(f,f) - \frac{\mathcal{E}(f,g)^2}{\mathcal{E}(g,g)},
$$
where the last inequality follows by minimizing over $a$.
Condition $1$ follows readily.

\vskip.2cm

$\boldsymbol{1\Rightarrow 2}$. Suppose condition $2$ is violated. As 
$0<\sup\spec\mathscr{A}<\infty$, the spectral theorem implies that 
$H_+:=1_{(0,\infty)}(\mathscr{A})H$ satisfies $H_+\subset\Dom\mathscr{A}$, 
$\mathscr{A}H_+\subseteq H_+$, $\dim H_+\ge 2$, and $\mathcal{E}(g,g)>0$ 
for $g\in H_+\backslash\{0\}$. That is, $\mathscr{A}$ 
is a bounded positive definite operator on its positive eigenspace, which 
has dimension at least two.

Now choose any nonzero $g\in H_+$ and $f\in H_+$ such that $f\perp 
\mathscr{A}g$. Then $\mathcal{E}(f,f)>0$, $\mathcal{E}(g,g)>0$, and 
$\mathcal{E}(f,g)=\langle f,\mathscr{A}g\rangle=0$. Thus condition
$1$ is violated.

\vskip.2cm

$\boldsymbol{1'\Rightarrow 2'}$. Suppose condition $1'$ holds. Then each 
inequality in the proof of $2\Rightarrow 1$ must be equality. In 
particular, $\mathcal{E}(z,z)=0$. Now note that by the spectral theorem,
$z\in H_-:=1_{(-\infty,0]}(\mathscr{A})H$ and $\mathscr{A}$ is a 
nonpositive operator on $H_-$. By Lemma \ref{lem:ker}, we obtain
$z\in\Dom\mathscr{A}$ and $\mathscr{A}z=0$. Thus we have established
condition $2'$.

\vskip.2cm

$\boldsymbol{2'\Rightarrow 1'}$. Suppose that condition $2'$ holds.
As $z=f-ag\in\ker\mathscr{A}$, we have 
$0=\mathcal{E}(z,g) = \mathcal{E}(f,g)-a\mathcal{E}(g,g)$, so
$a=\mathcal{E}(f,g)/\mathcal{E}(g,g)$.
But then
$$
	0=\mathcal{E}(z,z) = \mathcal{E}(f,f) -
	\frac{\mathcal{E}(f,g)^2}{\mathcal{E}(g,g)},
$$
where we used again $z\in\ker\mathscr{A}$.
Condition $1'$ follows.
\end{proof}

\subsection{Elliptic operators}
\label{sec:elliptic}

In finite dimension, the spectrum of a symmetric matrix always coincides 
with its set of eigenvalues. This is not the case for general self-adjoint 
operators, however, unless additional regularity assumptions are imposed.

The basic criterion for a semibounded self-adjoint operator 
$\mathscr{A}\le cI$ to possess a nice eigendecomposition is that is has a 
\emph{compact resolvent}. For our purposes, the term need not be defined 
precisely, as we may take it to be synonymous with the following 
equivalent statement \cite[Theorem XIII.64]{RS78}: there is an orthonormal 
basis $\{\psi_n\}_{n\ge 1}\subset\Dom\mathscr{A}$ that spans $H$, and 
$c\ge\lambda_1\ge\lambda_2\ge\cdots$ with $\lambda_n\to-\infty$,
so that $\mathscr{A}\psi_n=\lambda_n\psi_n$ for each $n$ (and thus
$\spec\mathscr{A}=\{\lambda_n\}_{n\ge 1}$).
Just as for matrices, the eigenvalues $\lambda_n$ may be computed by 
the min-max principle \cite[section XIII.1]{RS78}.

\begin{lem}
\label{lem:minmaxprinciple}
Let $\mathscr{A}\le cI$ be a semibounded self-adjoint operator with 
compact resolvent, with eigenvalues $\{\lambda_n\}_{n\ge 1}$
and eigenvectors $\{\psi_n\}_{n\ge 1}$. Then
$$
	\lambda_n =
	\inf_{L\subset H:\dim L=n-1}
	\sup_{f\perp L}
	\frac{\langle f,\mathscr{A}f\rangle}{\|f\|^2} =
	\sup_{f\perp\{\psi_1,\ldots,\psi_{n-1}\}}
	\frac{\langle f,\mathscr{A}f\rangle}{\|f\|^2}.
$$
\end{lem}

We now recall a very useful example of self-adjoint operators with compact 
resolvent: elliptic second-order differential operators on compact 
manifolds. For concreteness, we limit the discussion to the special case 
of such operators that will be needed in this paper; see
\cite{GT01,Gri09,Dav95} for a comprehensive treatment.

Let $\mathrm{M}$ be a compact, connected Riemannian manifold (without 
boundary). Define the measure $d\mu:=\rho\,d\mathrm{Vol}_{\mathrm{M}}$, 
where $\rho>0$ is a positive smooth function on $\mathrm{M}$. Moreover, 
let $V$ be a smooth function on $\mathrm{M}$, and let $A$ be a smooth 
symmetric $(1,1)$-tensor field on $\mathrm{M}$. We now define for any 
$f\in C^2$
$$
	\mathscr{A}f(x) := \Tr[A(x)\nabla^2f(x)] + V(x)f(x).
$$
The operator $\mathscr{A}$ is called \emph{elliptic} if $A(x)$ is positive 
definite for every $x\in\mathrm{M}$, and is called \emph{symmetric} if 
$\langle f,\mathscr{A}g\rangle_{L^2(\mathrm{M},\mu)} = 
\langle\mathscr{A}f,g\rangle_{L^2(\mathrm{M},\mu)}$ for all $f,g\in C^2$.

Now note that by integration by parts, we have
for $f,g\in C^2$
$$
	\langle f,\mathscr{A}g\rangle_{L^2(\mathrm{M},\mu)}
	=
	-\int f \langle \rho^{-1}\div(\rho A),\nabla g\rangle\,d\mu
	-\int \langle \nabla f,A\nabla g\rangle\,d\mu
	+\int Vfg\, d\mu.
$$
If $\mathscr{A}$ is symmetric, it follows readily that $\div(\rho A)=0$.
Therefore, if $\mathscr{A}$ is a symmetric elliptic operator, then 
$\langle
f,\mathscr{A}f\rangle_{L^2(\mathrm{M},\mu)}\le (\sup
V)\|f\|^2_{L^2(\mathrm{M},\mu)}$ for all $f\in C^2$. We may then
apply Lemma \ref{lem:friedrichs} to extend $\mathscr{A}$ to a self-adjoint
operator on $L^2(\mathrm{M},\mu)$. As no confusion can arise, we will
denote the Friedrichs extension by the same symbol $\mathscr{A}$, and
refer to it as a self-adjoint elliptic operator.

The following result summarizes some basic facts of elliptic regularity 
theory; we refer to \cite[section 8.12]{GT01} and \cite[chapter 10]{Gri09} 
for a detailed treatment.

\begin{prop}
\label{prop:elliptic}
Let $\mathscr{A}$ be a self-adjoint elliptic operator on 
$(\mathrm{M},\mu)$. Then $\mathscr{A}$ has compact resolvent. Moreover, 
its top eigenvalue is simple $\lambda_1>\lambda_2$, and its eigenfunction 
$\psi_k$ has the same sign everywhere on $\mathrm{M}$ if and only if 
$k=1$.
\end{prop}

%That elliptic operators have a simple top eigenvalue and a 
%positive top eigenfunction is reminiscent of the Perron-Frobenius theorem 
%for positive matrices. This is not a coincidence: in probability theory, 
%positive matrices generate discrete Markov semigroups and elliptic 
%operators generate continuous Markov semigroups \cite{BGL14}.

\section{Mixed volumes and Dirichlet forms}
\label{sec:forms}

The aim of this section is to show that mixed volumes of arbitrary convex 
bodies can be represented as closed quadratic forms associated to 
self-adjoint operators. This was stated in the setting of Minkowski's 
quadratic inequality as Theorem \ref{thm:formsminkq} above, which is a 
special case of the following general theorem.

\begin{thm}
\label{thm:forms}
Let $\mathcal{C}=(C_1,\ldots,C_{n-2})$ be convex bodies in
$\mathbb{R}^n$ with $S_{B,\mathcal{C}}\not\equiv 0$.
Then there is a self-adjoint operator $\mathscr{A}$
on $L^2(S_{B,\mathcal{C}})$ with 
$C^2(S^{n-1}) \subset \Dom\mathscr{A}$ such that:
\begin{enumerate}[a.]
\item $\spec\mathscr{A}\subseteq (-\infty,0]\cup\{\frac{1}{n}\}$.
\item $\rank 1_{(0,\infty)}(\mathscr{A})=1$ and $\mathscr{A}1=\frac{1}{n}1$.
\item $\mathscr{A}\ell=0$ for any linear function
$\ell:x\mapsto\langle v,x\rangle$ on $S^{n-1}$ 
\upn{(}$v\in\mathbb{R}^n$\upn{)}.
\end{enumerate}
Moreover, the closed quadratic form $\mathcal{E}$ associated to 
$\mathscr{A}$ satisfies the following:
\begin{enumerate}[a.]
\setcounter{enumi}{3}
\item $h_K\in\Dom\mathcal{E}$ for every convex body $K$ in $\mathbb{R}^n$.
\item $\mathcal{E}(h_K,h_L)=\V(K,L,\mathcal{C})$ for any convex
bodies $K,L$ in $\mathbb{R}^n$.
\end{enumerate}
\end{thm}

The above theorem states, in particular, that linear functions are always 
in the kernel of $\mathscr{A}$. For the reasons explained in section 
\ref{sec:overview}, the central question in the study of the extremals of 
the Alexandrov-Fenchel inequality is whether these are the \emph{only} 
elements of $\ker\mathscr{A}$ that are differences of support functions.
The main part of this paper (sections 
\ref{sec:weak}--\ref{sec:lower}) will be devoted to settling this question 
in the setting $C_1=\cdots=C_{n-2}=M$ of Minkowski's quadratic inequality. 

The operator $\mathscr{A}$ of Theorem \ref{thm:forms} is defined somewhat 
abstractly. To actually work with such operators, one would like to have a 
more explicit formulation. Explicit constructions can be obtained in 
various special cases that will play an important role in the remainder of 
this paper. When $C_1,\ldots,C_{n-2}$ are $C^\infty_+$ bodies, 
$\mathscr{A}$ is a classical elliptic operator on $S^{n-1}$ in the sense 
of section \ref{sec:elliptic}. This case is already used in the proof of 
Theorem \ref{thm:forms}, which will be given in section 
\ref{sec:proofforms}. When $C_1,\ldots,C_{n-2}$ are polytopes, it turns 
out that $\mathscr{A}$ is a quantum graph \cite{BK13}; this setting will 
be developed in detail in section \ref{sec:qgraph}. A third setting in 
which $\mathscr{A}$ can be explicitly described will be encountered in 
section \ref{sec:lower} below.

In complete generality, we do not know how to give an expression for 
$\mathscr{A}$ that is amenable to explicit computations. Nonetheless, we 
may always view $\mathscr{A}$ as a highly degenerate elliptic second-order 
operator in the sense of the theory of Dirichlet forms \cite{FOT11,BGL14}. 
While the latter theory will not be needed in the remainder of this paper, 
we briefly develop this viewpoint in section \ref{sec:diri} in order to 
highlight the general structure that lies behind the representation of 
Theorem \ref{thm:forms}.

\begin{rem}
\label{rem:compres}
It should be emphasized at the outset that while $\mathscr{A}$ may be 
thought of quite generally as a kind of elliptic operator on the 
sphere, it does not necessarily possess some of the nice regularity 
properties of classical elliptic operators on compact manifolds. In 
particular, it turns out that $\mathscr{A}$ need not have a compact 
resolvent, as we will see in section 
\ref{sec:lower} below (cf.\ Remark \ref{rem:lowcompres}).
\end{rem}

\begin{rem}
The assumption $S_{B,\mathcal{C}}\not\equiv 0$ in Theorem \ref{thm:forms}
is innocuous. Indeed, if $S_{B,\mathcal{C}}\equiv 0$, then it follows
(for example, using Lemma \ref{lem:density} below) that
$\V(K,L,\mathcal{C})=0$ for all convex bodies $K,L$, so that
the representation of mixed volumes is trivial.
\end{rem}

\subsection{Proof of Theorem \ref{thm:forms}}
\label{sec:proofforms}

Throughout this section, the assumptions and notation of Theorem 
\ref{thm:forms} are in force. The proof is based on the following 
observation.

\begin{lem}
\label{lem:density}
For any $C^2$ function $f$ on $S^{n-1}$, we have
$S_{f,\mathcal{C}}\ll S_{B,\mathcal{C}}$ and
$$
	\bigg\|\frac{dS_{f,\mathcal{C}}}{dS_{B,\mathcal{C}}}\bigg\|_\infty
	\le \|D^2f\|_\infty,
$$
where we use the notation $\|D^2f\|_\infty:=\sup_{x\in 
S^{n-1}}\|D^2f(x)\|$.
\end{lem}

\begin{proof}
Let $C_i^{(s)}$ be convex bodies of class $C^\infty_+$ such that
$C_i^{(s)}\to C_i$ as $s\to\infty$ in the sense of Hausdorff convergence
for every $i=1,\ldots,n-2$. The existence of such smooth approximations
is classical, cf.\ \cite[section 3.4]{Sch14}.

For $\mathcal{C}^{(s)}=(C_1^{(s)},\ldots,C_{n-2}^{(s)})$, Lemmas 
\ref{lem:repc2} and \ref{lem:diffsf} imply
$$
        S_{f,\mathcal{C}^{(s)}}(d\omega) =
        \D(D^2f,D^2h_{C_1^{(s)}},\ldots,D^2h_{C_{n-2}^{(s)}})\,d\omega
$$
for any $C^2$ function $f$. In particular, as $D^2h_B=I$, we have
$$
	\frac{dS_{f,\mathcal{C}^{(s)}}}{dS_{B,\mathcal{C}^{(s)}}}
	=
	\frac{\D(D^2f,D^2h_{C_1^{(s)}},\ldots,D^2h_{C_{n-2}^{(s)}})}
	{\D(I,D^2h_{C_1^{(s)}},\ldots,D^2h_{C_{n-2}^{(s)}})}
$$
(note that the denominator is strictly positive by
Lemma \ref{lem:mdprop}(\textit{d}), so the expression is well-defined).
As $\lambda_{\rm min}(D^2f)I \le D^2f \le\lambda_{\rm max}(D^2f)I$,
Lemma \ref{lem:mdprop}(\textit{d}) shows that 
$$
	\bigg|\frac{dS_{f,\mathcal{C}^{(s)}}}{dS_{B,\mathcal{C}^{(s)}}}
	\bigg|\le \|D^2f\|
$$
pointwise. Therefore, Theorem \ref{thm:conv} implies that
\begin{align}
\nonumber
	\bigg|\int g\, dS_{f,\mathcal{C}}\bigg|
	&=
	\lim_{s\to\infty}
	\bigg|\int g\, dS_{f,\mathcal{C}^{(s)}}\bigg|
	\\
\label{eq:densineq}
	&\le
	\limsup_{s\to\infty}
	\int |g|\, \|D^2f\|\,dS_{B,\mathcal{C}^{(s)}}
	\\
\nonumber
	&\le
	\|D^2f\|_\infty
	\int |g|\,dS_{B,\mathcal{C}}
\end{align}
for every continuous function $g$ on $S^{n-1}$. In particular,
$$
	g\mapsto\int g\,dS_{f,\mathcal{C}}
$$
is a bounded linear map from $C^0(S^{n-1})\subset 
L^1(S_{B,\mathcal{C}})$ to $\mathbb{R}$, and therefore extends uniquely to a 
bounded linear functional on $L^1(S_{B,\mathcal{C}})$.
As $L^1(S_{B,\mathcal{C}})^*=L^\infty(S_{B,\mathcal{C}})$,
$$
	\int g \,dS_{f,\mathcal{C}} = 
	\int g\,\varrho\,dS_{B,\mathcal{C}}
$$
for some $\varrho\in L^\infty(S_{B,\mathcal{C}})$.
This proves absolute
continuity $S_{f,\mathcal{C}}\ll S_{B,\mathcal{C}}$, and the bound 
on $\varrho$ follows by taking the supremum in \eqref{eq:densineq}
over $g$ with 
$\int|g|\,dS_{B,\mathcal{C}}\le 1$.
\end{proof}

We will also need the following simple consequence.

\begin{lem}
\label{lem:densuniq}
Let $f,g\in C^2(S^{n-1})$ satisfy $f=g$ $S_{B,\mathcal{C}}$-a.e.
Then $S_{f,\mathcal{C}}=S_{g,\mathcal{C}}$.
\end{lem}

\begin{proof}
Let $h\in C^2(S^{n-1})$. Then
$$
	\int h \,dS_{f,\mathcal{C}} -
	\int h \,dS_{g,\mathcal{C}} =
	n\,\V(h,f-g,\mathcal{C}) = 
	\int (f-g)\,dS_{h,\mathcal{C}} = 0,
$$
where we have used \eqref{eq:mvrep}, the symmetry of mixed volumes, and 
Lemma \ref{lem:density}. As this holds for any $h\in C^2(S^{n-1})$, 
the conclusion follows.
\end{proof}

The idea behind Theorem \ref{thm:forms} is now simple.
By Lemma \ref{lem:density}, we can define
\begin{equation}
\label{eq:defa}
	\mathscr{A}f := 
	\frac{1}{n}\frac{dS_{f,\mathcal{C}}}{dS_{B,\mathcal{C}}}
	\quad \mbox{for }f\in C^2(S^{n-1}),
\end{equation}
and Lemma \ref{lem:densuniq} ensures that
$\mathscr{A}$ is well-defined as a linear operator on $C^2(S^{n-1})
\subset L^2(S_{B,\mathcal{C}})$ (here we implicitly identified 
those functions in $C^2$ that agree up to $S_{B,\mathcal{C}}$-null sets).
We aim to extend $\mathscr{A}$ to a bona fide self-adjoint operator by
Friedrichs extension. To this end, we must show that $\mathscr{A}$ 
is semibounded.

\begin{lem}
\label{lem:asemib}
For every $f\in C^2(S^{n-1})$, we have
$$
	n\int f\mathscr{A}f\,dS_{B,\mathcal{C}}=
	\int f\,dS_{f,\mathcal{C}} \le 
	\int f^2\,dS_{B,\mathcal{C}}.
$$
\end{lem}

\begin{proof}
Let $C_i^{(s)}$ be convex bodies of class $C^\infty_+$ such that
$C_i^{(s)}\to C_i$ as $s\to\infty$. Then
$$
	\int f\,dS_{f,\mathcal{C}^{(s)}} \to
	\int f\,dS_{f,\mathcal{C}},\qquad
	\int f^2\,dS_{B,\mathcal{C}^{(s)}} \to
	\int f^2\,dS_{B,\mathcal{C}}
$$
as $s\to\infty$ by Theorem \ref{thm:conv}. It therefore suffices to assume
in the remainder of the proof that $C_1,\ldots,C_{n-2}$ are convex bodies 
of class $C^\infty_+$.

For $C^\infty_+$ bodies, we may write as in the proof of Lemma 
\ref{lem:density}
$$
	\mathscr{A}f = 
	\frac{1}{n}\frac{dS_{f,\mathcal{C}}}{dS_{B,\mathcal{C}}}
	=
	\frac{1}{n}\frac{\D(D^2f,D^2h_{C_1},\ldots,D^2h_{C_{n-2}})}
	{\D(I,D^2h_{C_1},\ldots,D^2h_{C_{n-2}})}.
$$
Using $D^2f=\nabla_{S^{n-1}}^2f+fI$ and Lemma \ref{lem:mdprop}(\textit{d}),
we see that $\mathscr{A}$ is an elliptic operator in the 
sense of section \ref{sec:elliptic}. Moreover, by 
\eqref{eq:mvrep} and the
symmetry of mixed volumes, $\langle 
f,\mathscr{A}g\rangle_{L^2(S_{B,\mathcal{C}})}=\V(f,g,\mathcal{C})$ is
symmetric in $f,g$. It follows from section 
\ref{sec:elliptic} that $\mathscr{A}$ 
has a self-adjoint extension for which $\lambda:=\sup\spec\mathscr{A}$ is 
the unique eigenvalue associated to a positive eigenfunction. As 
$\mathscr{A}1=\frac{1}{n}1$, we have $\lambda=\frac{1}{n}$ and 
thus
$$
	\int f\,dS_{f,\mathcal{C}} =
	n\langle f,\mathscr{A}f\rangle_{L^2(S_{B,\mathcal{C}})}
	\le \|f\|^2_{L^2(S_{B,\mathcal{C}})},
$$
concluding the proof.
\end{proof}

We are now ready to prove Theorem \ref{thm:forms}.

\begin{proof}[Proof of Theorem \ref{thm:forms}]
For $f,g\in C^2(S^{n-1})$, define $\mathscr{A}f$ as in \eqref{eq:defa} and
let
$$
	\mathcal{E}(f,g):=
	\langle f,\mathscr{A}g\rangle_{L^2(S_{B,\mathcal{C}})}
	= \V(f,g,\mathcal{C}).
$$
By the symmetry of mixed volumes and Lemma \ref{lem:asemib}, $\mathcal{E}$ 
is a densely defined and $\frac{1}{n}$-semibounded symmetric quadratic 
form on $L^2(S_{B,\mathcal{C}})$. It follows from Lemma \ref{lem:friedrichs} 
that $\mathcal{E}$ is closable and that its closure defines a self-adjoint
extension of $\mathscr{A}$. This concludes the construction of the objects
$\mathscr{A}$ and $\mathcal{E}$ that appear in the statement of Theorem 
\ref{thm:forms}. We now proceed to verify each of the claimed properties.

\vskip.2cm

\textbf{Proof of \textit{d} and \textit{e}}. By construction, 
$\mathcal{E}(h_K,h_L)=\V(K,L,\mathcal{C})$ whenever $K,L$ are $C^2_+$ 
bodies. Now let $K$ be an arbitrary convex body, and let $K^{(s)}$ be
convex bodies of class $C^2_+$ such that $K^{(s)}\to K$
as $s\to\infty$. That
$$
	\|h_{K^{(s)}}-h_K\|_{L^2(S_{B,\mathcal{C}})}^2
	\le
	n\,\V(B,B,\mathcal{C})\,
	\|h_{K^{(s)}}-h_K\|_\infty^2\to 0
$$
as $s\to\infty$ 
follows immediately from Hausdorff convergence. Moreover,
\begin{align*}
	&\mathcal{E}(h_{K^{(s)}}-h_{K^{(t)}},
	h_{K^{(s)}}-h_{K^{(t)}}) = \\
	&\V(K^{(s)},K^{(s)},\mathcal{C})
	- 2\,\V(K^{(s)},K^{(t)},\mathcal{C})+
	\V(K^{(t)},K^{(t)},\mathcal{C})
	\to 0
\end{align*}
as $s,t\to\infty$ by Theorem \ref{thm:conv}. Thus $h_K$ is in the 
completion of $C^2(S^{n-1})$ for the norm
$[2\|f\|_{L^2(S_{B,\mathcal{C}})}^2 - 
\mathcal{E}(f,f)]^{1/2}$. Therefore, by construction,
$h_K\in\Dom\mathcal{E}$ and 
$$
	\mathcal{E}(h_K,h_K)= \lim_{s\to\infty}\mathcal{E}(h_{K^{(s)}},
	h_{K^{(s)}}) 
	= \lim_{s\to\infty}\V(K^{(s)},K^{(s)},\mathcal{C})
	= \V(K,K,\mathcal{C}).
$$
The analogous conclusion for $\mathcal{E}(h_K,h_L)$ follows by 
polarization.

\vskip.2cm

\textbf{Proof of \textit{a} and \textit{b}}. Note first that as
$h_B=1$, it follows immediately from the definition \eqref{eq:defa} 
that $\mathscr{A}1=\frac{1}{n}1$. Now suppose $\rank 
1_{(0,\infty)}(\mathscr{A})>1$,
and choose any nonzero
 $f\in 1_{(0,\infty)}(\mathscr{A})L^2(S_{B,\mathcal{C}})$
so that $\langle f,1\rangle_{L^2(S_{B,\mathcal{C}})}=0$. Choose $f_s\in 
C^2(S^{n-1})$ so that $\|f_s-f\|_{L^2(S_{B,\mathcal{C}})}\to 0$ and 
$\mathcal{E}(f_s,f_s)\to \mathcal{E}(f,f)$ as $s\to\infty$ (the existence 
of such a sequence is guaranteed by construction as $f\in\Dom\mathcal{E}$).
By Lemma \ref{lem:diffsf}, for all $s$ and for all sufficiently large 
$t>0$ (depending on $s$), there is a $C^2_+$ convex body $K_t$ such 
that $h_{K_t^{(s)}}=t+f_s$. Therefore, by the Alexandrov-Fenchel inequality,
\begin{align*}
	\langle t+f_s,\mathscr{A}1\rangle_{L^2(S_{B,\mathcal{C}})}^2 &=
	\V(K^{(s)}_t,B,\mathcal{C})^2 \ge
	\V(K^{(s)}_t,K^{(s)}_t,\mathcal{C})\,
	\V(B,B,\mathcal{C}) \\ &=
	\langle t+f_s,\mathscr{A}(t+f_s)\rangle_{L^2(S_{B,\mathcal{C}})}
	\langle 1,\mathscr{A}1\rangle_{L^2(S_{B,\mathcal{C}})}.
\end{align*}
Expanding the squares and using $\mathscr{A}1=\frac{1}{n}1$ yields
$$
	\langle f_s,1\rangle_{L^2(S_{B,\mathcal{C}})}^2 
	\ge
	n^2\,\mathcal{E}(f_s,f_s)\,
	\V(B,B,\mathcal{C}).
$$
In particular, letting $s\to\infty$ yields $\mathcal{E}(f,f)\le 0$,
where we used that $\V(B,B,\mathcal{C})>0$ by the assumption 
$S_{B,\mathcal{C}}\not\equiv 0$. But this contradicts the assumption
that $f\in 1_{(0,\infty)}(\mathscr{A})L^2(S_{B,\mathcal{C}})$. Thus
we have shown that $\rank 1_{(0,\infty)}(\mathscr{A})=1$, proving part
\textit{b}. Part \textit{a} now follows as an immediate consequence.

\vskip.2cm

\textbf{Proof of \textit{c}}. Let $\ell:x\mapsto\langle v,x\rangle$ be a 
linear function. Then $\ell=h_{K+v}-h_K$, so
$$
	\langle f,\mathscr{A}\ell\rangle_{L^2(S_{B,\mathcal{C}})}
	= \V(f,K+v,\mathcal{C})-\V(f,K,\mathcal{C}) = 0
$$
for every $f\in C^2(S^{n-1})$ by Lemma \ref{lem:mvprop}(\textit{e}).
Thus $\mathscr{A}\ell=0$.
\end{proof}

\subsection{Polytopes and quantum graphs}
\label{sec:qgraph}

The aim of this section is to provide an explicit description of the 
objects that appear in Theorem \ref{thm:forms} in the case where 
$C_1,\ldots,C_{n-2}$ are polytopes. In this case, it turns out that 
$\mathscr{A}$ is a ``quantum graph'': the one-dimensional 
Laplacian on the edges of a certain metric graph, with appropriately 
chosen boundary conditions at the vertices \cite{BK13}. For 
simplicity, we will develop this result in detail in the Minkowski case 
$C_1=\cdots=C_{n-2}=M$ that is of primary interest in this paper, and 
explain at the end of the section how the representation for general 
polytopes $C_1,\ldots,C_{n-2}$ may be obtained.

Throughout this section, let $M$ be a fixed polytope in $\mathbb{R}^n$
with nonempty interior.
We denote by $\mathcal{F}$ the set of facets of $M$, and by $n_F$ the 
outer unit normal vector of $F\in\mathcal{F}$.\footnote{
	For a smooth body $K$, we denote by $n_K:\partial K
	\to S^{n-1}$ the outer unit normal map on its boundary (cf.\ 
	section \ref{sec:smooth}). For a polytope $K$, the outer
	unit normal map is constant in the interior of each facet $F$;
	we then denote the facet normal as $n_F\in S^{n-1}$ by
	a slight abuse of notation.
}
We say that $F,F'\in\mathcal{F}$ are neighbors $F\sim F'$ if
$\dim(F\cap F')=n-2$.

We now associate to $M$ a metric graph $G=(V,E)$ inscribed in 
$S^{n-1}$ that will play a basic role in the sequel. The vertices of
$G$ are the facet normals
$$
	V=\{n_F:F\in\mathcal{F}\}\subset S^{n-1}.
$$
Moreover, for every pair of neighboring facets $F\sim F'$, we connect the 
corresponding vertices by an edge $e_{F,F'}\subset S^{n-1}$ that is the 
(shortest) geodesic segment in $S^{n-1}$ connecting $n_F$ and $n_{F'}$. 
Thus the undirected edges of $G$ are
$$
	E=\{e_{F,F'}:F\sim F'\}.
$$
The length of the edge $e_{F,F'}$ will be denoted as 
$l_{F,F'}:=\mathcal{H}^1(e_{F,F'})$.
\begin{figure}
\centering
\begin{tikzpicture}[scale=.9]

\shade[ball color = blue, opacity = 0.15] (1,0) circle [radius=2];

\draw[thick, densely dashed] (3,0) arc [start angle = 0, end angle = 180, 
x radius = 2, y radius = .5];

\draw[thick] (3,0) arc [start angle = 0, end angle = -180, x radius = 2, 
y radius = .5];

\draw[thick] (1,2) [rotate=90] arc [start angle = 0, end angle = 
-180, x radius = 2, y radius = -1];

\draw[thick, densely dashed] (1,2) [rotate=90] arc [start angle 
= 0, end angle = 
-180, x radius = 2, y radius = 1];

\draw[thick] (1,2) [rotate=90] arc [start angle = 0, end angle = 
-180, x radius = 2, y radius = 1.75];

\draw[thick, densely dashed] (1,2) [rotate=90] arc [start angle 
= 0, end angle = 
-180, x radius = 2, y radius = -1.75];

\draw[fill=black] (1,2) circle [radius=.07];
\draw[fill=black] (1,-2) circle [radius=.07];
\draw[fill=black] (2.74,-.24) circle [radius=.07] node[below right] 
{$n_{F}$};
\draw[fill=black] (-0.74,.24) circle [radius=.07];
\draw[fill=black] (.03,-.45) circle [radius=.07];
\draw[fill=black] (1.97,.45) circle [radius=.07];

\draw (-.25,-.44) node[below] {$n_{F'}$};

\draw (1.4,-.26) node {$e_{F,F'}$};

\draw[thick,->] (2.74,-.25) to (2,-.6);
\draw (3.5,-1.2) node[below right] {$n_{F\to F'}$};

\draw[->] (3.5,-1.2) to [out=110,in=-45] (2.4,-.5);

\draw (-1,0) node[left] {$G=$};

%%%%%%%%%%%%%%%%%%

\begin{scope}[scale=.9]

\fill[blue!15] (-5.95,.55) -- (-4.03,1.55) -- 
(-4.03,-1.25) -- (-5.95,-2.25) -- (-5.95,.55);

\fill[blue!25] (-5.95,.55) -- (-8.47,.89) -- 
(-8.47,-1.91) -- (-5.95,-2.25) -- (-5.95,.55);

\fill[blue!5] (-8.47,.89) -- (-6.53,1.79) -- 
(-4.03,1.55) -- (-5.95,.55) -- (-8.47,.89);

\draw[thick] (-5.95,.55) -- (-4.03,1.55) -- (-4.03,-1.25) --
(-5.95,-2.25) -- (-5.95,.55);

\draw[thick] (-5.95,.55) -- (-8.47,.89) -- (-8.47,-1.91) --
(-5.95,-2.25);

\draw[thick] (-8.47,.89) -- (-6.53,1.79) -- (-4.03,1.55);

\draw[thick,densely dashed] (-6.53,1.79) -- (-6.53,-1.01) --
(-4.03,-1.25);

\draw[thick,densely dashed] (-6.53,-1.01) -- (-8.47,-1.91);
\end{scope}

\draw (-7.7,0) node[left] {$M=$};

\draw (-6.5,-.6) node {$F'$};

\draw (-4.4,-.3) node {$F$};

\end{tikzpicture}
\caption{Metric graph $G$ associated to a polytope $M$.\label{fig:graph}}
\end{figure}
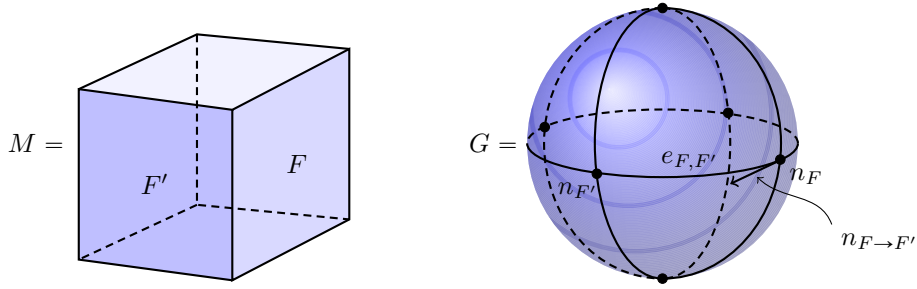

We will parametrize each edge by arclength. To this end, it is convenient 
to fix an arbitrary orientation of the edges by introducing a total 
ordering $\le$ on $\mathcal{F}$. When $F\le F'$, the edge $e_{F,F'}$ will 
be parametrized as $\theta\in[0,l_{F,F'}]$, where $\theta=0$ corresponds 
to $n_F$ and $\theta=l_{F,F'}$ corresponds to $n_{F'}$. For a function 
$f:e_{F,F'}\to\mathbb{R}$, we will write $f'=\frac{df}{d\theta}$ with 
respect to this parametrization. We denote by $H^k$ the usual Sobolev space
of functions on an edge with $k$ weak derivatives in $L^2$.

A function defined on all the edges of $G$ will be denoted 
$f:G\to\mathbb{R}$, and its (weak) derivative $f'$ is defined as above 
on each edge. In particular, any function $f:S^{n-1}\to\mathbb{R}$ may 
be viewed as a function on $G$ by restricting it to the edges, and then 
its derivative $f'$ is defined as above. We will write $f\in C^0(G)$ if 
$f:G\to\mathbb{R}$ is continuous on the edges as well as at each vertex.

Finally, we will denote by $n_{F\to F'}$ the unit tangent vector to the
geodesic segment $e_{F,F'}$ at $n_F$ in the direction of $n_{F'}$. The
directional derivative of a function $f:S^{n-1}\to\mathbb{R}$ at $n_F$ in 
the direction $n_{F\to F'}$ will be denoted $\nabla_{n_{F\to F'}}f(n_F)$. In 
terms of the arclength parametrization, we may evidently write
$$
	\nabla_{n_{F\to F'}}f(n_F) = f'|_{e_{F,F'}}(0)\,1_{F'\ge F}-
	f'|_{e_{F,F'}}(l_{F,F'})\,1_{F'\le F}.
$$
Note, in particular, that $\nabla_{n_{F\to F'}}f(n_F)$ is well defined
whenever $f|_{e_{F,F'}}\in H^2$ as $H^2(e_{F,F'})\subset C^1(e_{F,F'})$ by 
the Sobolev embedding theorem \cite[Theorem 7.26]{GT01}.
The various objects that we have defined above are illustrated in 
Figure \ref{fig:graph}.

We are now ready to state the main result of this section.

\begin{thm}
\label{thm:qgraph}
Let $M$ be a polytope in $\mathbb{R}^n$ with nonempty interior, and
define its metric graph $G$ as above. Then for any bounded 
measurable $f:S^{n-1}\to\mathbb{R}$, we have
\begin{align*}
	\int f\,dS_{M,\mathcal{M}} &=
	\sum_{F\in\mathcal{F}}\mathcal{H}^{n-1}(F)\,f(n_F),\\
	\int f\,dS_{B,\mathcal{M}} &=
	\frac{1}{n-1}
	\sum_{e_{F,F'}\in E}\mathcal{H}^{n-2}(F\cap F') 
	\int_{e_{F,F'}} f\,d\mathcal{H}^1.
\end{align*}
Moreover, the self-adjoint operator $\mathscr{A}$ and closed
quadratic form $\mathcal{E}$ on $L^2(S_{B,\mathcal{M}})$ defined in 
Theorem \ref{thm:forms} can be expressed in this case as follows:
$$
	\mathscr{A}f=\frac{1}{n}\{f''+f\}
$$
where
\begin{align*}
	\Dom\mathscr{A}=\Bigg\{ & f\in C^0(G):
	f|_{e_{F,F'}}\in H^2\mbox{ for all }e_{F,F'}\in E, \\
	&\sum_{F':F'\sim F} \mathcal{H}^{n-2}(F\cap F')\, 
	\nabla_{n_{F\to F'}}f(n_F)
	=0
	\mbox{ for all }F\in\mathcal{F}\Bigg\},
\end{align*}
and
$$
	\mathcal{E}(f,g) =
	\frac{1}{n(n-1)}\sum_{e_{F,F'}\in E} \mathcal{H}^{n-2}(F\cap F')
	\int_{e_{F,F'}} \{fg-f'g'\}\,d\mathcal{H}^1
$$
where $\Dom\mathcal{E}=\{f\in C^0(G):
f|_{e_{F,F'}}\in H^1\mbox{ for all }e_{F,F'}\in E\}$.
\end{thm}

In the remainder of this section, we fix the setting and notation of 
Theorem \ref{thm:qgraph}. The starting point for the proof is an integral 
formula for the mixed area measures $S_{K,\mathcal{M}}$. The cases $K=M$ 
and $K=B$ that appear in the first part of the statement of Theorem 
\ref{thm:qgraph} are classical, cf.\ \cite[eq.\ (4.24)]{Sch14}. The 
analogous representation for general convex bodies $K$ is also well known 
to experts, cf.\ \cite[eq.\ (7.175)]{Sch14} (where it is stated without 
proof) or \cite[Corollary 2.13]{Hug00} for a more general setting.
For completeness, we include a full proof for the case that will be needed here.

\begin{prop}
\label{prop:polymixa}
Let $K$ be a convex body of class $C^2_+$. Then
$$
	\int f\,dS_{K,\mathcal{M}} =
	\frac{1}{n-1}
	\sum_{e_{F,F'}\in E}\mathcal{H}^{n-2}(F\cap F')
	\int_{e_{F,F'}} (h_K''+h_K)\,f\,d\mathcal{H}^1.
$$
\end{prop}

\begin{proof}
Denote by $\mathcal{F}_i$ the collection of $i$-dimensional faces of $M$, 
and define 
$$
	N_F:=\{u\in S^{n-1}:F(M,u)=F\}.
$$
Then $\{N_F:F\in\bigcup_i\mathcal{F}_i\}$ is a 
partition of $S^{n-1}$. Note that if 
$F,F'\in\mathcal{F}=\mathcal{F}_{n-1}$ are 
neighboring facets $F\sim F'$, we have $F\cap F'\in\mathcal{F}_{n-2}$ 
and $N_{F\cap F'}=e_{F,F'}$.

To compute $S_{K,\mathcal{M}}$, we will compute the area measure 
$S(M+\varepsilon K,\cdot\,)$, and then deduce the mixed area measure 
from its definition given in section \ref{sec:mixva}. To compute the 
area measure, we partition the boundary of the set $M+\varepsilon K$ 
into disjoint sets $F(M+\varepsilon K,N_F):=\bigcup_{u\in N_F} 
F(M+\varepsilon K,u)$ corresponding to boundary points with normals in 
$N_F$, and apply the change of variables formula to each of these sets.
This partition is illustrated in Figure \ref{fig:sch} (in the figure
$M$ is a cube and $K=B$).
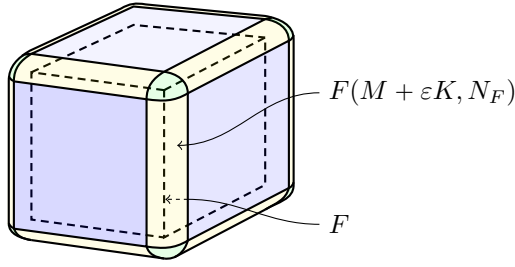
\begin{figure}
\centering
\begin{tikzpicture}[scale=.7]

\draw[thick,densely dashed] (-5.95,.55) -- (-4.03,1.55) -- (-4.03,-1.25) --
(-5.95,-2.25) -- (-5.95,.55);
\draw[thick,densely dashed] (-5.95,.55) -- (-8.47,.89) -- (-8.47,-1.91) --
(-5.95,-2.25);
\draw[thick, densely dashed] (-8.47,.89) -- (-6.53,1.79) -- (-4.03,1.55);

\fill[green,opacity=0.1] (-5.5,.49) to[out=90,in=-7.75] (-5.95,.9)
to[in=90,out=211] (-6.283,.377) to[out=-7.75,in=211] (-5.5,.49);

\fill[green,opacity=0.1] (-6.283,-2.42) to[out=-75,in=172.25] (-5.95,-2.65)
to[in=-105,out=-7.75] (-5.5,-2.31) to[in=-7.75,out=211]
(-6.283,-2.42);

\fill[yellow,opacity=0.1] (-5.5,.49) to[out=90,in=-7.75] (-5.95,.9)
to (-4.03,1.9) to[in=90,out=-7.75] (-3.58,1.49);
\draw[thick] (-5.5,.49) to[out=90,in=-7.75] (-5.95,.9); 
\draw[thick] (-3.58,1.49) to[out=90,in=-7.75] (-4.03,1.9);

\fill[yellow,opacity=0.1] (-6.283,.377) to[out=90,in=211] (-5.95,.9)
to (-8.47,1.24) to[in=90,out=211] (-8.803,.727);
\draw[thick] (-6.283,.377) to[out=90,in=211] (-5.95,.9);
\draw[thick] (-8.803,.727) to[out=90,in=211] (-8.47,1.24);

\fill[yellow,opacity=0.1] (-6.283,.377) to[out=-7.75,in=211] (-5.5,.49) to
(-5.5,-2.31) to[in=-7.75,out=211] (-6.283,-2.423);
\draw[thick] (-6.283,.377) to[out=-7.75,in=211] (-5.5,.49);
\draw[thick] (-6.283,-2.423) to[out=-7.75,in=211] (-5.5,-2.31);

\fill[yellow,opacity=0.1] (-6.283,-2.42) to[out=-75,in=172.25] (-5.95,-2.65)
to (-8.47,-2.3) to[in=-75,out=172.25] (-8.803,-2.07) to
(-6.283,-2.42);
\draw[thick] (-6.283,-2.42) to[out=-75,in=172.25] (-5.95,-2.65);
\draw[thick] (-8.803,-2.07) to[out=-75,in=172.25] (-8.47,-2.3);
\draw[thick] (-5.95,-2.65) to (-8.47,-2.3);

%THE THIN EDGE
\fill[yellow,opacity=0.1] (-8.5,1.3) to (-6.77,2.1)
to[out=31,in=170] (-6.3,2.17) to (-4.03,1.96) to (-3.9,1.93) to
(-4.03,1.9) to (-6.53,2.14) -- (-8.47,1.24) -- 
(-8.5,1.3);

\fill[green,opacity=0.1] (-8.803,.727) to[out=172.25,in=160] (-8.90,.817)
to[out=90,in=190] (-8.47,1.24) 
to[in=90,out=211] (-8.803,.727);

\fill[green,opacity=0.1] (-3.58,1.49) to[out=21,in=221] (-3.49,1.60)
to[in=-10,out=90] (-4.03,1.96) to (-4.03,1.9) to[in=90,out=-7.75] (-3.58,1.49);

\fill[green,opacity=0.1] (-3.49,-1.25) to[out=-90,in=31] (-4.03,-1.72)
to[in=-105,out=31] (-3.58,-1.31) to[out=21,in=221] (-3.49,-1.25);

\fill[yellow,opacity=0.1] (-8.803,-2.07) to[out=172.25,in=160] 
(-8.925,-1.98) to (-8.92,.817) to (-8.803,.727) to
(-8.803,-2.07);

\fill[yellow,opacity=0.1] (-3.58,1.49) to[out=21,in=221] (-3.49,1.60)
to (-3.49,-1.2) to[in=21,out=221] (-3.58,-1.31) to (-3.58,1.49);

\fill[green,opacity=0.1] (-8.803,-2.07) to[out=172.25,in=160]
(-8.925,-1.98) to[in=170,out=270] (-8.47,-2.3)
to[in=-75,out=172.25] (-8.803,-2.07);

\draw[thick] (-8.47,-2.3) to[out=170,in=270] (-8.9,-1.98) to
(-8.9,.727) to[out=90,in=211] (-8.5,1.3) to (-6.77,2.1)
to[out=31,in=170] (-6.3,2.18) to (-4.03,1.96) 
to[out=-10,in=90] (-3.49,1.52) to (-3.49,-1.25)
to[out=-90,in=31] (-4.03,-1.72);
\draw[thick] (-8.803,-2.07) to[out=172.25,in=-60] (-8.9,-1.98);
\draw[thick] (-8.803,.727) to[out=172.25,in=-60] (-8.9,.817);
\draw[thick] (-8.47,1.24) to[out=180,in=160] (-8.62,1.2);
\draw[thick] (-3.58,-1.31) to[out=21,in=251] (-3.49,-1.2);
\draw[thick] (-3.58,1.49) to[out=21,in=251] (-3.49,1.60);
\draw[thick] (-3.9,1.93) to (-4.03,1.9);
%END THIN EDGE

\fill[yellow,opacity=0.1] (-5.95,-2.65) to[in=-105,out=-7.75] (-5.5,-2.31)
to (-3.58,-1.31) to[out=-105,in=31] (-4.03,-1.72) to (-5.75,-2.625);
\draw[thick] (-5.5,-2.31) to[out=-105,in=-7.75] (-5.95,-2.65);
\draw[thick] (-3.58,-1.31) to[out=-105,in=31] (-4.03,-1.72);
\draw[thick] (-4.03,-1.72) to (-5.75,-2.625);

\begin{scope}[xshift=.45cm,yshift=-.06cm]
	\fill[blue,opacity=0.1] (-5.95,.55) -- (-4.03,1.55) -- 
	(-4.03,-1.25) -- (-5.95,-2.25) -- (-5.95,.55);
	\draw[thick] (-5.95,.55) -- (-4.03,1.55) -- (-4.03,-1.25) --
	(-5.95,-2.25) -- (-5.95,.55);
\end{scope}

\begin{scope}[xshift=-.333cm,yshift=-0.173cm]
	\fill[blue,opacity=0.15] (-5.95,.55) -- (-8.47,.89) -- 
	(-8.47,-1.91) -- (-5.95,-2.25) -- (-5.95,.55);
	\draw[thick] (-5.95,.55) -- (-8.47,.89) -- (-8.47,-1.91) --
	(-5.95,-2.25) -- (-5.95,.55);
\end{scope}

\begin{scope}[yshift=.35cm]
	\fill[blue,opacity=0.05] (-8.47,.89) -- (-6.53,1.79) -- 
	(-4.03,1.55) -- (-5.95,.55) -- (-8.47,.89);
	\draw[thick] (-8.47,.89) -- (-6.53,1.79) -- (-4.03,1.55) --
	(-5.95,.55) -- (-8.47,.89);
\end{scope}

\draw (-3,-2) to[out=180,in=0] (-5.5,-1.525);
\draw[->,dash pattern=on 1.5pt off 1.5pt] (-5.5,-1.525) to (-5.925,-1.525);
\draw (-3,-2) node[right] {$F$};

\draw[->] (-3,.5) to[out=180,in=0] (-5.7,-.5);
\draw (-3,.5) node[right] {$F(M+\varepsilon K,N_F)$};

\end{tikzpicture}
\caption{Decomposition of the boundary of $M+\varepsilon K$.\label{fig:sch}}
\end{figure}

\textbf{Step 1}.
Fix $F\in\mathcal{F}_i$. We first conveniently parametrize
$F(M+\varepsilon K,N_F)$.

Recall that as $K$ is a 
$C^2_+$-body, the outer normal map $n_K:\partial K\to S^{n-1}$ is a
$C^1$-diffeomorphism and $n_K^{-1}=\nabla h_K$, cf.\ section \ref{sec:smooth}.
Therefore
$$
	F(M+\varepsilon K,u) =   
	F(M,u)+\varepsilon F(K,u) =
	F+\varepsilon \nabla h_K(u)
$$
for any $u\in N_F$ by Lemma \ref{lem:dirdir}. It follows that
$F(M+\varepsilon K,N_F)$ is the image of the map
$\iota:F\times N_F\to F(M+\varepsilon K,N_F)$ defined by 
$
	\iota(x,u):=x+\varepsilon  \nabla h_K(u)
$.

\textbf{Step 2}.
We now show that $\iota$ is a $C^1$-diffeomorphism.

Let $L_F:=\spn(F-F)$ be the tangent space of $F$, and note that 
$N_F\subset S^{n-1}\cap L_F^\perp$.
Assume for simplicity that $0\in F$, so that $F\subset L_F$ (when 
considering a single face, we may always reduce to this setting by 
translation). Then we may express 
$$
	\iota(x,u)=(x+\varepsilon 
	P_{L_F}\nabla h_K(u),\varepsilon P_{L_F^\perp}\nabla h_K(u))\in
	L_F\oplus L_F^\perp,
$$
where $P_L$ denotes orthogonal projection onto $L$.

Now note that as $h_{P_{L_F^\perp}K}(u)=h_K(P_{L_F^\perp}u)$, 
differentiating yields that 
$$
	P_{L_F^\perp}\nabla h_K(u)=\nabla h_{P_{L_F^\perp}K}(u)
	\quad\mbox{for }u\in S^{n-1}\cap L_F^\perp.
$$
As $P_{L_F^\perp}K$ is a $C^2_+$ body in $L_F^\perp$,  
$n_{P_{L_F^\perp}K}:\partial P_{L_F^\perp}K\to 
S^{n-1}\cap L_F^\perp$ is a $C^1$-diffeomorphism and
$n_{P_{L_F^\perp}K}^{-1}=P_{L_F^\perp}\nabla h_K$. It now follows 
that $\iota$ is a diffeomorphism, as $\iota$ is $C^1$ and
$$
	\iota^{-1}(z,v) = 
	(z-\varepsilon P_{L_F}\nabla h_K(n_{P_{L_F^\perp}K}(\varepsilon^{-1}v)),
	n_{P_{L_F^\perp}K}(\varepsilon^{-1}v))
$$
for $(z,v)\in F(M+\varepsilon K,N_F)\subset L_F\oplus L_F^\perp$ is also
$C^1$.

It is readily seen that the differential $d\iota$ has a block-triangular
form with respect to $L_F\oplus L_F^\perp$, and that its determinant
may be written as $\det(\varepsilon D^2h_{P_{L_F^\perp}K})$
where $D^2h_{P_{L_F^\perp}K}$ is computed in $L_F^\perp$
(equivalently, $D^2h_{P_{L_F^\perp}K}$ is the projection of the Hessian
$\nabla^2h_K$ in $\mathbb{R}^n$ on the tangent space of 
$S^{n-1}\cap L_F^\perp$).

\textbf{Step 3.} Now recall that $\{N_F\}$ partitions $S^{n-1}$. We
can therefore write
\begin{align*}
	S(M+\varepsilon K,A) &=
	\sum_{i=0}^{n-1}
        \sum_{F\in\mathcal{F}_i}
	\mathcal{H}^{n-1}(\{x\in F(M+\varepsilon K,u)\mbox{ for some }u\in 
	A\cap N_F
	\})
	\\ &=
	\sum_{i=0}^{n-1}
	\varepsilon^{n-1-i}
	\sum_{F\in\mathcal{F}_i}
	\mathcal{H}^i(F)
	\int_{N_F} 1_A\det(D^2h_{P_{L_F^\perp}K})\,d\mathcal{H}^{n-1-i},
\end{align*}
where we used the map $\iota$ to perform a change of variables in the
second line.
By the definition of mixed area measures in section
\ref{sec:mixva}, we have
\begin{align*}
	(n-1)\,S_{K,\mathcal{M}}(A) &= \frac{d}{d\varepsilon}S(M+\varepsilon 
	K,A)\bigg|_{\varepsilon=0^+} \\
	&=
        \sum_{F\in\mathcal{F}_{n-2}}
        \mathcal{H}^{n-2}(F)
        \int_{N_F} 1_A\det(D^2h_{P_{L_F^\perp}K})\,d\mathcal{H}^{1}.
\end{align*}
But note that for every $(n-2)$-face $F$, the set $N_F$ is an edge of the 
quantum graph. In this case $D^2h_{P_{L_F^\perp}K}$ is a scalar function
(as $S^{n-1}\cap L_F^\perp$ is one-dimensional) and coincides precisely
with $h_K''+h_K$ on each edge. Thus the proof is complete.
\end{proof}

We also record a simple but essential observation. It immediately implies, 
for example, that $C^2(S^{n-1})\subset \Dom\mathscr{A}$ in Theorem
\ref{thm:qgraph}.

\begin{lem}
\label{lem:facecond}
For every $F\in\mathcal{F}$, we have
$$
	\sum_{F':F'\sim F} \mathcal{H}^{n-2}(F\cap F')\,
	n_{F\to F'} = 0.
$$
\end{lem}

\begin{proof}
Consider $F$ as an $(n-1)$-dimensional convex body in 
$\mathop{\mathrm{aff}}F$. The facets of $F$ are precisely the sets
$F\cap F'$ for $F'\sim F$, and the corresponding facet normals are
$n_{F\to F'}$. Thus the conclusion follows from Lemma 
\ref{lem:maprop}(\textit{e}).
\end{proof}

We now proceed to the proof of Theorem \ref{thm:qgraph}.

\begin{proof}[Proof of Theorem \ref{thm:qgraph}]
The expressions for $S_{M,\mathcal{M}}$ and $S_{B,\mathcal{M}}$ follow
immediately from the definition of mixed area measures (cf.\ section 
\ref{sec:mixva}) and Proposition \ref{prop:polymixa}.

In the remainder of the proof, we define the operator $\mathscr{A}$ 
as in the statement of Theorem \ref{thm:qgraph}. Our aim is to show 
that the operator thus defined coincides with the general construction of 
Theorem \ref{thm:forms} in the present setting.
To this end, consider first $f\in C^2(S^{n-1})$. By Lemma 
\ref{lem:facecond}, we have $f\in\Dom\mathscr{A}$. Moreover,
$$
	\mathscr{A}f := \frac{1}{n}\{f''+f\} =
	\frac{1}{n}\frac{dS_{f,\mathcal{M}}}{dS_{B,\mathcal{M}}}
$$
by Proposition \ref{prop:polymixa}. Thus the restriction of $\mathscr{A}$ 
to $C^2(S^{n-1})\subset\Dom\mathscr{A}$ agrees with the operator defined 
in \eqref{eq:defa}. It remains to show that $\mathscr{A}$ as defined in 
Theorem \ref{thm:qgraph} is self-adjoint, that $\mathcal{E}$ in Theorem 
\ref{thm:qgraph} is the associated closed quadratic form, and that 
$\mathscr{A}$ agrees with the Friedrichs extension of its restriction to 
$C^2(S^{n-1})$.

Consider any $f,g\in\Dom\mathscr{A}$. Then
$$
	\langle g,\mathscr{A}f\rangle_{L^2(S_{B,\mathcal{M}})} =
	\frac{1}{n(n-1)}
	\sum_{e_{F,F'}\in E}\mathcal{H}^{n-2}(F\cap F')
	\int_{e_{F,F'}} (f''+f)\,g\,d\mathcal{H}^1
$$
by definition. Integrating by parts yields
\begin{align*}
	&\int_{e_{F,F'}} (f''+f)\,g\,d\mathcal{H}^1 = 
	\\ &\quad \int_{e_{F,F'}} \{fg-f'g'\}\,d\mathcal{H}^1 
	- \nabla_{n_{F\to F'}}f(n_F)g(n_F)
	- \nabla_{n_{F'\to F}}f(n_{F'})g(n_{F'}).
\end{align*}
Multiplying this expression by $\mathcal{H}^{n-2}(F\cap F')$
and summing over $e_{F,F'}\in E$ yields
$$
	\langle g,\mathscr{A}f\rangle_{L^2(S_{B,\mathcal{M}})} =
	\frac{1}{n(n-1)}\sum_{e_{F,F'}\in E} \mathcal{H}^{n-2}(F\cap F')
	\int_{e_{F,F'}} \{fg-f'g'\}\,d\mathcal{H}^1,
$$
where the boundary terms vanish by the definition of $\Dom\mathscr{A}$.
We have therefore shown that $\mathcal{E}(f,g)=\langle 
g,\mathscr{A}f\rangle_{L^2(S_{B,\mathcal{M}})}$ for all 
$f,g\in\Dom\mathscr{A}$. It is now a standard exercise to prove that 
$\mathscr{A}$ is self-adjoint and that $\mathcal{E}$ is the associated 
closed quadratic form with the given domains; see, e.g., \cite[Theorems 
1.4.4 and 1.4.11]{BK13}.

Finally, we claim that $\mathscr{A}$ is the Friedrichs extension of its 
restriction to $C^2(S^{n-1})$. Indeed, by applying Lemma 
\ref{lem:edgeapprox} below to every edge, we can write
any $f\in\Dom\mathcal{E}$ as the limit with respect to the norm 
$[2\|f\|_{L^2(S_{B,\mathcal{M}})}^2-\mathcal{E}(f,f)]^{1/2}$ of functions 
$f_k\in C^0(G)$ that are $C^\infty$ on each edge and constant in a 
neighborhood of each vertex. By a 
classical extension argument \cite[Lemma 2.26]{Lee13}, any such function 
may be extended to a $C^\infty$ function on the entire sphere $S^{n-1}$. 
We have therefore shown that $\mathcal{E}$ is the closure of its 
restriction to $C^2(S^{n-1})$, completing the proof.
\end{proof}

Above we used the following simple approximation argument. We 
recall that $H^1([0,l])\subset C^0([0,l])$ by the Sobolev embedding 
theorem \cite[Theorem 7.26]{GT01}, so that the value of a 
function $h\in H^1$ is well defined at every point.

\begin{lem}
\label{lem:edgeapprox}
Let $h:[0,l]\to\mathbb{R}$ satisfy $h\in H^1$. Then there exists a 
sequence of functions $h_k:[0,l]\to\mathbb{R}$ such that for all 
sufficiently large $k$ the following hold:
\begin{enumerate}[a.]
\item $h_k\in C^\infty$.
\item $h_k(x)=h(0)$ on $x\in[0,a_k]$ and
$h_k(x)=h(l)$ on $x\in[l-a_k,l]$ 
for some $a_k>0$.
\item $h_k\to h$ and $h_k'\to h'$ in $L^2$
as $k\to\infty$.
\end{enumerate}
\end{lem}

\begin{proof}
The argument is classical; see, e.g., \cite[Theorem 7.2.1]{Dav95}. By 
convolving with a smooth mollifier, we first construct $C^\infty$ 
functions $u_k$ such that $u_k\to h$ and $u_k'\to h'$ in 
$L^2$ \cite[section 7.6]{GT01}. By the Sobolev embedding theorem, 
$u_k\to h$ uniformly. We may therefore assume without loss of generality 
that $u_k(0)=h(0)$ and $u_k(l)=h(l)$ for every $k$; otherwise 
this may be accomplished by adding a linear function to $u_k$ without 
changing the above approximation properties.

Now define for every $m\ge 4/l$ a function $\iota_m:[0,l]\to[0,l]$ with 
the following properties: $\iota_m$ is $C^\infty$ and $0\le\iota_m'\le 3$; 
$\iota_m(x)=x$ for $x\in[\frac{2}{m},l-\frac{2}{m}]$; $\iota_m(x)=0$ for 
$x\in[0,\frac{1}{m}]$; and $\iota_m(x)=l$ for $x\in[l-\frac{1}{m},l]$. 
Then we have $u_k\circ\iota_m\to u_k$ and 
$(u_k\circ\iota_m)'=(u_k'\circ\iota_m)\iota_m'\to u_k'$ in $L^2$ as 
$m\to\infty$ for any $k$. Moreover, $u_k\circ\iota_m$ clearly satisfies 
properties a.\ and b.\ of the statement of the lemma. We may therefore 
construct $\{h_k\}$ by extracting a diagonal subsequence of 
$\{u_k\circ\iota_m\}_{k,m\ge 1}$.
\end{proof}

We conclude this section with two remarks.

\begin{rem}[General mixed volumes]
\label{rem:genqgraph}
For simplicity, we have restricted attention in this section to the
special case of Theorem \ref{thm:forms} where $C_1=\cdots=C_{n-2}=M$.
This is the only setting that will be needed in the sequel. However,
Theorem \ref{thm:qgraph} is readily extended to the setting
where $C_1,\ldots,C_{n-2}$ are arbitrary polytopes. Let us briefly sketch
the relevant constructions, leaving the details to the reader.

Let $C_1,\ldots,C_{n-2}$ be polytopes. Then the polytopes 
$M_\lambda:=\lambda_1C_1+\cdots+\lambda_{n-2}C_{n-2}$ are strongly 
isomorphic for all choices of $\lambda_1,\ldots,\lambda_{n-2}>0$ 
\cite[Corollary 2.4.12]{Sch14}; in particular, all $M_\lambda$ induce the 
same metric graph $G=(V,E)$, whose vertices are the facet normals of 
$M_1:=C_1+\cdots+C_{n-2}$ and whose edges are indexed by 
the $(n-2)$-faces of $M_1$.
Now let $e\in E$, and let $\tilde F^e$ be the associated $(n-2)$-face
of $M_1$. By Lemma \ref{lem:dirdir}, there exists for each $i$ a face
$\tilde F_i^e$ of $C_i$ such that $\tilde F^e=\tilde F_1^e+\cdots+\tilde 
F_{n-2}^e$. It follows from Proposition \ref{prop:polymixa} and the
definitions in section \ref{sec:mixva} that
$$
	\int f\,dS_{K,\mathcal{C}} =
	\frac{1}{n-1}
	\sum_{e\in E}\V(\tilde F_1^e,\ldots,\tilde F_{n-2}^e)
	\int_{e} (h_K''+h_K)\,f\,d\mathcal{H}^1
$$
for every convex body $K$ of class $C^2_+$, where the mixed volume that
appears here
is computed in the $(n-2)$-dimensional subspace spanned by $\tilde F^e$
(modulo translation of the faces $\tilde F_i^e$).
Using this expression we may 
readily adapt Theorem \ref{thm:qgraph} to the present setting with the
same proof. The underlying graph is now the metric graph associated
to $M_1$, and the weights $\mathcal{H}^{n-2}(F\cap F')$ are replaced 
everywhere by the mixed volumes $\V(\tilde F_1^e,\ldots,\tilde F_{n-2}^e)$.
Let us note that some of these weights may vanish in the general case,
unlike in the more restricted setting of Theorem \ref{thm:qgraph}.
\end{rem}

\begin{rem}[The kernel of a quantum graph]
Lemma \ref{lem:hyper} states that the study of equality cases in 
the Minkowski or Alexandrov-Fenchel inequalities is equivalent to 
understanding $\ker\mathscr{A}$. 
In the setting of Theorem \ref{thm:qgraph}, we can attempt to compute this 
kernel explicitly. Indeed, if $\mathscr{A}f=0$, then we must evidently 
have $f|_e(\theta)=a_e\cos(\theta)+b_e\sin(\theta)$ on every edge $e\in 
E$. Thus we only need to compute the vector of coefficients 
$(a_e,b_e)_{e\in E}$. By using the vertex boundary conditions of 
$\Dom\mathscr{A}$, it can be shown that $f\in\ker\mathscr{A}$ if and only 
if the vector of coefficients is in the kernel of a certain matrix, cf.\ 
\cite[section 3.6]{BK13}. In the present setting, this matrix turns out
to be precisely what appears in Alexandrov's polytope proof of the 
Alexandrov-Fenchel inequality \cite{Ale96}. This is essentially the idea 
behind Schneider's proof of the equality cases of the Alexandrov-Fenchel
inequality in the case that $C_1,\ldots,C_{n-2}$ are strongly isomorphic 
simple polytopes \cite[Theorem 7.6.21]{Sch14}.

The polytope setting is rather special, however, and such direct 
computations cannot be performed for general bodies. One may view the 
methods that will be developed in subsequent sections as a kind of 
quantitative replacement for this argument, that enables us to pass to the 
limiting case of arbitrary convex bodies.
\end{rem}

\subsection{Dirichlet forms}
\label{sec:diri}

Our aim in this section is to provide some additional insight into the 
structure of the operator defined in a somewhat abstract manner by Theorem 
\ref{thm:forms}. The results of this section will not be used elsewhere in 
this paper, but are included in order to clarify our general 
constructions.

In two special cases, when $C_1,\ldots,C_{n-2}$ are polyopes and when they 
are bodies of class $C^\infty_+$, we have seen that $\mathscr{A}$ can be 
expressed explicitly as a second-order differential operator. Another 
explicit setting will be encountered in section \ref{sec:lower} below. In 
the general setting, however, it is not so clear what such an operator may 
look like: it is perfectly possible in general for a limit of second-order 
differential operators to exhibit quite different behavior \cite{Mos94}. 
The study of such operators is enabled in a very general setting by the 
theory of Dirichlet forms \cite{FOT11,BGL14}. In the present setting, 
however, we can give a more concrete representation that avoids the 
abstract theory. The main result of this section is the following.

\begin{thm}
\label{thm:diri}
Let $C_1,\ldots,C_{n-2}$ be arbitrary convex bodies in $\mathbb{R}^n$ with 
$S_{B,\mathcal{C}}\ne 0$, and let $\mathscr{A}$ and $\mathcal{E}$
be the operator and quadratic form on $L^2(S_{B,\mathcal{C}})$ of Theorem 
\ref{thm:forms}. Then there exists a measurable function
$A:S^{n-1}\to\mathrm{Sym}_n(\mathbb{R})$ such that\footnote{
	As in the definition of $D^2f$ in section \ref{sec:smooth},
	we view $A(x)$ as a matrix acting on the tangent space $x^\perp$ 
	at each $x\in S^{n-1}$ (that is, $A$ is a measurable 
	$(1,1)$-tensor field on $S^{n-1}$).
}
$A(x)x=0$, $A(x)\ge 0$, and $\Tr[A(x)]=1$ for all $x\in S^{n-1}$, and such 
that
$$
	\mathscr{A}f = \frac{1}{n}\Tr[A\,D^2f],\qquad
	\mathcal{E}(f,g) = \frac{1}{n}\int  \{fg-
	\langle\nabla_{S^{n-1}}f,A
	\,\nabla_{S^{n-1}} g\rangle\}\,
	dS_{B,\mathcal{C}}
$$
for every $f,g\in C^2(S^{n-1})$, where
$\nabla_{S^{n-1}}$ is the covariant derivative on $S^{n-1}$.
\end{thm}

From a practical viewpoint, the problem with Theorem \ref{thm:diri} is 
that we do not have a useful explicit expression for the tensor field $A$ 
when $C_1,\ldots,C_{n-2}$ are arbitrary bodies. For this reason, this 
result will not be used in the rest of the paper. Nonetheless, Theorem 
\ref{thm:diri} shows that the objects in Theorem \ref{thm:forms} may be 
viewed rather concretely as highly degenerate elliptic second-order 
differential operators.

\begin{rem}
As we are working on $L^2(S_{B,\mathcal{C}})$, the tensor field $A$ need
only be defined $S_{B,\mathcal{C}}$-a.e. It may be viewed as a 
kind of Riemannian metric on $\supp S_{B,\mathcal{C}}$, cf.\ \cite{Hin13}. 
For example, for polytopes it follows from Remark \ref{rem:genqgraph} that 
$\rank A=1$ a.e., which reflects the fact that the underlying metric graph 
is one-dimensional.
\end{rem}

The main observation behind the proof of Theorem \ref{thm:diri} is the 
following.

\begin{lem}
\label{lem:chainr}
Define the functions $\ell_i:S^{n-1}\to\mathbb{R}$, $i=1,\ldots,n$
as $\ell_i(x):= x_i$. Then the following are equivalent.
\begin{enumerate}[1.]
\item There exists a measurable function 
$A:S^{n-1}\to\mathrm{Sym}_n(\mathbb{R})$
such that $A(x)x=0$ and $\Tr[A(x)]=1$ for all $x\in S^{n-1}$, and
$$
	\mathscr{A}f = \frac{1}{n}\Tr[A\,D^2f]\quad\mbox{for all }
	f\in C^2(S^{n-1}).
$$
\item We have
$$
	dS_{f,\mathcal{C}} =
	\frac{1}{2}\sum_{i,j=1}^n
	\frac{\partial^2\bar f}{\partial x_i\partial x_j}
	\,dS_{\ell_i\ell_j,\mathcal{C}}
	\quad\mbox{for all }f\in C^2(S^{n-1}),
$$
where $\bar f$ is the $1$-homogeneous extension of $f$ to $\mathbb{R}^n$.
\end{enumerate}
\end{lem}

\begin{proof}
To prove $2\Rightarrow 1$, recall that as $\bar f$ is $1$-homogeneous,
its Hessian in $\mathbb{R}^n$ satisfies $\nabla^2\bar f(x)x=0$ (cf.\ 
section \ref{sec:smooth}). If we therefore define $\bar A(x),A(x)\in
\mathrm{Sym}_n(\mathbb{R})$ as
$$
	\bar A_{ij} := \frac{1}{2}\frac{dS_{\ell_i\ell_j,\mathcal{C}}}
	{dS_{B,\mathcal{C}}},\qquad
	A(x):=P_{x^\perp}\bar A(x)P_{x^\perp}
$$
using Lemma \ref{lem:density}, where
$P_{x^\perp}$ denotes the orthogonal projection on $x^\perp$
in $\mathbb{R}^n$, then
$$
	\mathscr{A}f = 
	\frac{1}{n}\frac{dS_{f,\mathcal{C}}}{dS_{B,\mathcal{C}}}
	= \frac{1}{n}\Tr[A\,D^2f]
$$
by \eqref{eq:defa}. That $\Tr[A(x)]=1$ follows by choosing $f=1$ and
using that $\mathscr{A}1=\frac{1}{n}1$ by Theorem \ref{thm:forms}.
Thus the claim is established.

To prove the converse implication $1\Rightarrow 2$, note first that 
$f=\bar f(\ell_1,\ldots,\ell_n)$. Substituting this into the expression
for $\mathscr{A}f$ and using the chain rule gives
\begin{equation}
\label{eq:chainr}
	n\,\mathscr{A}f = f +
	\sum_{i=1}^n \frac{\partial\bar f}{\partial x_i}\,
	\Tr[A\,\nabla_{S^{n-1}}^2\ell_i]
	+
	\sum_{i,j=1}^n \frac{\partial^2\bar f}{\partial x_i\partial x_j}
	\,\langle\nabla_{S^{n-1}}\ell_i,A\,\nabla_{S^{n-1}}\ell_j\rangle,
\end{equation}
where we used
$D^2f=\nabla_{S^{n-1}}^2f+fI$ and $\Tr[A(x)]=1$. As $\ell_i$ is a 
linear function, we obtain $\Tr[A\,\nabla_{S^{n-1}}^2\ell_i]=
n\mathscr{A}\ell_i -\ell_i=-\ell_i$ by Theorem \ref{thm:forms}. But
note that
$$
	\sum_{i=1}^n \ell_i(x) \frac{\partial\bar f}{\partial x_i}(x)=
	\langle x,\nabla\bar f(x)\rangle = \bar f(x)
$$
by $1$-homogeneity. Thus
the first two terms on the right-hand side of \eqref{eq:chainr} cancel.
To simplify the last term, we use again the chain rule to write
\begin{align*}
	2\langle\nabla_{S^{n-1}}\ell_i,A\,\nabla_{S^{n-1}}\ell_j\rangle
	&= 
	\Tr[A\,\nabla_{S^{n-1}}^2(\ell_i\ell_j)] 
	- \ell_i \Tr[A\,\nabla_{S^{n-1}}^2\ell_j] 
	- \ell_j \Tr[A\,\nabla_{S^{n-1}}^2\ell_i]
	\\ &=
	n\,\mathscr{A}(\ell_i\ell_j)
	+ \ell_i\ell_j.
\end{align*}
But as
$$
	\sum_{i,j=1}^n \ell_i(x)\ell_j(x)\frac{\partial^2\bar f}{\partial 
	x_i\partial x_j}(x) = \langle x,\nabla^2\bar f(x)x\rangle = 0
$$
by $1$-homogeneity, we have shown that
$$
	\mathscr{A}f =
	\frac{1}{2} \sum_{i,j=1}^n \frac{\partial^2\bar f}{\partial x_i\partial x_j}
        \,\mathscr{A}(\ell_i\ell_j).
$$
The proof is concluded by invoking again \eqref{eq:defa}.
\end{proof}

\begin{proof}[Proof of Theorem \ref{thm:diri}]
Suppose first that $C_1,\ldots,C_{n-2}$ are convex bodies of class 
$C^\infty_+$. Then we obtain as in the proof of Lemma \ref{lem:density}
$$
	\mathscr{A}f = \frac{1}{n}\frac{
	\D(D^2f,D^2h_{C_1},\ldots,D^2h_{C_{n-2}})}{
	\D(I,D^2h_{C_1},\ldots,D^2h_{C_{n-2}})}
	\quad\mbox{for }f\in C^2(S^{n-1}).
$$
By linearity of mixed discriminants, this expression may be
written as
$$
	\mathscr{A}f = \frac{1}{n}\Tr[A\,D^2f]
$$
for a continuous function $A:S^{n-1}\to\mathrm{Sym}_n(\mathbb{R})$. As 
$D^2f$ only acts on $x^\perp$, we may choose $A(x)x=0$ without loss of 
generality. Moreover, that $A(x)\ge 0$ follows from Lemma 
\ref{lem:mdprop}(\textit{d}).
Finally, that $\Tr[A(x)]=1$ follows by choosing $f=1$ (so that $D^2f=I$)
in the above expressions for $\mathscr{A}$.

Now let $C_1,\ldots,C_{n-2}$ be arbitrary convex bodies. Note that
by Theorem \ref{thm:conv}, the second condition of Lemma \ref{lem:chainr}
is preserved by taking Hausdorff limits. As we have established above 
the first condition of Lemma \ref{lem:chainr} for bodies of class
$C^\infty_+$, we may extend this conclusion to arbitrary bodies by 
approximation as in the proof of Lemma \ref{lem:asemib}. We have therefore
shown that
$$
	\mathscr{A}f = \frac{1}{n}\Tr[A\,D^2f]
$$
for some $A:S^{n-1}\to\mathrm{Sym}_n(\mathbb{R})$ such that $A(x)x=0$
and $\Tr[A(x)]=1$.

We now show that $A(x)\ge 0$, that is, that $\mathscr{A}$ is
(semi)elliptic. To this end, define the quadratic function 
$q_v(x):=\langle v,x\rangle^2$ for $v\in\mathbb{R}^n$. Then
$$
	\nabla_{S^{n-1}}^2q_v(x) = 
	2P_{x^\perp}v\,(P_{x^\perp}v)^*-2q_v(x)I,
$$
where $P_{x^\perp}$ is the orthogonal projection on $x^\perp$ and
$z^*$ denotes the transpose of the column vector $z$. It follows that
$$
	\mathscr{A}q_v + \frac{1}{n}q_v = 
	\frac{1}{n}\Tr[A\,D^2q_v] + \frac{1}{n}q_v =
	\frac{2}{n}\langle v,Av\rangle,
$$
where we used $A(x)x=0$ and $\Tr[A(x)]=1$. Thus the condition
$A\ge 0$ is equivalent to the statement that $\mathscr{A}q_v + 
\frac{1}{n}q_v\ge 0$ pointwise for every 
$v\in\mathbb{R}^n$, or, equivalently by \eqref{eq:defa}, that
$d\mu_v := dS_{q_v,\mathcal{C}} + q_v dS_{B,\mathcal{C}}$
is a (nonnegative) measure for every $v\in\mathbb{R}^n$. But we have 
already established ellipticity in the case that $C_1,\ldots,C_{n-2}$ are 
of class $C^\infty_+$, and the nonnegativity of $\mu_v$ is preserved by 
Hausdorff convergence due to Theorem \ref{thm:conv}. We can therefore 
conclude that $A\ge 0$ for arbitrary convex bodies
by approximation as in the proof of Lemma \ref{lem:asemib}.

It remains to compute the quadratic form associated to $\mathscr{A}$. Note 
that
$$
	\mathcal{E}(f,g) =
	\langle f,\mathscr{A}g\rangle_{L^2(S_{B,\mathcal{C}})}
	=
	\frac{1}{2}\int \bigg\{f\mathscr{A}g +
	g\mathscr{A}f - \mathscr{A}(fg) + \frac{1}{n}fg
	\bigg\}\,dS_{B,\mathcal{C}}
$$
for $f,g\in C^2(S^{n-1})$, where we have used that $\mathscr{A}$ is 
self-adjoint and $\mathscr{A}1=\frac{1}{n}1$. But 
$$
	\mathscr{A}(fg) =
	\frac{1}{n}\{\Tr[A\,\nabla_{S^{n-1}}^2(fg)] + fg\}
	=
	\frac{2}{n}\langle\nabla_{S^{n-1}}f,A\,\nabla_{S^{n-1}}g\rangle
	+ f\,\mathscr{A}g
	+ g\,\mathscr{A}f
	- \frac{1}{n}fg
$$
by the product rule. We have therefore shown that
$$
	\mathcal{E}(f,g) =
	\frac{1}{n}\int \{
	fg - \langle\nabla_{S^{n-1}}f,A\,\nabla_{S^{n-1}}g\rangle
	\}\,dS_{B,\mathcal{C}}
$$
for $f,g\in C^2(S^{n-1})$. The proof is complete.
\end{proof}

\section{A weak stability theorem}
\label{sec:weak}

The aim of this section is to prove a weak stability result that will be 
used in section \ref{sec:rigid} below as input to the main part of the 
proof of Theorem \ref{thm:mainfull}. The main result of this section is 
the following quantitative form of Theorem \ref{thm:weakintro}; the latter 
follows immediately by combining the following theorem with Theorem 
\ref{thm:supp}.

\begin{thm}
\label{thm:weak}
Let $M$ be a convex body in $\mathbb{R}^n$ with $0\in\intr M$.
Then there is a constant $C_M>0$, depending only on $M$ and on 
the dimension $n$, so that
\begin{align*}
	\V(K,L,\mathcal{M})^2\ge \mbox{} &
	\V(K,K,\mathcal{M})\,\V(L,L,\mathcal{M}) + \mbox{} \\ & \mbox{}
	C_M\V(L,L,\mathcal{M})\inf_{v\in\mathbb{R}^n,a\ge 0}
	\int (h_K-ah_L-\langle v,\cdot\,\rangle)^2
	\,\frac{dS_{M,\mathcal{M}}}{h_M}
\end{align*}
for all convex bodies $K,L$ in $\mathbb{R}^n$.
\end{thm}

Theorem \ref{thm:weak} is different in spirit than most of the theory 
developed in this paper, in that it is $S_{M,\mathcal{M}}$ rather than 
$S_{B,\mathcal{M}}$ that appears here as the reference measure. As $\supp 
S_{M,\mathcal{M}}$ is much smaller than $\supp S_{B,\mathcal{M}}$, only 
weak information on the extremals may be extracted from this result (cf.\ 
Example \ref{ex:weak}). Nonetheless, this weak information provides 
crucial input to the quantitative rigidity analysis of the following 
section, and we therefore develop it first.

The idea behind the proof of Theorem \ref{thm:weak} is as follows. Lemma 
\ref{lem:hyper} shows that Minkowski's quadratic inequality follows if one 
can show that an associated self-adjoint operator has a one-dimensional 
positive eigenspace. One may readily modify the proof of this fact to show 
that if, in addition, the zero eigenvalue of the operator is separated 
from the rest of the spectrum by a positive gap, then one obtains a 
quantitative improvement along the lines of Theorem \ref{thm:weak}.

It was recently observed by Kolesnikov and Milman \cite{KM17} in their 
study of local $L^p$-Brunn-Minkowski inequalities that, in the case where 
all bodies involved are smooth and symmetric, such a spectral separation 
may be established by means of a differential-geometric technique (albeit 
when the operator is normalized differently than in Theorem 
\ref{thm:forms}, which is responsible for the presence of 
$S_{M,\mathcal{M}}$ rather than $S_{B,\mathcal{M}}$ in Theorem 
\ref{thm:weak}). The approach of \cite{KM17} may also be used to obtain 
a quantitative result for non-symmetric bodies, as we will show in section 
\ref{sec:extri}. In this case, however, the resulting bound no longer has 
a direct spectral interpretation. Nonetheless, we will show in section 
\ref{sec:pfweakpf} that the requisite spectral property can be recovered, 
for a suitably normalized operator, using the min-max principle. Theorem 
\ref{thm:weak} then follows for smooth bodies by reasoning as in 
the proof Lemma \ref{lem:hyper}, and the proof is concluded by smooth 
approximation. The background from Riemannian geometry that is needed in 
this section may be found, e.g., in \cite{CLN06}.

\begin{rem}
By exploiting the linear equivariance of the quantities that appear in 
Theorem \ref{thm:weak}, it may be shown that the constant $C_M$ can in 
fact be chosen to depend on the dimension $n$ only. This property, which 
is important in the theory of \cite{KM17}, is not relevant in our 
setting. We therefore do not repeat the arguments leading to this 
observation, and refer the interested reader to \cite{KM17}.
\end{rem}

\subsection{An extrinsic formulation}
\label{sec:extri}

Behind the proof of Theorem \ref{thm:weak} lies a different perspective 
on mixed volumes than we have encountered so far. By definition we have 
$\V(K,K,\mathcal{M})=\frac{1}{n(n-1)}\frac{d^2}{dt^2}\Vol(M+tK)\big|_{t=0^+}$, 
so that mixed volumes of this kind may be viewed as arising from the 
second variation of the volume of the convex body $M$. First and second 
variation formulae play a classical role in Riemannian geometry, for 
example, in the theory of minimal surfaces \cite{CLN06,CM11}. In this 
setting, however, the variation formulae are not expressed on the sphere 
as in Lemma \ref{lem:repc2}, but rather in terms of the extrinsic geometry 
of $\partial M$ viewed as a hypersurface in $\mathbb{R}^n$. These two 
viewpoints are related by using the outer normal map $n_M$ as a change of 
variables, as we did in section \ref{sec:smooth}. The change of 
perspective is useful, however, as it enables us to exploit classical 
techniques from Riemannian geometry.

\begin{rem}
\label{rem:whyjustmink}
This is the main point in this paper where we rely specifically on the 
restricted setting of Minkowski's quadratic inequality, as opposed to the 
general Alexandrov-Fenchel inequality: the mixed volumes that appear in 
\eqref{eq:minkq} are precisely those that arise as second variations of 
the volume of $M$. For general mixed volumes, the body $M$ no longer plays 
any distinguished role, and it seems unlikely that the methods of this 
section could be useful in this context. In contrast, the techniques of 
sections \ref{sec:forms}, \ref{sec:rigid}, and \ref{sec:lower} do not 
appear to be fundamentally tied to the special setting of 
\eqref{eq:minkq}, and could potentially be adapted to a much more general 
context.
\end{rem}

Throughout this section, we will work with convex bodies $K,L,M$ in 
$\mathbb{R}^n$ of class $C^\infty_+$. We will also assume that $0\in\intr 
M$, so that $h_M>0$. We denote by $\II:=\nabla n_M$ the second fundamental 
form of $\partial M$ (viewed, as usual, as a symmetric linear map 
$\II(x):T_x\partial M\to T_x\partial M$). As in previous sections, the 
symbols $\nabla$, $\nabla_{S^{n-1}}$, and $\nabla_{\partial M}$ denote 
covariant differentiation in $\mathbb{R}^n$, $S^{n-1}$, and $\partial M$, 
respectively.

We begin by making explicit the second variation formula alluded to above.
Following \cite{Col08}, we will derive the formula by a change of 
variables.

\begin{lem}
\label{lem:mixext}
Let $K,L,M$ be convex bodies in $\mathbb{R}^n$ of class $C^\infty_+$. Then
\begin{align*}
	n(n-1)\V(K,L,\mathcal{M}) &=
	\int_{\partial M}(h_K\circ n_M)\,(h_L\circ n_M)\Tr[\II]\,dx \\
	&\qquad
	-
	\int_{\partial M}\langle \nabla_{\partial M}(h_K\circ n_M),
	\II^{-1}\nabla_{\partial M}(h_L\circ n_M)\rangle\,dx.
\end{align*}
\end{lem}

\begin{proof}
First note that by Lemma \ref{lem:repc2} and Lemma 
\ref{lem:mdprop}(\textit{b}), we have
$$
	\V(K,L,\mathcal{M}) = 
	\frac{1}{n(n-1)}
	\int_{S^{n-1}} h_K\Tr[\cof(D^2h_M)D^2h_L]\,
	d\omega.
$$
Let $\mathscr{L}f:=\Tr[\cof(D^2h_M)D^2f]$. It follows from the 
symmetry of mixed volumes that $\mathscr{L}$ is a symmetric operator on 
$C^2(S^{n-1})\subset L^2(\omega)$. We may therefore write
\begin{align*}
	&n(n-1)\V(K,L,\mathcal{M})
	\\ &=
	\frac{1}{2}
	\int_{S^{n-1}} \{h_K\mathscr{L}h_L + h_L\mathscr{L}h_K
	-\mathscr{L}(h_Kh_L)+h_Kh_L\mathscr{L}1\}\,d\omega \\
	&=
	\int_{S^{n-1}}\{h_Kh_L\Tr[\cof(D^2h_M)]-
	\langle\nabla_{S^{n-1}}h_K,\cof(D^2h_M)\nabla_{S^{n-1}}h_L\rangle\}
	\,d\omega,
\end{align*}
where we used $D^2f=\nabla_{S^{n-1}}^2f+fI$ and the product rule.

Now note that as $n_M^{-1}=\nabla h_M$, we have
$$
	D^2h_M\circ n_M =
	\nabla n_M^{-1}\circ n_M =
	(\nabla n_M)^{-1} = \II^{-1}.
$$
Similarly, we can compute by the chain rule
$$
	\nabla_{S^{n-1}}f\circ n_M =
	\II^{-1}\nabla_{\partial M}(f\circ n_M).
$$
Finally, as $D^2h_M>0$, we may write 
$$
	\cof(D^2h_M)=(D^2h_M)^{-1}\det(D^2h_M).
$$
The proof is completed by changing variables according to $n_M^{-1}$
in the expression for $\V(K,L,\mathcal{M})$ and using the
above identities.
\end{proof}

The quantities that appear in Lemma \ref{lem:mixext} are strongly 
reminiscent of the following classical formula of Reilly, obtained by 
integrating the Bochner formula on a manifold with boundary; for the 
proof, we refer to \cite[Lemma A.17]{CLN06}.

\begin{lem}
\label{lem:reilly}
Let $M\subset\mathbb{R}^n$ be a compact set with $C^\infty$ boundary. Then
\begin{align*}
	\int_M (\Delta u)^2\,dx &= 
	\int_M \Tr[(\nabla^2u)^2]\,dx \\
	&\qquad+ 
	\int_{\partial M}
	\{\Tr[\II]u_n^2 + 
	\langle\nabla_{\partial M}u,\II\,\nabla_{\partial M}u\rangle
	- 2\langle\nabla_{\partial M}u_n,\nabla_{\partial M}u\rangle\}
	\,dx
\end{align*}
for any $u\in C^\infty(M)$, where we defined the normal derivative
$u_n:=\langle n_M,\nabla u\rangle$.
\end{lem}

That Minkowski's inequality may be deduced from Reilly's formula was 
observed in a special case by Reilly himself \cite{Rei80}, and more 
generally by Kolesnikov and Milman \cite{KM18}. In particular, it was 
noticed in \cite{KM17} that the latter proof admits a quantitative 
improvement when the bodies are symmetric. The following result is a 
straightforward adaptation of \cite[section 6]{KM17} to the 
non-symmetric case.

\begin{prop}
\label{prop:km}
Let $M$ be a convex body in $\mathbb{R}^n$ of class $C^\infty_+$
with $0\in M$. Then 
$$
	\int_{\partial M}
	\langle\nabla_{\partial M}g,\II^{-1}\nabla_{\partial M}g\rangle
	\,dx -
	\int_{\partial M} g^2\Tr[\II]\,dx \ge
	\frac{r^2}{nR^2}\int_{\partial M}
	\frac{g^2}{h_M\circ n_M}\,dx
$$
for any $C^\infty$ function $g:\partial M\to\mathbb{R}$ such that
$\int_{\partial M}g(x)\,dx =0$ and $\int_{\partial M}xg(x)\,dx=0$,
where $r,R>0$ are chosen such that $rB\subseteq M\subseteq RB$.
\end{prop}

\begin{proof}
As $\int_{\partial M}g\,dx=0$, there exists a $C^\infty$ solution $u$
to the Neumann problem
\begin{align*}
	&\Delta u=0\quad\mbox{on }M,\\
	&u_n=g\quad\mbox{on }\partial M,
\end{align*}
see, e.g., \cite[section 5.7]{Tay11}.
Applying Lemma \ref{lem:reilly} yields
\begin{align*}
	0 &= 
	\int_M \Tr[(\nabla^2u)^2]\,dx 
	+ \int_{\partial M}
	\{g^2\Tr[\II] + 
	\langle\nabla_{\partial M}u,\II\,\nabla_{\partial M}u\rangle
	- 2\langle\nabla_{\partial M}g,\nabla_{\partial M}u\rangle\}
	\,dx \\
	&\ge 
	\int_M \Tr[(\nabla^2u)^2]\,dx +
	\int_{\partial M}
	\{g^2\Tr[\II] -
	\langle\nabla_{\partial M}g,\II^{-1}\nabla_{\partial M}g\rangle\}
	\,dx,
\end{align*}
where we used that $\II>0$ as $M$ is a body of class $C^\infty_+$. 
To complete the proof, it remains to lower bound the first term on the 
right-hand side.

To this end, note that $\div(\langle v,x\rangle\nabla u)=
\langle v,\nabla u\rangle$ for any $v\in\mathbb{R}^n$, as
$\Delta u=0$. Thus
$$
	\int_M \nabla u(x)\,dx = \int_{\partial M}xu_n(x)\,dx
	= \int_{\partial M}xg(x)\,dx = 0.
$$
By a classical Poincar\'e inequality of Payne and Weinberger \cite{PW60}, 
it follows that
$$
	\int_M \Tr[(\nabla^2u)^2]\,dx \ge
	\frac{\pi^2}{4R^2}\int_M\|\nabla u\|^2\,dx.
$$
As $r\le h_M\le R$ and $h_M(n_M(x))=\langle 
x,n_M(x)\rangle$, we can now estimate
\begin{align*}
	r^2\int_{\partial M}\frac{g^2}{h_M\circ n_M}\,dx 
	&\le
	\int_{\partial M}u_n^2\langle x,n_M\rangle\,dx
	\le \int_{\partial M}\|\nabla u\|^2\langle x,n_M\rangle\,dx \\
	&=\int_M \div(x\|\nabla u\|^2)\,dx \\
	&\le (n+1)\int_M\|\nabla u\|^2\,dx +
	R^2\int_M\Tr[(\nabla^2u)^2]\,dx \\
	&\le \bigg(1+\frac{4(n+1)}{\pi^2}
	\bigg)R^2\int_M\Tr[(\nabla^2u)^2]\,dx,
\end{align*}
where we used the divergence theorem and
$$
	\div(x\|\nabla u\|^2) = n\|\nabla u\|^2+2\langle\nabla 
	u,\nabla^2u\,x\rangle\le (n+1)\|\nabla u\|^2 + \|\nabla^2u\,x\|^2.
$$
The proof follows readily by combining the above estimates (for aesthetic 
reasons, we have estimated $1+\frac{4(n+1)}{\pi^2}\le n$ for $n\ge 3$ in 
the statement).
\end{proof}

Lemma \ref{lem:mixext} shows that the left-hand side of the expression in 
Proposition \ref{prop:km} is nothing other than a mixed volume 
$-n(n-1)\V(\tilde g,\tilde g,\mathcal{M})$ for $\tilde g:=g\circ 
n_M^{-1}$, but the significance of the inequality may not be immediately 
obvious. We will presently see that Proposition \ref{prop:km} controls a 
gap in the spectrum of a certain self-adjoint operator associated to mixed 
volumes (different from the one in Theorem \ref{thm:forms}). This 
observation has a number of interesting consequences. For symmetric 
bodies, it implies a local form of the $L^p$-Brunn-Minkowski inequality 
with $p<1$, which was the problem investigated in \cite{KM17}. In the 
present setting, the relevant spectral property will form the basis for 
the proof of Theorem \ref{thm:weak}.

\subsection{Proof of Theorem \ref{thm:weak}}
\label{sec:pfweakpf}

Let us fix, for the time being, a convex body $M$ in $\mathbb{R}^n$
of class $C^\infty_+$ with $0\in\intr M$. Define an operator and a
measure on $S^{n-1}$ by
$$
	\mathscr{\tilde A}f := \frac{1}{n(n-1)}
	h_M\Tr[(D^2h_M)^{-1}D^2f],\qquad
	d\tilde \nu := \frac{\det(D^2h_M)}{h_M}\,d\omega
$$
for $f\in C^\infty(S^{n-1})$. By Lemma \ref{lem:repc2} and Lemma 
\ref{lem:mdprop}(\textit{b}), we have
$$
	\V(f,g,\mathcal{M}) = 
	\langle f,\mathscr{\tilde A}g\rangle_{L^2(\tilde\nu)}
	\quad\mbox{for }f,g\in C^\infty(S^{n-1}).
$$
Consequently, we observe the following basic facts:
\begin{enumerate}[$\bullet$]
\item $\mathscr{\tilde A}$ is an elliptic operator
(as $D^2h_M>0$).
\item $\mathscr{\tilde A}$ defines a symmetric quadratic form
$\langle f,\mathscr{\tilde A}g\rangle_{L^2(\tilde\nu)}$ for
$f,g\in C^\infty(S^{n-1})$.
\item $\mathscr{\tilde A}$ has a self-adjoint extension with
compact resolvent. Moreover, as $\mathscr{\tilde 
A}h_M=\frac{1}{n}h_M$, its largest eigenvalue is $\frac{1}{n}$ and this
eigenvalue is simple (cf.\ section \ref{sec:elliptic}).
\item $\mathscr{\tilde A}\ell=0$ for any linear function 
$\ell(x)=\langle  v,x\rangle$, $v\in\mathbb{R}^n$.
\end{enumerate}
It should be emphasized that the normalization chosen in the definition
of $\mathscr{\tilde A}$ is very different than the one employed in section
\ref{sec:forms}; in particular, the present operator makes sense only for
smooth bodies $M$, and does not give rise to a well-behaved
limiting operator for arbitrary (non-smooth) bodies. Nonetheless, the
present normalization is the appropriate one for exploiting Proposition 
\ref{prop:km}. 

\begin{lem}
\label{lem:minmax}
Let $M$ be a convex body in $\mathbb{R}^n$ of class $C^\infty_+$ with
$rB\subseteq M\subseteq RB$. 
Whenever $f\in C^\infty(S^{n-1})$ satisfies
$f\perp \spn\{h_M,\ell:\ell\mbox{ is linear}\}$ in
$L^2(\tilde\nu)$, we have
$$
	\langle f,\mathscr{\tilde A}f\rangle_{L^2(\tilde\nu)}
	\le -\frac{1}{n^2(n-1)}\frac{r^2}{R^2}
	\|f\|_{L^2(\tilde\nu)}^2.
$$
\end{lem}

\begin{proof}
The statement is spectral in nature. As $\mathscr{\tilde A}$ has a compact 
resolvent, it has a discrete spectrum and a complete set of 
eigenfunctions (section \ref{sec:elliptic}).
As stated above, the largest eigenvalue of $\mathscr{\tilde A}$ is 
$\frac{1}{n}$ and its one-dimensional eigenspace is spanned by $h_M$. 
Moreover, Minkowski's inequality implies as in the proof of Theorem 
\ref{thm:forms} that 
$\spec\mathscr{\tilde A}\subseteq (-\infty,0]\cup\{\frac{1}{n}\}$, and all 
linear functions are eigenfunctions with eigenvalue $0$.
Therefore, by Lemma \ref{lem:minmaxprinciple},
$$
	\sup_{f\perp\spn\{h_M,\ell:\ell\text{ is linear}\}}
	\frac{\langle f,\mathscr{\tilde A}f\rangle_{L^2(\tilde\nu)}}{
	\|f\|^2_{L^2(\tilde\nu)}}
	\le
	\sup_{f\perp L}
	\frac{\langle f,\mathscr{\tilde A}f\rangle_{L^2(\tilde\nu)}}{
	\|f\|^2_{L^2(\tilde\nu)}}	
$$
for any linear space $L\subset C^\infty(S^{n-1})$ with $\dim L\le n+1$.

Now let $f:S^{n-1}\to\mathbb{R}$ be $C^\infty$.
Choosing $g:=f\circ n_M$ in Proposition \ref{prop:km} and changing
variables as in Lemma \ref{lem:mixext}, we find that
$$
	n(n-1)\,\langle f,\mathscr{\tilde A}f\rangle_{L^2(\tilde\nu)}
	=
	n(n-1)\,\V(f,f,\mathcal{M})
	\le
	-\frac{r^2}{nR^2}
	\|f\|^2_{L^2(\tilde\nu)}
$$
whenever
$$
	\int f h_M\,d\tilde\nu = 0,\qquad
	\int f h_M\nabla h_M\,d\tilde\nu=0.
$$
If we therefore choose $L=\spn\{h_M,h_M\langle v,\nabla h_M\rangle:
v\in\mathbb{R}^n\}$, then we have shown
$$
	\sup_{f\perp L}
        \frac{\langle f,\mathscr{\tilde A}f\rangle_{L^2(\tilde\nu)}}{
        \|f\|^2_{L^2(\tilde\nu)}}
	\le
	-\frac{1}{n^2(n-1)}\frac{r^2}{R^2}.
$$
As clearly $\dim L\le n+1$, the proof is complete.
\end{proof}

Before we proceed to the proof of Theorem \ref{thm:weak},
we state a simple lemma.

\begin{lem}
\label{lem:covariance}
For any convex body $M$ with $0\in\intr M$, we have
$$
	G_M := \int xx^*\, \frac{S_{M,\mathcal{M}}(dx)}{h_M(x)} > 0.
$$
\end{lem}

\begin{proof}
If the conclusion were false, then there must exist 
$w\in\mathbb{R}^n\backslash\{0\}$ such that $\langle w,G_Mw\rangle=0$. 
This would imply that $\langle w,x\rangle=0$ for all $0$-extreme normal 
vectors $x$ of $M$ by Theorem \ref{thm:supp}. But that is impossible, as 
a convex body with 
nonempty interior is the intersection of its regular supporting halfspaces 
\cite[Theorem 2.2.6]{Sch14}; if the normals of all these halfspaces were 
orthogonal to $w$, then $M$ would be noncompact, contradicting the 
definition of a convex body.
\end{proof}

We are now ready to complete the proof of Theorem \ref{thm:weak}.

\begin{proof}[Proof of Theorem \ref{thm:weak}]
We may clearly assume that
$\V(L,L,\mathcal{M})>0$, as otherwise the statement reduces to
Minkowski's quadratic inequality. Then $\V(L,M,\mathcal{M})>0$
as well, as by Minkowski's inequality
$\V(L,M,\mathcal{M})^2\ge\V(L,L,\mathcal{M})\Vol(M)>0$.

Let us first consider the case that $M$ is of class $C^\infty_+$ with
$rB\subseteq M\subseteq RB$ for $r,R>0$, and that $K,L$ are of class 
$C^\infty_+$. Using $\langle h_L,h_M\rangle_{L^2(\tilde\nu)} =
n\V(L,M,\mathcal{M})>0$ and Lemma \ref{lem:covariance}, we can 
define $a\ge 0$ and $v\in\mathbb{R}^n$ as
$$
	a := \frac{\langle h_K,h_M\rangle_{L^2(\tilde\nu)}}
	{\langle h_L,h_M\rangle_{L^2(\tilde\nu)}} =
	\frac{\V(K,M,\mathcal{M})}{\V(L,M,\mathcal{M})},\quad
	v := \int (h_K(x)-ah_L(x)) 
	G_M^{-1}x\,\frac{S_{M,\mathcal{M}}(dx)}{h_M(x)}.
$$
Then it is readily verified (using $d\tilde\nu=dS_{M,\mathcal{M}}/h_M$) that
$$
	\delta := h_K - ah_L - \langle v,\cdot\,\rangle
	\perp\spn\{h_M,\ell:\ell\mbox{ is linear}\}
	\quad\mbox{in }L^2(\tilde\nu).
$$
Applying Lemma \ref{lem:minmax} and translation-invariance of mixed
volumes yields
\begin{align*}
	-\frac{1}{n^2(n-1)}\frac{r^2}{R^2}
	\|\delta\|_{L^2(\tilde\nu)}^2 &\ge
	\V(\delta,\delta,\mathcal{M}) \\ &=
	\V(K,K,\mathcal{M}) - 2a\V(K,L,\mathcal{M})
	+a^2\V(L,L,\mathcal{M})
	\\ &\ge
	\V(K,K,\mathcal{M}) - 
	\frac{\V(K,L,\mathcal{M})^2}{\V(L,L,\mathcal{M})},
\end{align*}
where we minimized over $a$ in the last inequality. The conclusion follows
readily.

We now consider the general case where $K,L,M$ are arbitrary convex bodies 
in $\mathbb{R}^n$ and $0\in\intr M$. It is classical \cite[section 
3.4]{Sch14} that we may choose convex bodies $K^{(s)}$, $L^{(s)}$, 
$M^{(s)}$ of class $C^\infty_+$ such that $K^{(s)}\to K$, 
$L^{(s)}\to L$, $M^{(s)}\to M$ as $s\to\infty$ in the sense of
Hausdorff convergence. Note that as $0\in\intr M$, there exist $r,R>0$ 
such that $rB\subseteq M^{(s)}\subseteq RB$ for all $s$ sufficiently 
large. Thus we have shown that
\begin{align*}
	&\V(K^{(s)},L^{(s)},\mathcal{M}^{(s)})^2\ge 
	\V(K^{(s)},K^{(s)},\mathcal{M}^{(s)})\,\V(L^{(s)},L^{(s)},\mathcal{M}^{(s)}) 
	\\ & \qquad
	+ C_M
	\V(L^{(s)},L^{(s)},\mathcal{M}^{(s)})
	\int (h_{K^{(s)}}-a^{(s)}h_{L^{(s)}}-\langle 
	v^{(s)},\cdot\,\rangle)^2
	\,\frac{dS_{M^{(s)},\mathcal{M}^{(s)}}}{h_{M^{(s)}}}
\end{align*}
for all $s$ sufficiently large, where $a^{(s)},v^{(s)}$ are chosen
as in the first part of the proof and we defined
$C_M:=r^2/(R^2n^2(n-1))$. 
It remains to take 
$s\to\infty$ 
in this inequality. Convergence of the mixed volumes follows directly 
from Theorem \ref{thm:conv}. Moreover, that $a^{(s)}\to a$ and $v^{(s)}\to 
v$ for some $a\ge 0$, $v\in\mathbb{R}^n$ follows readily from 
Theorem~\ref{thm:conv} and the explicit expressions for $a,v$ 
given in the first 
part of the proof. Thus the integrand in the above inequality converges 
uniformly to $(h_K-ah_L-\langle v,\cdot\,\rangle)^2/h_M$, which implies 
convergence of the integral. This concludes the proof.
\end{proof}

\section{A quantitative rigidity theorem}
\label{sec:rigid}

The weak stability result of the previous section implies that if equality 
holds in \eqref{eq:minkq} (and $M$ has nonempty interior and 
$\V(L,L,\mathcal{M})>0$) then, up to homothety, $K$ and $L$ have the same 
supporting hyperplanes in the $0$-extreme normal directions of $M$. This 
is however far from characterizing the extremals of Minkowski's quadratic 
inequality, as was illustrated in Example \ref{ex:weak}. Nonetheless, we 
will show that this weak information can be amplified to recover the full 
equality cases, because the extremals of Minkowski's inequality turn out 
to be very rigid: once they are fixed in the $0$-extreme directions of 
$M$, their extension to the $1$-extreme directions of $M$ is uniquely 
determined. This rigidity property, formulated above as Theorem 
\ref{thm:rigidintro}, lies at the heart of our proof of Theorem 
\ref{thm:mainfull}.

Theorem \ref{thm:rigidintro} is an immediate consequence of the following
quantitative result.

\begin{thm}
\label{thm:rigid}
Let $M$ be a convex body in $\mathbb{R}^n$ with nonempty interior.
There exist $C_M>0$ and a measure $\mu_M$, depending only
on $M$, so that $\supp\mu_M\subseteq\supp S_{M,\mathcal{M}}$ and
\begin{align*}
	\V(K,L,\mathcal{M})^2\ge \mbox{} &
	\V(K,K,\mathcal{M})\,\V(L,L,\mathcal{M}) + \mbox{} \\ & \mbox{}
	C_M\V(L,L,\mathcal{M})\big\{
	\|h_K-h_L\|^2_{L^2(S_{B,\mathcal{M}})} -
	\|h_K-h_L\|^2_{L^2(\mu_M)}
	\big\}
\end{align*}
for all convex bodies $K,L$ in $\mathbb{R}^n$.
\end{thm}

The formulation of Theorem \ref{thm:rigid} is subtle due to the measure 
$\mu_M$ appearing here. As will be explained in the proof, this measure 
does not appear to have a canonical geometric interpretation. Ideally, one 
would have liked to prove a version of Theorem \ref{thm:rigid} where 
$\mu_M$ is replaced by the measure $S_{M,\mathcal{M}}/h_M$ that appears in 
Theorem \ref{thm:weak}. If this were possible, then one would even obtain 
a sharp quantitative analogue of Theorem \ref{thm:mainfull} (that is, a 
stability form of Minkowski's quadratic inequality). It is far from clear, 
however, how such a result might be proved: we do not know how to directly 
relate the measures $S_{M,\mathcal{M}}$ and $S_{B,\mathcal{M}}$. 
Fortunately, to characterize the extremals it suffices to work with the 
measure $\mu_M$ in Theorem \ref{thm:rigid}, which may be viewed as a 
projection of $S_{B,\mathcal{M}}$ on the $0$-extreme normal vectors of 
$M$.

\begin{rem}
\label{rem:gap}
Spectrally, a stability form of Minkowski's quadratic inequality 
\begin{align*}
	\V(K,L,\mathcal{M})^2 \stackrel{?}{\ge}\mbox{}
	&\V(K,K,\mathcal{M})\,\V(L,L,\mathcal{M}) +\mbox{}
	\\ & C_M\V(L,L,\mathcal{M})
	\inf_{v\in\mathbb{R}^n,a\ge 0}\|h_K-ah_L-\langle 
	v,\cdot\,\rangle\|^2_{L^2(S_{B,\mathcal{M}})}
\end{align*}
may be shown as in the proof of Lemma \ref{lem:hyper} to be
equivalent to the following: (i) the kernel of the operator 
$\mathscr{A}$ in Theorem \ref{thm:forms} consists only of linear functions 
(which characterizes the extremals); and (ii) the remainder of 
the 
spectrum is separated from zero by a positive constant (which quantifies 
the deficit). If $\mathscr{A}$ were to have compact resolvent, then (ii)
would follow directly from (i) by discreteness of the spectrum.
Unfortunately, as we will see in section \ref{sec:lower}, it is not true 
in general that $\mathscr{A}$ has compact resolvent. For this reason, it 
is far from clear whether we might expect even in principle to replace 
$\mu_M$ by $S_{M,\mathcal{M}}/h_M$ in Theorem \ref{thm:rigid}. 
Understanding the answer to this question would be of considerable 
interest.
\end{rem}

The rest of this section is organized as follows. In section 
\ref{sec:pfmainfull}, we complete the proof of Theorem \ref{thm:mainfull}
using Theorems \ref{thm:weak} and \ref{thm:rigid}. Sections 
\ref{sec:rigidpoly} and \ref{sec:rigidpf} are devoted to the proof of
Theorem \ref{thm:rigid}. In section \ref{sec:rigidpoly}, we consider
the special case where $M$ is a polytope. We then extend the conclusion
to general bodies $M$ in section \ref{sec:rigidpf}.

\subsection{Proof of Theorem \ref{thm:mainfull}}
\label{sec:pfmainfull}

Before we proceed to the proof of Theorem \ref{thm:rigid}, let us show how 
Theorems \ref{thm:weak} and \ref{thm:rigid} combine to complete the proof 
of Theorem \ref{thm:mainfull}. Note that the \emph{if} direction of 
Theorem \ref{thm:mainfull} was already proved by Schneider \cite{Sch85} by 
a slightly different method; it is the \emph{only if} direction that is 
new. For completeness, we include a proof here of 
both directions using the methods of this paper.

\begin{proof}[Proof of Theorem \ref{thm:mainfull}]
In the following, we assume that $K,L,M$ are convex bodies in 
$\mathbb{R}^n$ such that $M$ has nonempty interior and 
$\V(L,L,\mathcal{M})>0$.

Suppose first that there exist $a\ge 0$ and $v\in\mathbb{R}^n$ so that
$K$ and $aL+v$ have the same supporting hyperplanes
in all $1$-extreme normal directions of $M$. Then
$$
	h_K-ah_L = \langle v,\cdot\,\rangle
	\quad S_{B,\mathcal{M}}\mbox{-a.e.}
$$
by Theorem \ref{thm:supp}. Therefore, denoting by $\mathscr{A}$ and 
$\mathcal{E}$ the operator and quadratic form of Theorem \ref{thm:forms}, 
we have $h_K-ah_L\in\ker\mathscr{A}$ and 
$\mathcal{E}(h_L,h_L)=\V(L,L,\mathcal{M})>0$. Thus equality in
\eqref{eq:minkq} follows from Lemma \ref{lem:hyper}.

Conversely, suppose that we have equality in \eqref{eq:minkq}.
By translation-invariance of mixed volumes, we may assume without loss
of generality that $0\in\intr M$. Then Theorem \ref{thm:weak} implies that 
there exist $a\ge 0$ and $v\in\mathbb{R}^n$ such that
$$
	\delta := h_K-ah_L - \langle v,\cdot\,\rangle = 0
	\quad S_{M,\mathcal{M}}\mbox{-a.e.}
$$
(note that the infimum in Theorem \ref{thm:weak} is clearly attained, as 
it is the minimum of a nonnegative quadratic function).

By continuity, it follows that $\delta$ vanishes on $\supp\mu_M\subseteq
\supp S_{M,\mathcal{M}}$, where $\mu_M$ is as in 
Theorem \ref{thm:rigid}. Consequently, applying Theorem \ref{thm:rigid} 
with $L\mapsto aL+v$ yields
$$
        \delta = h_K-ah_L- \langle v,\cdot\,\rangle = 0
        \quad S_{B,\mathcal{M}}\mbox{-a.e.},
$$
where we have used the invariance of Minkowski's quadratic inequality
under translation and scaling of $L$. By continuity, it follows
that $\delta$ vanishes on $\supp S_{B,\mathcal{M}}$.
Thus Theorem \ref{thm:supp} implies that $K$ and $aL+v$ have the same 
supporting hyperplanes in all $1$-extreme normal directions of $M$,
completing the proof.
\end{proof}

\subsection{Proof of Theorem \ref{thm:rigid}: polytopes}
\label{sec:rigidpoly}

In this section we consider the case that $M$ is a polytope with 
nonempty interior. At a qualitative level, the rigidity property of the 
extremals of Minkowski's inequality admits in this case a very intuitive 
interpretation. Suppose we have equality in \eqref{eq:minkq}, so that 
$f:=h_K-ah_L\in\ker\mathscr{A}$ for some $a\ge 0$. Suppose in addition 
that we have fixed the values of $f$ in the $0$-extreme normal directions 
of $M$, which are in this case the vertices of the metric graph associated 
to $M$. Then it follows from Theorem \ref{thm:qgraph} that $f$ solves the 
Dirichlet problem $f''+f=0$ on each edge of the metric graph with 
boundary data on the vertices. It is readily verified by explicit 
computation that this one-dimensional Dirichlet problem has a unique 
solution as long as the lengths of all the edges are less than $\pi$, 
which must be the case as $M$ has nonempty interior. Thus the value of
$f$ is uniquely determined on the $1$-extreme normal directions of $M$
once we have fixed its values on the $0$-extreme normal directions.

This intuitive argument appears to be rather special to the case of 
polytopes: for a general body $M$, the structure of the sets of $0$- and 
$1$-extreme normal vectors can be highly irregular, and it is far from 
clear even how to make sense of the Dirichlet problem in this setting. 
Instead, we will proceed by developing a quantitative formulation of the 
above intuition for polytopes. The key point is to find the ``right'' 
formulation that does not degenerate when we approximate an arbitrary 
convex body $M$ by polytopes. Once such a formulation has been found, we 
will be able to extend its conclusion to the general setting by taking 
limits.

We now proceed to make these ideas precise. In the rest of this 
subsection, $M$ will be a polytope in $\mathbb{R}^n$ with nonempty 
interior, and we adopt without further comment the definitions and 
notation of section \ref{sec:qgraph}. Our starting point is the following
Poincar\'e-type inequality on a single edge of the metric graph.

\begin{lem}
\label{lem:edge}
Let $M$ be a polytope in $\mathbb{R}^n$ with nonempty interior, and let
$F\sim F'$ be neighboring facets. Then for any function 
$f\in H^1(e_{F,F'})$ and $0<\varepsilon<1$, we have
$$
	l_{F,F'}^2\int_{e_{F,F'}}
	(f')^2\,d\mathcal{H}^1 \ge
	(1-\varepsilon)^2\pi^2\int_{e_{F,F'}}f^2\,d\mathcal{H}^1
	-\frac{2}{\varepsilon}l_{F,F'}\{f(n_F)^2+f(n_{F'})^2\}.
$$
\end{lem}

\begin{proof}
Assume without loss of generality that $F\le F'$, and recall that we 
parametrize functions $f:e_{F,F'}\to\mathbb{R}$ as
$f(\theta)$ for $\theta\in[0,l_{F,F'}]$, where $\theta=0$ corresponds
to vertex $n_F$ and $\theta=l_{F,F'}$ corresponds to vertex $n_{F'}$.
Define the function
$$
	y(\theta) := \cos\bigg(
	\frac{(1-\varepsilon)\pi}{l_{F,F'}}\bigg(
	\theta-\frac{l_{F,F'}}{2}\bigg)\bigg),
$$
and note that $y(\theta)>0$ for $\theta\in[0,l_{F,F'}]$.
Defining $g:=f/y$, we compute
\begin{align*}
	\int_0^{l_{F,F'}} (f')^2\,d\theta &=
	\int_0^{l_{F,F'}} \{(g')^2y^2 + g^2(y')^2 + 
		(g^2)'yy'\}\,d\theta \\
	&=
	\int_0^{l_{F,F'}} \{(g')^2y^2-g^2yy''\}\,d\theta +
	g^2yy'\bigg|_0^{l_{F,F'}},
\end{align*}
where we integrated the last term by parts. But note that
$$
	yy'' = -\frac{(1-\varepsilon)^2\pi^2}{l_{F,F'}^2}y^2,\quad
	\frac{y'}{y}(0)=-\frac{y'}{y}(l_{F,F'}) =
	\frac{(1-\varepsilon)\pi}{l_{F,F'}}
		\tan\bigg(\frac{(1-\varepsilon)\pi}{2}\bigg)
	\le\frac{2}{l_{F,F'}\varepsilon}.
$$
It follows that
$$
	\int_0^{l_{F,F'}} (f')^2\,d\theta \ge
	\frac{(1-\varepsilon)^2\pi^2}{l_{F,F'}^2}
	\int_0^{l_{F,F'}} f^2\,d\theta 
	-\frac{2}{l_{F,F'}\varepsilon}\{f(l_{F,F'})^2+f(0)^2\}.
$$
Rearranging this expression yields the conclusion.
\end{proof}

\begin{rem}
Let us note that Lemma \ref{lem:edge} may indeed be viewed as a 
quantitative formulation of uniqueness of the Dirichlet problem on an 
edge. Indeed, suppose $f_1,f_2$ both satisfy $f_i''+f_i=0$ on $e_{F,F'}$, 
and that $f_1,f_2$ agree on the vertices $n_F,n_{F'}$. Then applying
Lemma \ref{lem:edge} to $f=f_1-f_2$ and letting $\varepsilon\to 0$
yields
$$
	0 = -l_{F,F'}^2\int_{e_{F,F'}} f(f''+f)\,d\mathcal{H}^1
	= l_{F,F'}^2\int_{e_{F,F'}} \{(f')^2-f^2\}\,d\mathcal{H}^1
	\ge C\int_{e_{F,F'}}f^2\,d\mathcal{H}^1
$$
with $C=\pi^2-l_{F,F'}^2$, where we integrated by parts in the second
equality. Thus provided $l_{F,F'}<\pi$, the two solutions must coincide
$f_1=f_2$.
\end{rem}

Next, we note that when $M$ has nonempty interior, then the lengths 
$l_{F,F'}$ of all edges must be bounded away from $\pi$. The following 
lemma quantifies this idea.

\begin{lem}
\label{lem:2d}
Let $M$ be a polytope in $\mathbb{R}^n$ such that $rB\subseteq 
M\subseteq RB$. Then
$$
	\tan\bigg(\frac{l_{F,F'}}{2}\bigg)\le\frac{R}{r}
	\quad\mbox{for all }e_{F,F'}\in E.
$$
In particular, we can estimate
$$
	l_{F,F'}^2 \le 
	\pi^2\bigg(
	1-\frac{r^2}{R^2+r^2}
	\bigg)
	\quad\mbox{for all }e_{F,F'}\in E.
$$
\end{lem}

\begin{proof}
Define for every facet $F\in\mathcal{F}$ the supporting hyperplane
$$
	H_F :=\{x\in\mathbb{R}^n:\langle n_F,x\rangle = h_M(n_F)\}. 
$$
Let $F\sim F'$ be neighboring facets. We make the following claims.
\begin{enumerate}[a.]
\itemsep\abovedisplayskip
\item We claim that $h_M(n_F)\ge r$ and $h_M(n_{F'})\ge r$. Indeed, note 
that as $rB\subseteq M$, we have $h_M(n_F)\ge h_{rB}(n_F) = r$,
and similarly for $F'$.
\item We claim that $H_F\cap H_{F'}\cap RB\ne\varnothing$. Indeed, this
follows readily by noting that
$F\cap F'\subset H_F\cap H_{F'}$ and 
$\varnothing\ne F\cap F'\subset M\subseteq RB$.
\end{enumerate}
Now note that we can write for any $x\in\mathbb{R}^n$
$$
	\|x\|^2 =
	\langle x,n_F\rangle^2 +
	\frac{(\langle x,n_{F'}\rangle - \cos(l_{F,F'})\langle x,n_F\rangle)^2
	}{\sin(l_{F,F'})^2}
	+
	\|P_{\{n_F,n_{F'}\}^\perp}x\|^2,
$$
where we used $\langle n_F,n_{F'}\rangle = \cos(l_{F,F'})$. Thus 
$$
	R^2 \ge \inf_{x\in H_F\cap H_{F'}}\|x\|^2
	= 
	h_M(n_F)^2 +
	\frac{(h_M(n_{F'}) - \cos(l_{F,F'})h_M(n_F))^2}{\sin(l_{F,F'})^2}
$$
by claim b above. Applying claim a yields
$$
	R^2 \ge  
	r^2\bigg(1 +
	\frac{(1 - \cos(l_{F,F'}))^2}{\sin(l_{F,F'})^2}\bigg)
$$
provided $\frac{\pi}{2}\le l_{F,F'}<\pi$. It follows that
$$
	\tan\bigg(\frac{l_{F,F'}}{2}\bigg) =
	\frac{1-\cos(l_{F,F'})}{\sin(l_{F,F'})} \le \frac{R}{r}.
$$
Indeed, for $\frac{\pi}{2}\le l_{F,F'}<\pi$ this is immediate from the
previous expression, while for $l_{F,F'}<\frac{\pi}{2}$ this
follows as $\tan(\frac{\pi}{4})=1\le \frac{R}{r}$. To deduce the
second part of the statement, it remains to note
that $4\arctan(x)^2\le \pi^2x^2/(1+x^2)$.
\end{proof}

Combining Lemmas \ref{lem:edge} and \ref{lem:2d} yields the following.

\begin{cor}
\label{cor:poincedge}
Let $M$ be a polytope in $\mathbb{R}^n$ such that $rB\subseteq M
\subseteq RB$. Then for any neighboring facets
$F\sim F'$ of $M$ and any function $f\in H^1(e_{F,F'})$, we have
$$
	\int_{e_{F,F'}}\{(f')^2-f^2\}\,d\mathcal{H}^1 \ge
	\frac{r^2}{2R^2}
	\int_{e_{F,F'}}f^2\,d\mathcal{H}^1
	-\frac{4R^2}{r^2}l_{F,F'}\{f(n_F)^2+f(n_{F'})^2\}.
$$
\end{cor}

\begin{proof}
Applying Lemma \ref{lem:2d} to the left-hand side of Lemma \ref{lem:edge}
and rearranging the resulting expression yields the following inequality:
\begin{align*}
	\int_{e_{F,F'}}
	\{(f')^2-f^2\}\,d\mathcal{H}^1 &\ge
	\bigg(
	\frac{R^2+r^2}{R^2}
	(1-\varepsilon)^2-1\bigg)
	\int_{e_{F,F'}}f^2\,d\mathcal{H}^1 \\ &\qquad
	-
	\frac{R^2+r^2}{R^2}\frac{2}{\varepsilon\pi^2}
	\,l_{F,F'}\{f(n_F)^2+f(n_{F'})^2\}.
\end{align*}
Now choose $\varepsilon = \frac{r^2}{4(R^2+r^2)}$. Then
$$
	\frac{R^2+r^2}{R^2}
        (1-\varepsilon)^2-1 \ge
	\frac{R^2+r^2}{R^2}
	(1-2\varepsilon) - 1
	=
	\frac{r^2}{2R^2},
$$
while
$$
	\frac{R^2+r^2}{R^2}\frac{2}{\varepsilon\pi^2} =
	\frac{8}{\pi^2}\frac{(R^2+r^2)^2}{r^2R^2}
	\le
	\frac{32}{\pi^2}\frac{R^2}{r^2}.
$$
To conclude, we estimate $\frac{32}{\pi^2}\le 4$ for aesthetic appeal.
\end{proof}

We are now ready to prove a form of Theorem \ref{thm:rigid} for polytopes.

\begin{prop}
\label{prop:polyrigid}
Let $M$ be a polytope in $\mathbb{R}^n$ such that
$rB\subseteq M\subseteq RB$. Define a measure $\mu_M$ on the
vertices of the associated metric graph by setting
$$
	\mu_M(\{n_F\}) := \frac{1}{n-1}\sum_{F':F'\sim F}
	\mathcal{H}^{n-2}(F\cap F')\,l_{F,F'}
$$
for all facets $F$ of $M$. Then we have
\begin{align*}
	&\V(K,L,\mathcal{M})^2
	\ge 
	\V(K,K,\mathcal{M})\,\V(L,L,\mathcal{M}) \\ &
	\qquad + \V(L,L,\mathcal{M})\bigg(
	\frac{r^2}{2nR^2}\int (h_K-h_L)^2\,dS_{B,\mathcal{M}}
	-\frac{4R^2}{nr^2}\int (h_K-h_L)^2\,d\mu_M\bigg)
\end{align*}
for all convex bodies $K,L$ in $\mathbb{R}^n$.
\end{prop}

\begin{proof}
Let $\mathcal{E}$ be the quadratic form of Theorem \ref{thm:qgraph}
and $f\in\Dom\mathcal{E}$. Multiplying the inequality of Corollary
\ref{cor:poincedge} by $\mathcal{H}^{n-2}(F\cap F')$ and summing
over all edges yields
$$
	0
	\ge
	\mathcal{E}(f,f)
	+\frac{r^2}{2nR^2}
	\int f^2\,dS_{B,\mathcal{M}}
	-\frac{4R^2}{nr^2}\int f^2\,d\mu_M.
$$
Now let $f=h_K-h_L$. Then we obtain
\begin{align*}
	0
	&\ge
	\V(K,K,\mathcal{M})-2\,\V(K,L,\mathcal{M})+\V(L,L,\mathcal{M})
	\\
	&\qquad+\frac{r^2}{2nR^2}
	\int (h_K-h_L)^2\,dS_{B,\mathcal{M}}
	-\frac{4R^2}{nr^2}\int (h_K-h_L)^2\,d\mu_M.
\end{align*}
It remains to note that
$$
	\V(K,K,\mathcal{M})-2\,\V(K,L,\mathcal{M})+\V(L,L,\mathcal{M}) 
	\ge
	\V(K,K,\mathcal{M})-\frac{\V(K,L,\mathcal{M})^2}{
	\V(L,L,\mathcal{M})}
$$
when $\V(L,L,\mathcal{M})>0$, so the conclusion follows readily in this
case.  On
the other hand, when $\V(L,L,\mathcal{M})=0$ the conclusion is trivial.
\end{proof}

\subsection{Proof of Theorem \ref{thm:rigid}: general case}
\label{sec:rigidpf}

In order to prove Theorem \ref{thm:rigid} for an arbitrary convex body 
$M$, we will approximate it by polytopes and take limits in Proposition 
\ref{prop:polyrigid}. The main issue that we will encounter is to 
understand the behavior of the measure $\mu_M$ under taking limits.

At first sight, one might hope that $\mu_M$ is a natural geometric object 
that remains meaningful for arbitrary convex bodies, just like 
$S_{B,\mathcal{M}}$ or $S_{M,\mathcal{M}}$. This does not appear to be the 
case, however. It is important to note that even within the class of 
polytopes, the measure $\mu_M$ is not continuous with respect to Hausdorff 
convergence, as is illustrated by the following example.

\begin{example}
Consider a cube $M_\varepsilon$ with one of its edges sliced off at width
$\varepsilon$; this construction is illustrated in Figure 
\ref{fig:noncan}. Then $M_\varepsilon$ has, for all $\varepsilon>0$, an 
additional facet $F$ as compared to $M_0$. It is readily seen that 
$\inf_{\varepsilon>0}\mu_{M_{\varepsilon}}(\{n_F\})>0$, while 
$\mu_{M_0}(\{n_F\})=0$. Thus $M_\varepsilon\to M_0$ but 
$\mu_{M_\varepsilon}\not\to\mu_{M_0}$ as $\varepsilon\to 0$.
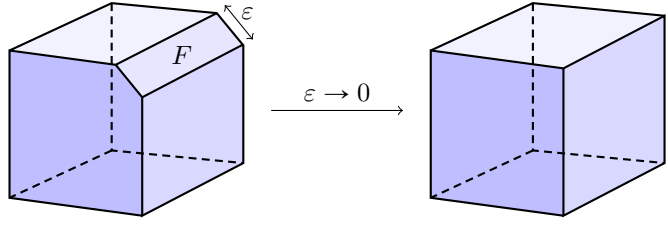
\begin{figure}
\centering
\begin{tikzpicture}[scale=.7]

\begin{scope}[xshift=-8cm]

\fill[blue!15] (-5.95,0) -- (-4.03,1) -- 
(-4.03,-1.25) -- (-5.95,-2.25) -- (-5.95,0);

\fill[blue!25] (-6.454,.618) -- (-8.47,.89) -- 
(-8.47,-1.91) -- (-5.95,-2.25) -- (-5.95,0);

\fill[blue!5] (-8.47,.89) -- (-6.53,1.79) -- 
(-4.534,1.6) -- (-6.454,.618) -- (-8.47,.89);

\fill[blue!10] (-6.454,.618) -- (-5.95,0) -- (-4.04,1)
-- (-4.534,1.6) -- (-6.454,.618);

\draw[thick] (-4.03,1) -- (-4.03,-1.25) --
(-5.95,-2.25) -- (-5.95,0);

%cutout
\draw[thick] (-6.454,.618) -- (-5.95,0) -- (-4.04,1)
-- (-4.534,1.6) -- (-6.454,.618);

\draw (-5.2,0.8) node {$F$};

\draw[thick] (-6.454,.618) -- (-8.47,.89) -- (-8.47,-1.91) --
(-5.95,-2.25);

\draw[thick] (-8.47,.89) -- (-6.53,1.79) -- (-4.534,1.6);

\draw[thick,densely dashed] (-6.53,1.79) -- (-6.53,-1.01) --
(-4.03,-1.25);

\draw[thick,densely dashed] (-6.53,-1.01) -- (-8.47,-1.91);

\draw[<->] (-3.85,1.1) -- (-4.38,1.75);

\draw (-3.95,1.6) node {$\varepsilon$};

\end{scope}

\draw[->] (-11.5,-.25) -- (-9,-.25) node[midway,above] {$\varepsilon\to 0$};

\fill[blue!15] (-5.95,.55) -- (-4.03,1.55) -- 
(-4.03,-1.25) -- (-5.95,-2.25) -- (-5.95,.55);

\fill[blue!25] (-5.95,.55) -- (-8.47,.89) -- 
(-8.47,-1.91) -- (-5.95,-2.25) -- (-5.95,.55);

\fill[blue!5] (-8.47,.89) -- (-6.53,1.79) -- 
(-4.03,1.55) -- (-5.95,.55) -- (-8.47,.89);

\draw[thick] (-5.95,.55) -- (-4.03,1.55) -- (-4.03,-1.25) --
(-5.95,-2.25) -- (-5.95,.55);

\draw[thick] (-5.95,.55) -- (-8.47,.89) -- (-8.47,-1.91) --
(-5.95,-2.25);

\draw[thick] (-8.47,.89) -- (-6.53,1.79) -- (-4.03,1.55);

\draw[thick,densely dashed] (-6.53,1.79) -- (-6.53,-1.01) --
(-4.03,-1.25);

\draw[thick,densely dashed] (-6.53,-1.01) -- (-8.47,-1.91);

\end{tikzpicture}
\caption{Example of discontinuity of the measure $\mu_M$.\label{fig:noncan}}
\end{figure}
\end{example}

For a polytope $M$, Theorem \ref{thm:qgraph} shows that the 
mass assigned by $\mu_M$ to a vertex of the metric graph is precisely the 
$S_{B,\mathcal{M}}$-measure of its incident edges. We may therefore view 
$\mu_M$ as a kind of projection of $S_{B,\mathcal{M}}$ onto the 
$0$-extreme normal vectors of $M$. It is not clear, however, what this 
might mean for a general convex body $M$, and the above example 
illustrates that one cannot hope to canonically define such projections by 
approximation of general bodies by polytopes. Nonetheless, as
$$
	\mu_M(S^{n-1}) = 2\, S_{B,\mathcal{M}}(S^{n-1}) =
	2n\,\V(B,B,\mathcal{M})
$$
by Theorem \ref{thm:qgraph}, the total mass of $\mu_M$ is uniformly 
bounded for any convergent sequence of polytopes, and we may therefore
extract a weakly convergent subsequence of these measures
(using the classical fact that any bounded family of measures on a compact 
metric space is
relatively sequentially compact for the topology of weak convergence, cf.\ 
\cite[Theorem 8.6.2]{Bog07}). While the
limiting measure is not uniquely defined by the limiting body, we
can nonetheless guarantee it satisfies our desired properties by 
working with specially chosen polytope approximations.

\begin{lem}
\label{lem:0approx}
Let $M$ be any convex body in $\mathbb{R}^n$ with $0\in\intr M$.
Then there exists a sequence of polytopes $M_k$ in $\mathbb{R}^n$ with the 
following properties:
\begin{enumerate}[a.]
\item $M_k\to M$ in Hausdorff metric.
\item There exist $r,R>0$ so that $rB\subseteq M_k\subseteq RB$ for all $k$.
\item $\supp S_{M_k,\mathcal{M}_k}\subseteq \supp S_{M,\mathcal{M}}$
for all $k$.
\item $\mu_{M_k}$ converges weakly to a limiting measure
$\mu_{M}$ with $\supp\mu_M\subseteq\supp S_{M,\mathcal{M}}$.
\end{enumerate}
\end{lem}

\begin{proof}
Recall that a regular boundary point of $M$ is a point in $\partial M$ 
that has a unique outer normal vector; in particular, the normal vector at 
a regular boundary point is $0$-extreme \cite[section 2.2]{Sch14}.
Choose a countable dense subset of the regular boundary points of
$M$, and let $\{n_i\}_{i\ge 1}$ be the corresponding normal directions.
Then $\{n_i\}\subseteq\supp S_{M,\mathcal{M}}$ by Theorem \ref{thm:supp}. 
Moreover, as a convex body with nonempty interior is the intersection of 
its regular supporting halfspaces \cite[Theorem 2.2.6]{Sch14}, we have
$$
	M = \bigcap_{i\ge 1}\{x\in\mathbb{R}^n:\langle x,n_i\rangle
	\le h_M(n_i)\}.
$$
Now define
$$
	M_k' := \bigcap_{1\le i\le k}\{x\in\mathbb{R}^n:
	\langle x,n_i\rangle\le h_M(n_i)\}.
$$
Then we have the following properties.
\begin{enumerate}[i.]
\item $M_k'$ is a polytope for all sufficiently large $k$.
\item $M_k'\to M$  as $k\to\infty$ in Hausdorff metric by
\cite[Lemma 1.8.2]{Sch14}.
\item $rB\subseteq M\subseteq M_k'$ for all $k$ with $r>0$, as 
$0\in\intr M$.
\item $M_k'\subseteq RB$ for all sufficiently large $k$ with
$R=\mathop\mathrm{diam}M$ by property ii.
\item $\supp S_{M_k',\mathcal{M}_k'}\subseteq\{n_i\}_{1\le i\le k}
\subseteq \supp S_{M,\mathcal{M}}$ for all $k$ by Theorem \ref{thm:supp}.
\end{enumerate}
Now note that when $M_k'\subseteq RB$, we can estimate
$$
	\mu_{M_k'}(S^{n-1}) =
	2n\,\V(B,B,\mathcal{M}_k') \le
	2n R^{n-2}\,\Vol(B).
$$
By property iv, the mass of $\mu_{M_k'}$ is uniformly bounded for all 
sufficiently large $k$. We may therefore extract a subsequence $\{M_k\}$ 
of $\{M_k'\}$ such that $\mu_{M_k}$ converges weakly to a limiting measure 
$\mu_M$, and such that properties a--c in the statement of the Lemma hold.
It remains to show that $\supp\mu_M\subseteq\supp S_{M,\mathcal{M}}$;
this follows immediately, however, as
$\supp\mu_{M_k}=\supp S_{M_k,\mathcal{M}_k}
\subseteq \supp S_{M,\mathcal{M}}$ for all $k$.
\end{proof}

We can now complete the proof of Theorem \ref{thm:rigid}.

\begin{proof}[Proof of Theorem \ref{thm:rigid}]
By translation-invariance of mixed volumes and mixed area measures, we may 
assume without loss of generality that $0\in \intr M$. Define
the sequence of polytopes $M_k$ and the measure $\mu_M$ as in
Lemma \ref{lem:0approx}. Applying Proposition \ref{prop:polyrigid}
to $M_k$ and taking the limit as $k\to\infty$, the conclusion follows
readily from Theorem \ref{thm:conv}. (For aesthetic reasons, we have
rescaled the definition of the measure $\mu_M$ in the statement of Theorem 
\ref{thm:rigid} so that only a single constant $C_M$ appears; this makes
no difference, of course, to the statement of the result.) 
\end{proof}

\section{The lower-dimensional case}
\label{sec:lower}

In the setting of Theorem \ref{thm:mainfull}, we have seen that the 
extremals of Minkowski's inequality have a simple spectral interpretation: 
the kernel of the operator $\mathscr{A}$ of Theorem \ref{thm:forms} always 
contains the linear functions, and Theorem \ref{thm:mainfull} shows that 
when $M$ has nonempty interior, these are the \emph{only} elements of the 
kernel that are differences of support functions of convex bodies.

When $M$ is a lower-dimensional body, however, Theorem \ref{thm:mainlower} 
states that new equality cases appear. Thus, unlike in the 
full-dimensional case, nonlinear differences of support functions can 
appear in the kernel of $\mathscr{A}$. This may suggest that the 
lower-dimensional situation is more complicated, as we must understand the 
new elements of the kernel. In fact, somewhat surprisingly, the 
lower-dimensional situation turns out to be considerably simpler: when $M$ 
has empty interior, the operator $\mathscr{A}$ can be described explicitly 
in complete generality (i.e., not just in special cases such as smooth 
bodies or polytopes). Once the operator has been constructed, we will be 
able to compute its kernel directly, and the proof of Theorem 
\ref{thm:mainlower} will follow. These ideas will be developed in the 
remainder of this section.

To gain some insight into the lower-dimensional situation, it is 
instructive to consider first the case of a lower-dimensional polytope 
$M\subset w^\perp$ for some $w\in S^{n-1}$. The following discussion is 
illustrated in Figure \ref{fig:lower}. To understand the operator 
associated to $M$, we first approximate it by the ``cylinder'' 
$M_\varepsilon:= M+\varepsilon[0,w]$ which has nonempty interior. The body 
$M_\varepsilon$ has two types of facets:
\begin{enumerate}[1.]
\item Two facets with normals $\pm w$ are translates of $M$.
\item The remaining facet normals are the normals of 
the $(n-2)$-faces of $M$ in $w^\perp$.
\end{enumerate}
The body $M_\varepsilon$ defines a quantum graph according to 
Theorem \ref{thm:qgraph}. We now formally let $\varepsilon\to 0$ and 
investigate what happens to the quantum graph in the limit. For each pair 
of facets $F\sim F'$ of $M_\varepsilon$ of type 2, we evidently
have $\mathcal{H}^{n-2}(F\cap F')=O(\varepsilon)$.
Thus all edges in the quantum graph associated to 
$M_\varepsilon$ that lie in $w^\perp$ vanish as $\varepsilon\to 0$. 
Consequently, the limiting graph has an extremely simple structure:
it has exactly two vertices at the antipodal points $\pm w$; and its 
edges are the geodesic arcs between $\pm w$ in the directions of 
the $(n-2)$-faces of $M$ in $w^\perp$.
\begin{figure} 
\centering 
\begin{tikzpicture}[scale=.7]

% epsilon-graph

\begin{scope}
\shade[ball color = blue, opacity = 0.15] (1,0) circle [radius=2];

\draw[thick, densely dashed] (3,0) arc [start angle = 0, end angle = 180, 
x radius = 2, y radius = .5];

\draw[thick] (3,0) arc [start angle = 0, end angle = -180, x radius = 2, 
y radius = .5];

\draw[thick] (1,2) [rotate=90] arc [start angle = 0, end angle = 
-180, x radius = 2, y radius = -1];

\draw[thick, densely dashed] (1,2) [rotate=90] arc [start angle 
= 0, end angle = 
-180, x radius = 2, y radius = 1];

\draw[thick] (1,2) [rotate=90] arc [start angle = 0, end angle = 
-180, x radius = 2, y radius = 1.75];

\draw[thick, densely dashed] (1,2) [rotate=90] arc [start angle 
= 0, end angle = 
-180, x radius = 2, y radius = -1.75];

\draw[fill=black] (1,2) circle [radius=.07];
\draw[fill=black] (1,-2) circle [radius=.07];
\draw[fill=black] (2.74,-.24) circle [radius=.07];
\draw[fill=black] (-0.74,.24) circle [radius=.07];
\draw[fill=black] (.03,-.45) circle [radius=.07];
\draw[fill=black] (1.97,.45) circle [radius=.07];

\draw (4.2,-1.2) node[below] 
{\small $S_{B,\mathcal{M}_\varepsilon}(e)=O(\varepsilon)$};

\draw[->] (3.5,-1.2) to [out=110,in=-45] (1.5,-.6);

\draw (-1,0) node[left] {$G_\varepsilon=$};
\end{scope}

% epsilon-cube

\begin{scope}[xshift=7.2cm,yshift=4.5cm]

\begin{scope}[scale=.9,yshift=.75cm,xshift=-.5cm]

\fill[blue!15] (-5.95,-.95) -- (-4.03,0.05) -- 
(-4.03,-1.25) -- (-5.95,-2.25) -- (-5.95,-.95);

\fill[blue!25] (-5.95,-.95) -- (-8.47,-.61) -- 
(-8.47,-1.91) -- (-5.95,-2.25) -- (-5.95,-.95);

\fill[blue!5] (-8.47,-.61) -- (-6.53,.29) -- 
(-4.03,.05) -- (-5.95,-.95) -- (-8.47,-.61);

\draw[thick] (-5.95,-.95) -- (-4.03,.05) -- (-4.03,-1.25) --
(-5.95,-2.25) -- (-5.95,-.95);

\draw[thick] (-5.95,-.95) -- (-8.47,-.61) -- (-8.47,-1.91) --
(-5.95,-2.25);

\draw[thick] (-8.47,-.61) -- (-6.53,.29) -- (-4.03,.05);

\draw[thick,densely dashed] (-6.53,.29) -- (-6.53,-1.01) --
(-4.03,-1.25);

\draw[thick,densely dashed] (-6.53,-1.01) -- (-8.47,-1.91);

\draw[<->] (-3.7,0.05) -- (-3.7,-1.25) node[midway,right] {$\varepsilon$};
\end{scope}

\draw (-8.2,0) node[left] {$M_\varepsilon=$};

\end{scope}

% 0-graph

\begin{scope}[xshift=11cm]

\shade[ball color = blue, opacity = 0.15] (1,0) circle [radius=2];

\draw[thick] (1,2) [rotate=90] arc [start angle = 0, end angle = 
-180, x radius = 2, y radius = -1];

\draw[thick, densely dashed] (1,2) [rotate=90] arc [start angle 
= 0, end angle = 
-180, x radius = 2, y radius = 1];

\draw[thick] (1,2) [rotate=90] arc [start angle = 0, end angle = 
-180, x radius = 2, y radius = 1.75];

\draw[thick, densely dashed] (1,2) [rotate=90] arc [start angle 
= 0, end angle = 
-180, x radius = 2, y radius = -1.75];

\draw[fill=black] (1,2) circle [radius=.07];
\draw[fill=black] (1,-2) circle [radius=.07];

\draw (-1,0) node[left] {$G=$};

\end{scope}

% 0-cube

\begin{scope}[xshift=18.2cm,yshift=4.5cm]

\begin{scope}[scale=.9,xshift=-.5cm,yshift=.4cm]

\fill[blue!5] (-8.47,-.61) -- (-6.53,.29) -- 
(-4.03,.05) -- (-5.95,-.95) -- (-8.47,-.61);

\draw[thick] (-8.47,-.61) -- (-6.53,.29) -- 
(-4.03,.05) -- (-5.95,-.95) -- (-8.47,-.61);

\end{scope}

\draw (-8.2,0) node[left] {$M=$};

\draw[opacity=0] (-10,0) rectangle (-11.5,0);

\end{scope}

\draw[->] (4.5,0) -- (8,0) node[midway,above] {$\varepsilon\to 0$};
\draw[->] (4.5,4.5) -- (8,4.5) node[midway,above] {$\varepsilon\to 0$};

\end{tikzpicture}
\caption{Metric graph associated to a polytope with empty interior.\label{fig:lower}}
\end{figure}
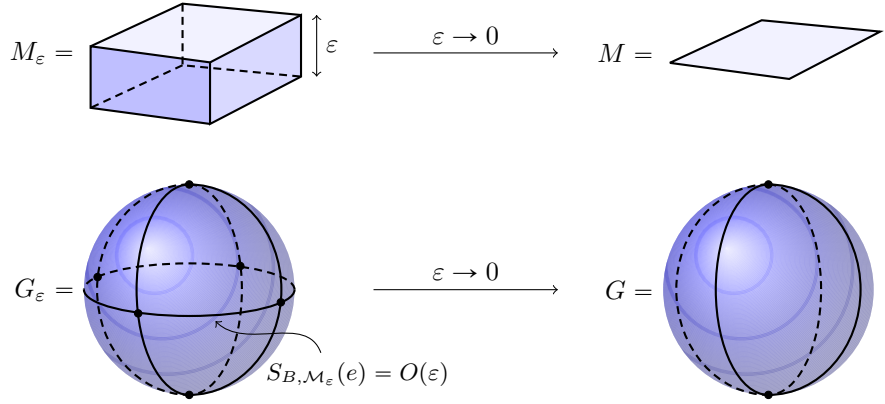

This structure also explains why additional equality cases appear in the 
lower-dimensional setting: as all edges of the graph associated to $M$ 
have length $\pi$, the solution to the Dirichlet problem 
$\mathscr{A}f=\frac{1}{n}\{f''+f\}=0$ on each edge is no longer unique. 
Indeed, if $f$ is such a solution on a given edge, then $\theta\mapsto 
f(\theta)+a\sin(\theta)$ is also a solution with the same boundary data on 
the vertices for any $a\in\mathbb{R}$. This resonance phenomenon results 
in many new elements of the kernel of $\mathscr{A}$ in the setting of 
lower-dimensional polytopes. Because of the simple structure of the graph, 
however, one can compute all elements of $\ker\mathscr{A}$ explicitly, and 
we encounter none of the challenges that arose in the full-dimensional 
setting.

It is not difficult to work out the details of the above argument for 
polytopes. In this case, Theorem \ref{thm:mainlower} was proved in 
\cite[Theorem 4.2]{Sch94} (see also \cite{ET94}) from a somewhat different 
perspective. However, we will show below that essentially the same 
construction remains valid when $M$ is any lower-dimensional convex body. 
In this case, the operator $\mathscr{A}$ turns out to be very similar to 
the quantum graph of a lower-dimensional polytope, except there may now be 
an infinite (even uncountable) number of edges in the graph. Some care 
must be taken, therefore, to construct this operator properly and to 
compute its domain, which will be done in section \ref{sec:lowerop}. Once 
this has been accomplished, however, the proof of Theorem 
\ref{thm:mainlower} will follow readily in section \ref{sec:pflower} from 
an explicit computation of $\ker\mathscr{A}$.

\subsection{Construction of the operator}
\label{sec:lowerop}

Throughout this section, we fix $w\in S^{n-1}$ and a convex body 
$M\subset w^\perp$. We define the measure $S_{\mathcal{M}}$ on 
$S^{n-1}\cap w^\perp$ to be the area measure of $M$ when viewed as a
convex body in $w^\perp$, that is,
$$
	S_{\mathcal{M}}(A) := \mathcal{H}^{n-2}(\{x\in w^\perp:
	x\in F(M,u)\mbox{ for some }u\in A\})
$$
for $A\subseteq S^{n-1}\cap w^\perp$. We will assume 
that $\dim M \ge n-2$, so that $S_{\mathcal{M}}\not\equiv 0$.

In view of the structure illustrated in Figure \ref{fig:lower}, it will
be convenient to parametrize $S^{n-1}$ in polar coordinates $(\theta,z)\in 
[0,\pi]\times (S^{n-1}\cap w^\perp)$ as
$$
	\iota:[0,\pi]\times (S^{n-1}\cap w^\perp)\to S^{n-1},\qquad
	\iota(\theta,z) := w\cos\theta+z\sin\theta.
$$
Note that the parametrization is unique except at $\theta\in\{0,\pi\}$,
where $\iota(0,z)=w$ and $\iota(\pi,z)=-w$ for every $z$. Therefore,
a continuous function $f\in C^0(S^{n-1})$ is given in this parametrization
by a function $f(\theta,z)$ such that $f(0,\cdot\,)$ and $f(\pi,\cdot\,)$ 
are constant functions. Note, however, that the directional derivatives
$\frac{\partial f}{\partial\theta}(\theta,\cdot\,)$ are generally not 
constant functions at $\theta\in\{0,\pi\}$ even when $f\in C^1(S^{n-1})$.

The main result of this section is the following.

\begin{thm}
\label{thm:lowera}
Let $w\in S^{n-1}$ and let $M\subset w^\perp$ be a convex body with
$\dim M\ge n-2$. Then for any $f:S^{n-1}\to\mathbb{R}$, we have
$$
	\int f\,dS_{B,\mathcal{M}} =
	\frac{1}{n-1}
	\int_0^\pi
	\int_{S^{n-1}\cap w^\perp}
	f(\theta,z)\,S_{\mathcal{M}}(dz)\,d\theta.
$$
Moreover, the operator $\mathscr{A}$ defined by
$$
	\mathscr{A}f(\theta,z) = \frac{1}{n}\bigg\{
	\frac{\partial^2f}{\partial\theta^2}(\theta,z)+f(\theta,z)
	\bigg\}
$$
with
\begin{align*}
	\Dom\mathscr{A} =
	\bigg\{ 
	&
	f\in L^2(S_{B,\mathcal{M}}):f(\,\cdot\,,z)\in H^2((0,\pi))
	\mbox{ for }S_{\mathcal{M}}\mbox{-a.e.\ } z, \\
	&
	\frac{\partial^2f}{\partial\theta^2}\in L^2(S_{B,\mathcal{M}}),~~
	f(0,\cdot\,)\mbox{ and }f(\pi,\cdot\,)\mbox{ are }
	S_{\mathcal{M}}\mbox{-a.e.\ constant},
\\
	&\int_{S^{n-1}\cap w^\perp}
	\frac{\partial f}{\partial\theta}(\theta,z)\,S_{\mathcal{M}}(dz)=0
	\mbox{ for }\theta\in\{0,\pi\}\bigg\}
\end{align*}
is self-adjoint on $L^2(S_{B,\mathcal{M}})$ and satisfies all the
properties of Theorem \ref{thm:forms}.
\end{thm}

\begin{rem}
With some additional work, one can show that the operator $\mathscr{A}$ 
of Theorem \ref{thm:lowera} is in fact the one constructed in the 
proof of Theorem \ref{thm:forms}, that is, it is the Friedrichs 
extension of \eqref{eq:defa} in the present setting. This is not needed, 
however, for the applications of this theorem, and the particularly simple 
structure of the present setting enables us to short-circuit some 
technical arguments.
\end{rem}

%%%%%%%%%%%%%%%%%%%%%%%%%%%%%%%%%%%%%%%%%%%%%%%%%%%%%%%%%%%%%%%%%%%%%%%%%%%
% To check that $\mathscr{A},\mathcal{E}$ coincides with Friedrichs 
% extension as in Theorem \ref{thm:qgraph}, the main new issue is that one 
% needs a replacement for smooth extension. The following works. First, 
% note that any $f\in\Dom\mathcal{E}$ means $f(\theta,z) = c + 
% \int_0^\theta (c'+h(\theta',z)-\frac{1}{\pi}\int_0^\pi 
% h(\theta'',z)d\theta'')\,d\theta'$ for some $h\in L^2$. Replace $h$ by a 
% continuous function; then $f$ is continuous and $f'$ is continuous. So 
% such functions constitute a form core. But these can be uniformly 
% approximated by $C^2$ in the usual way. Note that:
%$$
%	\mathcal{E}(f,g) = 
%	\frac{1}{n(n-1)}\int_0^\pi
%        \int_{S^{n-1}\cap w^\perp}
%	\bigg\{f(\theta,z)\,g(\theta,z)-
%	\frac{\partial f}{\partial\theta}(\theta,z)\,
%	\frac{\partial g}{\partial\theta}(\theta,z)
%	\bigg\}\,S_{\mathcal{M}}(dz)\,d\theta
%$$
%where
%\begin{align*}
% 	\Dom\mathcal{E}=\bigg\{ &f\in L^2(S_{B,\mathcal{M}}):
%	f(\,\cdot\,,z)\in H^1([0,\pi])\mbox{ for a.e.\ }z, \\
%	&
%	\frac{\partial f}{\partial\theta}\in 
%	L^2(S_{B,\mathcal{M}}),~~
%	f(0,\cdot\,)\mbox{ and }f(\pi,\cdot\,)\mbox{ are constant a.e.}
%	\bigg\}.
%\end{align*}
%%%%%%%%%%%%%%%%%%%%%%%%%%%%%%%%%%%%%%%%%%%%%%%%%%%%%%%%%%%%%%%%%%%%%%%%%%%

The proof of Theorem \ref{thm:lowera} is similar to that of Theorem 
\ref{thm:qgraph}. We begin by making precise the procedure illustrated in 
Figure \ref{fig:lower}.

\begin{lem}
\label{lem:lowermix}
Let $w\in S^{n-1}$, let $M\subset w^\perp$ be a convex body with
$\dim M\ge n-2$, and let $K$ be a convex body of class $C^2_+$. Then
$$
	\int f\,dS_{K,\mathcal{M}} =
	\frac{1}{n-1}
	\int_0^\pi
	\int_{S^{n-1}\cap w^\perp}
	\bigg\{
	\frac{\partial^2h_K}{\partial\theta^2}(\theta,z)+h_K(\theta,z)
	\bigg\}
	\,f(\theta,z)
	\,S_{\mathcal{M}}(dz)\,d\theta.
$$
\end{lem}

\begin{proof}
Suppose first that $M$ is a polytope of dimension $\dim M=n-1$. Denote by
$\mathcal{F}_M$ the set of its $(n-2)$-dimensional faces, and by
$n_F\in S^{n-1}\cap w^\perp$ the outer normal of $F\in\mathcal{F}_M$ when
viewed as a convex body in $w^\perp$.
Now define for $\varepsilon>0$ the convex body 
$M_\varepsilon:=M+\varepsilon[0,w]$ in $\mathbb{R}^n$. Then 
$M_\varepsilon$ has the following facets:
\begin{enumerate}[1.]
\item $F(M_\varepsilon,w)=M+\varepsilon w$ and
$F(M_\varepsilon,-w)=M$.
\item $F(M_\varepsilon,n_F)=F+\varepsilon[0,w]$ for $F\in\mathcal{F}_M$.
\end{enumerate}
Thus $\mathcal{H}^{n-2}(F(M_\varepsilon,n_F)\cap 
F(M_{\varepsilon},n_{F'}))=O(\varepsilon)$ for any
$F,F'\in\mathcal{F}_M$, so we obtain
\begin{align*}
	&\int f\,dS_{K,\mathcal{M}_\varepsilon} = \\
	&\quad\frac{1}{n-1}
	\sum_{F\in\mathcal{F}_M}\mathcal{H}^{n-2}(F)
	\int_0^\pi
	\bigg\{
	\frac{\partial^2h_K}{\partial\theta^2}(\theta,n_F)+h_K(\theta,n_F)
	\bigg\}
	\,f(\theta,n_F)\,
	d\theta
	+ O(\varepsilon)
\end{align*}
for any continuous function $f$ by Proposition \ref{prop:polymixa}. 
Letting $\varepsilon\to 0$ using Theorem \ref{thm:conv},
and noting that $S_{\mathcal{M}}$ is, by definition, the measure defined
by $S_{\mathcal{M}}(\{n_F\})=\mathcal{H}^{n-2}(F)$ for 
$F\in\mathcal{F}_M$, concludes the proof when $M$ is an
$(n-1)$-dimensional polytope.

Now note that any convex body $M\subset w^\perp$ is the limit in
Hausdorff metric of a sequence of $(n-1)$-dimensional polytopes in
$w^\perp$ \cite[p.\ 39]{BF87}. Thus the conclusion extends to arbitrary
$M$ by approximation using Theorem \ref{thm:conv}.
\end{proof}

The expression for $S_{B,\mathcal{M}}$ in Theorem \ref{thm:lowera} follows 
immediately from Lemma \ref{lem:lowermix}. We now turn our attention to 
proving that $\mathscr{A}$ is self-adjoint on $L^2(S_{B,\mathcal{M}})$.

\begin{lem}
\label{lem:salower}
The operator $\mathscr{A}$ of Theorem \ref{thm:lowera} is self-adjoint.
\end{lem}

\begin{proof}
Note first that if $f,g\in\Dom\mathscr{A}$, then
\begin{align*}
	&\langle f,\mathscr{A}g\rangle_{L^2(S_{B,\mathcal{M}})} 
	=
	\frac{1}{n(n-1)}
	\int_0^\pi
	\int_{S^{n-1}\cap w^\perp}
	f(\theta,z)\,
	\bigg\{
	\frac{\partial^2g}{\partial\theta^2}(\theta,z)+g(\theta,z)
	\bigg\}
	\,S_{\mathcal{M}}(dz)\,d\theta \\
	&\quad=
	\frac{1}{n(n-1)}
	\int_0^\pi
	\int_{S^{n-1}\cap w^\perp}
	\bigg\{f(\theta,z)\,g(\theta,z) -
	\frac{\partial f}{\partial\theta}(\theta,z)\,
	\frac{\partial g}{\partial\theta}(\theta,z)\bigg\}
	\,S_{\mathcal{M}}(dz)\,d\theta
	\\ &\quad\qquad
	+
	\frac{1}{n(n-1)}
	\int_{S^{n-1}\cap w^\perp}
	f(\theta,z)\,\frac{\partial g}{\partial\theta}(\theta,z) 
	\,S_{\mathcal{M}}(dz)
	\bigg|_{\theta=0}^{\theta=\pi},
\end{align*}
where we integrated by parts. But the definition of $\Dom\mathscr{A}$ 
ensures that the boundary term vanishes. It follows that
$\mathscr{A}$ is a symmetric operator, and in particular
$\Dom\mathscr{A}\subseteq\Dom\mathscr{A}^*$. It therefore remains to
prove the converse inclusion.

Fix in the rest of the proof $f\in\Dom\mathscr{A}^*$. We must show that 
$f$ satisfies each of the defining properties of $\Dom\mathscr{A}$.

First, note that for any smooth compactly supported function
$\varphi\in C_0^\infty((0,\pi))$ and any $h\in L^2(S_{\mathcal{M}})$, the
function $g(\theta,z):=\varphi(\theta)h(z)$ satisfies $g\in\Dom\mathscr{A}$.
Thus
\begin{align*}
	&\frac{1}{n(n-1)}
	\int \bigg[\int_0^\pi f(\theta,z)\,
	\{\varphi''(\theta)+\varphi(\theta)\}\,d\theta\bigg]\,
	h(z)\,S_{\mathcal{M}}(dz)
	=
	\langle f,\mathscr{A}g\rangle_{L^2(S_{B,\mathcal{M}})} \\
	&\qquad
	=
	\langle \mathscr{A}^*f,g\rangle_{L^2(S_{B,\mathcal{M}})} 
	=
	\frac{1}{n-1}
	\int \bigg[\int_0^\pi \mathscr{A}^*f(\theta,z)\,
	\varphi(\theta)\,d\theta\bigg]\,
	h(z)\,S_{\mathcal{M}}(dz).
\end{align*}
As $h$ is arbitrary, we have
$$
	\int_0^\pi f(\theta,z)\,
        \varphi''(\theta)\,d\theta
	=
	\int_0^\pi (n\mathscr{A}^*-I)f(\theta,z)\,
        \varphi(\theta)\,d\theta\quad 
	\mbox{for }S_{\mathcal{M}}\mbox{-a.e.\ }z
$$
for any $\varphi\in C_0^\infty((0,\pi))$. As $H^2((0,\pi))$ is separable
\cite[section 7.5]{GT01}, this identity remains valid 
simultaneously for all $\varphi\in 
C_0^\infty((0,\pi))$ (that is, the exceptional set may be chosen
independent of $\varphi$). As $(n\mathscr{A}^*-I)f(\,\cdot\,,z)\in 
L^2((0,\pi))$ for $S_{\mathcal{M}}$-a.e.\ $z$ by Fubini's theorem, we 
have shown that $f(\,\cdot\,,z)\in H^2((0,\pi))$ for
$S_{\mathcal{M}}$-a.e.\ $z$ and that
$\mathscr{A}^*f = \frac{1}{n}\{\frac{\partial^2f}{\partial\theta^2}+f\}$
$S_{B,\mathcal{M}}$-a.e.\ (in particular, 
$\frac{\partial^2f}{\partial\theta^2}\in L^2(S_{B,\mathcal{M}})$).

It remains only to establish the vertex boundary conditions at
$\theta\in\{0,\pi\}$. To this end, note first that if 
$g\in\Dom\mathscr{A}$ is arbitrary, then
$$
	\langle f,\mathscr{A}g\rangle_{L^2(S_{B,\mathcal{M}})}
	=
	\langle \mathscr{A}^*f,g\rangle_{L^2(S_{B,\mathcal{M}})}
	=
	\langle 
	\textstyle{\frac{1}{n}\{\frac{\partial^2f}{\partial\theta^2}+f\}},
	g\rangle_{L^2(S_{B,\mathcal{M}})}.
$$
Integrating by parts as in the beginning of the 
proof shows that 
$$
	\int_{S^{n-1}\cap w^\perp}
	\bigg\{
	f(\theta,z)\,\frac{\partial g}{\partial\theta}(\theta,z) -
	g(\theta,z)\,\frac{\partial f}{\partial\theta}(\theta,z) 
	\bigg\}
	\,S_{\mathcal{M}}(dz)
	\bigg|_{\theta=0}^{\theta=\pi}
	= 0
$$
for every $g\in\Dom\mathscr{A}$. We choose a different test function
$g$ to deduce each boundary condition. First, let
$g(\theta,z):=1\pm\cos(\theta)$. Then $g\in\Dom\mathscr{A}$, so
we conclude
$$
	\int_{S^{n-1}\cap w^\perp}
	\frac{\partial f}{\partial\theta}(\theta,z) \,S_{\mathcal{M}}(dz)
	= 0\quad\mbox{for }\theta\in\{0,\pi\}.
$$
Next, let $g(\theta,z):=\sin(k\theta)h(z)$ for $k=1,2$ and
$h\in L^2(S_{\mathcal{M}})$ with $\int h\,dS_{\mathcal{M}}=0$. Then $g\in 
\Dom\mathscr{A}$, so we conclude that
$$
	\int_{S^{n-1}\cap w^\perp}
	f(0,z)\,h(z)
	\,S_{\mathcal{M}}(dz) =
	\int_{S^{n-1}\cap w^\perp}
	f(\pi,z)\,h(z)
	\,S_{\mathcal{M}}(dz)
	= 0.
$$
As this holds for all $h$ of the above form, it must be the case
that $f(0,\cdot\,)$ and $f(\pi,\cdot\,)$ are $S_{\mathcal{M}}$-a.e.\ 
constant. The proof is complete.
\end{proof}

The reason that the setting of this section is particularly simple
is that we can compute the full spectral decomposition of $\mathscr{A}$.

\begin{lem}
\label{lem:lowerspec}
The operator $\mathscr{A}$ of Theorem \ref{thm:lowera} satisfies
$$
	\spec\mathscr{A} = \{\lambda_k:k\in\mathbb{Z}_+\},\qquad
	\lambda_k := \frac{1-k^2}{n}.
$$
Moreover, the eigenspace $E_k$ associated to eigenvalue $\lambda_k$
is given by
$$
	E_k := \bigg\{f:
	f(\theta,z)=h(z)\sin(k\theta)+a\cos(k\theta),~
	a\in\mathbb{R},~
	h\in L^2(S_\mathcal{M}),~
	\int h\,dS_{\mathcal{M}}=0
	\bigg\}.
$$
\end{lem}

\begin{proof}
Let $E_k$ be the spaces defined in the statement of the lemma.
It is readily verified that $E_k\subset\Dom\mathscr{A}$ and 
$\mathscr{A}f=\lambda_kf$ for each $k$ and $f\in E_k$.

We now claim that 
$$
	E:=\overline{\spn}\Bigg(\bigcup_{k\ge 0}E_k\Bigg)
	=L^2(S_{B,\mathcal{M}}).
$$
As $S_{B,\mathcal{M}}$ is a product measure on $[0,\pi]\times (S^{n-1}\cap 
w^\perp)$, it suffices to show that any function of the form 
$g(\theta,z):=\varphi(\theta)\,h(z)$ with $\varphi\in L^2([0,\pi])$ and 
$h\in L^2(S_{\mathcal{M}})$ lies in $E$. To this end, note that both
$\{\sin(k\theta):k\ge 1\}$ and $\{\cos(k\theta):k\ge 0\}$ are complete 
orthogonal bases of $L^2([0,\pi])$ (these are the Dirichlet and Neumann 
eigenfunctions of the Laplacian on $[0,\pi]$, respectively). Thus
we may write
$$
	\varphi(\theta) = 
	\sum_{k\ge 0}a_k\cos(k\theta) =
	\sum_{k\ge 1}b_k\sin(k\theta)
$$
in $L^2([0,\pi])$ for some coefficient sequences $a_k,b_k$. Moreover, we 
can evidently write $h(z)=h_0(z)+c$ where $h_0\in L^2(S_{\mathcal{M}})$ 
and $\int h_0\,dS_{\mathcal{M}}=0$.
Thus
$$
	g(\theta,z) =
	\sum_{k\ge 0}\{b_kh_0(z)\sin(k\theta)
	+ ca_k\cos(k\theta)\} \in E,
$$
completing the proof of the claim.

Denote by $P_k$ the orthogonal projection in $L^2(S_{B,\mathcal{M}})$ onto 
$E_k$. As the eigenspaces of a self-adjoint operator are orthogonal, it 
follows that
$$
	I = \sum_{k\ge 0} P_k \le 
	\sum_{k\ge 0} 1_{\{\lambda_k\}}(\mathscr{A}) =
	1_{\{\lambda_k:k\in\mathbb{Z}_+\}}(\mathscr{A}) \le
	I.
$$
Thus $\spec\mathscr{A}=\{\lambda_k:k\in\mathbb{Z}_+\}$ and
$P_k=1_{\{\lambda_k\}}(\mathscr{A})$ for all $k$.
\end{proof}

Finally, we make the following simple observation.

\begin{lem}
\label{lem:lowerc2}
In the setting of Theorem \ref{thm:lowera}, we have
$C^2(S^{n-1})\subset\Dom\mathscr{A}$.
\end{lem}

\begin{proof}
It suffices to note that for $f\in C^2(S^{n-1})$, we have
$$
	\int
        \frac{\partial f}{\partial\theta}(0,z)\,S_{\mathcal{M}}(dz)
	=
	\int \langle z,\nabla f(w)\rangle\,S_{\mathcal{M}}(dz)
	=0 
$$
and
$$
	\int
        \frac{\partial f}{\partial\theta}(\pi,z)\,S_{\mathcal{M}}(dz)
	=
	\int \langle z,\nabla f(-w)\rangle\,S_{\mathcal{M}}(dz)
	=0 
$$
by Lemma \ref{lem:maprop}(\textit{e}). The remaining properties 
are trivial.
\end{proof}

We can now complete the proof of Theorem \ref{thm:lowera}.

\begin{proof}[Proof of Theorem \ref{thm:lowera}]
The expression for $S_{B,\mathcal{M}}$ follows from Lemma 
\ref{lem:lowermix}, and self-adjointness of $\mathscr{A}$ on
$L^2(S_{B,\mathcal{M}})$ was proved in Lemma \ref{lem:salower}. Properties 
\textit{a}--\textit{c} of Theorem \ref{thm:forms} can be read off from 
Lemma \ref{lem:lowerspec}.

Finally, note that by Lemmas \ref{lem:lowermix} and \ref{lem:lowerc2}, the 
operator $\mathscr{A}$ agrees with \eqref{eq:defa} on $C^2(S^{n-1})$. Thus 
the closed quadratic form associated to $\mathscr{A}$ is a closed 
extension of its restriction to $C^2(S^{n-1})$. The quadratic form 
$\mathcal{E}$ of Theorem \ref{thm:forms} is the smallest such extension 
(the Friedrichs extension); thus properties \textit{d}--\textit{e} of 
Theorem \ref{thm:forms} remain valid in the present setting (and for any 
other closed extension of $\mathcal{E}$).
\end{proof}

\begin{rem}
\label{rem:lowcompres}
Unless $M$ is a polytope, Lemma \ref{lem:lowerspec} implies
that the eigenspaces of $\mathscr{A}$ are infinite-dimensional. Thus
$\mathscr{A}$ does not have a compact resolvent.
This provides an explicit example of the issue that was highlighted in
Remark \ref{rem:compres}. Let us also note that this phenomenon is not
specific to the lower-dimensional setting; for example, it may be verified
that a similar situation occurs if we replace $M$ by $M_\varepsilon:=
M+\varepsilon[0,w]$, which has nonempty interior. We omit the details.
\end{rem}

\begin{rem}
\label{rem:fdj}
Even when $M$ has nonempty interior, we have
only established that $\ker\mathscr{A}\cap\{h_Q-h_R:Q,R\mbox{ convex 
bodies}\}$ consists of linear functions. A function that is not a
difference of support functions does not give rise to extremals of 
Minkowski's inequality, so is not relevant for the purposes of this paper. 
Nonetheless, one may wonder whether it is possible that $\ker\mathscr{A}$ 
contains such functions.

When $M$ has empty interior, Lemma \ref{lem:lowerspec} shows that this may 
in fact happen. Indeed, as the function $h\in L^2(S_{\mathcal{M}})$ need 
only be measurable, it is perfectly possible in general to construct 
eigenfunctions that do not admit a continuous extension to $S^{n-1}$. Such 
eigenfunctions cannot arise as the difference of support functions, as 
support functions are always continuous. Thus $\ker\mathscr{A}$ may 
contain many elements that do not contribute to the characterization of 
extremals.

It is natural to conjecture that when $M$ has nonempty interior, such 
examples cannot occur and that $\ker\mathscr{A}$ consists \emph{only} of 
linear functions (cf.\ Remark \ref{rem:gap}). Such a result cannot be
achieved, however, by the methods of this paper. The key obstruction is 
the theory of section \ref{sec:rigid}: that $\supp\mu_M\subseteq\supp 
S_{M,\mathcal{M}}$ only guarantees that \emph{continuous} functions that
vanish $S_{M,\mathcal{M}}$-a.e.\ must vanish $\mu_M$-a.e., so we 
cannot rule out discontinuous elements of $\ker\mathscr{A}$ by this 
method.
\end{rem}

\subsection{Proof of Theorem \ref{thm:mainlower}}
\label{sec:pflower}

By Lemma \ref{lem:hyper}, understanding the equality cases of Minkowski's 
quadratic inequality reduces to understanding the kernel of $\mathscr{A}$. 
In the present setting, however, we have already computed the kernel in 
Lemma \ref{lem:lowerspec}. It therefore remains to furnish its elements 
with a geometric interpretation.

Before we turn to the proof of Theorem \ref{thm:mainlower}, however, we 
must extend the conclusion of Theorem \ref{thm:supp} to the present 
setting.

\begin{lem}
\label{lem:lowersupp}
Let $w\in S^{n-1}$ and let $M\subset w^\perp$ be a convex body with
$\dim M\ge n-2$. Then
$\supp S_{B,\mathcal{M}} =
\cl\{u\in S^{n-1}:u\mbox{ is a }1\mbox{-extreme normal vector
of } M\}$.
\end{lem}

\begin{proof}
As $w$ is normal to every point in $M$, a vector $u\in S^{n-1}$ is normal 
to a given point in $M$ if and only if its projection
$P_{w^\perp}u$ is normal to that point. In particular, it follows readily 
that $u\ne w$ is a $1$-extreme normal vector of $M$ if and only if 
$P_{w^\perp}u$ is a $0$-extreme normal vector of $M$ when viewed as
a convex body in $w^\perp$.

Now note that, by the expression for $S_{B,\mathcal{M}}$ given in
Theorem \ref{thm:lowera}, we have
$\supp S_{B,\mathcal{M}} = [0,\pi]\times\supp S_{\mathcal{M}}$ 
(in polar coordinates). When $\dim M=n-1$, we have
$\supp S_{\mathcal{M}}=\cl\{0\mbox{-extreme normal vectors of }M\mbox{ in }
w^\perp\}$ by Theorem \ref{thm:supp}, and the conclusion follows.
On the other hand, when $\dim M=n-2$, we have $M\subset\spn\{v,w\}^\perp$ 
for some
$v\perp w$. Thus $\pm v$ are the only $0$-extreme normal directions of $M$ 
in $w^\perp$. On the other hand, it follows readily from the definition 
that $\supp S_{\mathcal{M}}=\{\pm v\}$ in this case, so that the 
conclusion again follows.
\end{proof}

We are now ready to complete the proof of Theorem \ref{thm:mainlower}.

\begin{proof}[Proof of Theorem \ref{thm:mainlower}]
By translation-invariance of mixed volumes, we may assume without loss of 
generality that $M\subset w^\perp$. Moreover, we may assume $\dim M\ge 
n-2$, as otherwise $\V(L,L,\mathcal{M})=0$ for all bodies $L$ 
\cite[Theorem 5.1.8]{Sch14}.

Now let $K,L$ be any convex bodies in $\mathbb{R}^n$ with 
$\V(L,L,\mathcal{M})>0$.
By Theorem \ref{thm:lowera} and Lemma \ref{lem:hyper}, we have equality
in Minkowski's inequality
$$
	\V(K,L,\mathcal{M})^2=
	\V(K,K,\mathcal{M})\,\V(L,L,\mathcal{M})
$$
if and only if
$h_K-ah_L\in\ker\mathscr{A}$ for some $a\in\mathbb{R}$. Thus the proof
will be concluded once we establish that the following two statements are 
equivalent:
\begin{enumerate}[1.]
\itemsep\abovedisplayskip
\item $h_K-ah_L\in\ker\mathscr{A}$ for some $a\in\mathbb{R}$.
\item $\tilde L:=\frac{\V(K,L,\mathcal{M})}{\V(L,L,\mathcal{M})}L$ has the
property that $K+F(\tilde L,w)$ and $\tilde L+F(K,w)$ have the same 
supporting hyperplanes in all $1$-extreme normal directions of $M$.
\end{enumerate}
Let us prove each in turn.

\vskip.2cm

$\boldsymbol{1\Rightarrow 2}$. Let $h_K-ah_L\in\ker\mathscr{A}$.
First, note that
$$
	0 = \langle h_L,\mathscr{A}(h_K-ah_L)\rangle_{L^2(S_{B,\mathcal{M}})}
	= \V(K,L,\mathcal{M})-a\,\V(L,L,\mathcal{M})
$$
by Theorem \ref{thm:lowera}. Thus we must have
$$
	a = \frac{\V(K,L,\mathcal{M})}{\V(L,L,\mathcal{M})},
$$
and it follows that $\tilde L=aL$.
To proceed, we observe that Lemma \ref{lem:lowerspec} implies that
$$
	h_K(\theta,z)-h_{\tilde L}(\theta,z) = 
	\eta(z)\sin\theta+ \alpha\cos\theta\quad 
	S_{B,\mathcal{M}}\mbox{-a.e.\ }(\theta,z)
$$
for some $\alpha\in\mathbb{R}$ and $\eta\in L^2(S_{\mathcal{M}})$ with
$\int \eta\,dS_{\mathcal{M}}=0$. Evidently
$$
	\alpha = h_K(0,z)-h_{\tilde L}(0,z) = h_K(w)-h_{\tilde L}(w),
$$
and
$$
	\eta(z) = \frac{\partial h_K}{\partial\theta}(0,z)
        -\frac{\partial h_{\tilde L}}{\partial\theta}(0,z) =
	\nabla_zh_K(w)-\nabla_zh_{\tilde L}(w) =
	h_{F(K,w)}(z)-h_{F(\tilde L,w)}(z)
$$
by Lemma \ref{lem:dirdir}. But note that at 
the point $x=w\cos\theta+z\sin\theta\in S^{n-1}$ corresponding
to the polar coordinates $(\theta,z)$, we can write using
$F(K,w)-wh_K(w)\subset w^\perp$
$$
	h_{F(K,w)}(x) = h_{F(K,w)}(z)\sin\theta + h_K(w)\cos\theta.
$$
The analogous formula holds for $\tilde L$, and we conclude that
$$
	h_K(x)-h_{\tilde L}(x) = 
	h_{F(K,w)}(x)-h_{F(\tilde L,w)}(x)\quad
	S_{B,\mathcal{M}}\mbox{-a.e.\ }x.
$$
By continuity of support functions, this identity remains valid for all 
$x\in\supp S_{B,\mathcal{M}}$, and the implication $1\Rightarrow 2$ 
follows by Lemma \ref{lem:lowersupp}.

\vskip.2cm

$\boldsymbol{2\Rightarrow 1}$. 
By Lemma \ref{lem:lowersupp} and continuity, we can assume
$$
	h_K(x)-h_{\tilde L}(x) = 
	h_{F(K,w)}(x)-h_{F(\tilde L,w)}(x)\quad
	\mbox{for }
	x\in\supp S_{B,\mathcal{M}}.
$$
We will prove directly that this implies equality in Minkowski's quadratic
inequality (and hence $h_K-ah_L\in\ker\mathscr{A}$ by Lemma \ref{lem:hyper}).

We begin by noting that $\supp S_{C,\mathcal{M}}=\supp S_{B,\mathcal{M}}$
for every body $C$ of class $C^2_+$ by Lemma \ref{lem:lowermix}.
Thus Theorem \ref{thm:conv} implies that
$\supp S_{C,\mathcal{M}}\subseteq\supp S_{B,\mathcal{M}}$ for any convex
body $C$. Choosing $C=F(K,w)$ and $C=F(\tilde L,w)$, respectively, we find
\begin{align*}
	&\V(h_K-h_{\tilde L},h_{F(K,w)}-h_{F(\tilde L,w)},\mathcal{M}) \\ &\quad=
	\frac{1}{n}\int (h_K-h_{\tilde L})\,dS_{h_{F(K,w)}-h_{F(\tilde L,w)},\mathcal{M}}
	\\
	&\quad=
	\frac{1}{n}\int (h_{F(K,w)}-h_{F(\tilde L,w)})\,dS_{h_{F(K,w)}-h_{F(\tilde L,w)},\mathcal{M}}
	\\
	&\quad= 
	\V(h_{F(K,w)}-h_{F(\tilde L,w)},h_{F(K,w)}-h_{F(\tilde L,w)}, \mathcal{M})=0,
\end{align*}
where the last equality holds as $\dim(F(K,w)+F(\tilde L,w)+M)<n$. On the
other hand, choosing $C=K$ and $C=\tilde L$, we obtain similarly
$$
	\V(h_K-h_{\tilde L},h_K-h_{\tilde L},\mathcal{M}) =
	\V(h_{F(K,w)}-h_{F(\tilde L,w)},h_K-h_{\tilde L},\mathcal{M})=0.
$$
Thus we have shown that
$$
	\V(K,K,\mathcal{M}) 
	-2\,\V(K,\tilde L,\mathcal{M})+\V(\tilde L,\tilde L,\mathcal{M})=0.
$$
The conclusion follows by substituting
$\tilde L = \frac{\V(K,L,\mathcal{M})}{\V(L,L,\mathcal{M})}L$ into this
expression.
\end{proof}

\subsection*{Acknowledgments}

We thank Emanuel Milman and Fedja Nazarov for helpful comments, and we are 
grateful to the anonymous referees whose detailed suggestions helped
us improve the presentation of the paper.

This work was supported in part by NSF grants CAREER-DMS-1148711 and 
DMS-1811735, ARO through PECASE award W911NF-14-1-0094, and the Simons 
Collaboration on Algorithms \& Geometry. The project was initiated while 
the authors were in residence at MSRI in Berkeley, CA in Fall 2017, 
supported by NSF grant DMS-1440140. The hospitality of MSRI and of the 
organizers of the program on Geometric Functional Analysis is gratefully 
acknowledged.

%\bibliographystyle{abbrv}
%\bibliography{ref}

\end{document}